\documentclass[12pt,a4paper]{amsart}
\usepackage{amssymb,paralist,xspace,url,amscd,euscript,mathrsfs,stmaryrd,
epic,eepic,longtable}
\usepackage[centering,text={14.5cm,23cm}]{geometry}
\usepackage[all]{xy}
\SelectTips{cm}{}
\usepackage[colorlinks=true]{hyperref}

\usepackage[dvips,final]{graphicx}
\usepackage{color} %include even if images aren't in color
\usepackage{epsfig}
\usepackage{latexsym}
\usepackage{pstricks}
\usepackage{comment}
\numberwithin{equation}{section}

\allowdisplaybreaks[1]

\newenvironment{enumeratei}
{\begin{enumerate}[\upshape (i)]}
{\end{enumerate}}

\newenvironment{enumerate1}
{\begin{enumerate}[\upshape (1)]}
{\end{enumerate}}

\newtheorem*{namedtheorem}{\theoremname}
\newcommand{\theoremname}{testing}

\theoremstyle{plain}

\newtheorem{theorem}{Theorem}[section]
\newtheorem{proposition}[theorem]{Proposition}
\newtheorem{proposition-definition}[theorem]{Proposition-Definition}
\newtheorem{lemma-definition}[theorem]{Lemma-Definition}
\newtheorem{corollary}[theorem]{Corollary}
\newtheorem{lemma}[theorem]{Lemma}

\theoremstyle{definition}
\newtheorem{definition}[theorem]{Definition}
\newtheorem{notation}[theorem]{Notation}
\newtheorem{example}[theorem]{Example}
\newtheorem{examples}[theorem]{Examples}
\newtheorem{remark}[theorem]{Remark}

\newtheorem{construction}[theorem]{Construction}

\theoremstyle{remark}

\newcommand\ul[1]{\underline{#1}}
\newcommand\ocM{\overline{\mathcal{M}}}

\newcommand\cl{\mathrm{cl}}

\newcommand\scrM{\mathscr{M}}

\newcommand\Int{\operatorname{Int}}

\newcommand{\ev}{\mathrm{ev}}

\newcommand\ulf{\underline{f}}
\newcommand\ulC{\underline{C}}

\newcommand\cA{\mathcal{A}}

\newcommand\cC{\mathcal{C}}

\newcommand\cF{\mathcal{F}}
\newcommand\cG{\mathcal{G}}

\newcommand\cK{\mathcal{K}}
\newcommand\cL{\mathcal{L}}
\newcommand\cM{\mathcal{M}}

\newcommand\cO{\mathcal{O}}
\newcommand\cP{\mathcal{P}}

\newcommand\cX{\mathcal{X}}

\newcommand\uB{\underline{B}}
\newcommand\uC{\underline{C}}

\newcommand\uf{\underline{f}}

\newcommand\uX{\underline{X}}

\renewcommand\AA{\mathbb{A}}

\newcommand\CC{\mathbb{C}}

\newcommand\FF{\mathbb{F}}
\newcommand\GG{\mathbb{G}}

\newcommand\kk{\Bbbk}

\newcommand\NN{\mathbb{N}}

\newcommand\PP{\mathbb{P}}
\newcommand\QQ{\mathbb{Q}}
\newcommand\RR{\mathbb{R}}

\newcommand\ZZ{\mathbb{Z}}

\newcommand\bA{\mathbf{A}}

\newcommand\bC{\mathbf{C}}

\newcommand\bE{\mathbf{E}}

\newcommand\bL{\mathbf{L}}
\newcommand\bM{\mathbf{M}}

\newcommand\bg{\mathbf{g}}
\newcommand\bp{\mathbf{p}}
\newcommand\bu{\mathbf{u}}

\newcommand\fC{\mathfrak{C}}

\newcommand\fE{\mathfrak{E}}

\newcommand\fM{\mathfrak{M}}

\newcommand\ttau{{\widetilde\tau}}
\newcommand\bsigma{\boldsymbol{\sigma}}

\newcommand\Cones{\mathbf{Cones}}
\newcommand\arr{\ifinner\to\else\longrightarrow\fi}
\newcommand\larr{\longrightarrow}

\newcommand\into{\hookrightarrow}

\newcommand\im{\operatorname{im}}

\def\displaytimes_#1{\mathrel{\mathop{\times}\limits_{#1}}}
\def\displayotimes_#1{\mathrel{\mathop{\bigotimes}\limits_{#1}}}

\newcommand\Aut{\operatorname{Aut}}

\newcommand\tor{{\operatorname{tor}}}
\newcommand\sing{\mathrm{sing}}
\newcommand\reg{\mathrm{reg}}
\newcommand\pic{\operatorname{Pic}}

\newcommand\Spec{\operatorname{Spec}}

\newcommand\virt{{\operatorname{virt}}}

\newcommand\pr{\operatorname{pr}}

\newdir{ >}{{}*!/-5pt/@{>}}

\newcommand\doublelong[2]{\mathbin{\xymatrix{{}\ar@<3pt>[r]^{#1}
\ar@<-3pt>[r]_{#2}&}}}

\newlength{\ignora}

\newcommand{\ind}{\operatorname{Ind}}

\renewcommand{\red}{{\mathrm{red}}}

\newcommand{\ol}{\overline}
\newcommand{\Log}{{\operatorname{Log}}}

\newcommand{\gp}{{\mathrm{gp}}}

\renewcommand{\setminus}{\smallsetminus}

\newcommand{\Hom}{{\operatorname{Hom}}}

\def\mapright#1{\smash{
\mathop{\longrightarrow}\limits^{#1}}}

\begin{document}

\title{Decomposition of degenerate Gromov-Witten invariants}


\author{Dan Abramovich}
\address{\tiny Department of Mathematics, Brown University, Box 1917,
Providence, RI~02912, USA}
\email{abrmovic@math.brown.edu}
\thanks{Research by D.A.\ was supported in part by NSF grants DMS-1162367,
DMS-1500525 and DMS-1759514}

\author{Qile Chen}
\address{\tiny Department of Mathematics, Boston College, Chestnut Hill,
MA~02467-3806, USA}
\email{qile.chen@bc.edu}
\thanks{Research by Q.C. was supported in part by NSF grant DMS-1403271 and DMS-
1560830.}

\author{Mark Gross}
\address{\tiny DPMMS, Centre for Mathematical Sciences, Wilberforce Road,
Cambridge, CB3 0WB, UK}
\email{mgross@dpmms.cam.ac.uk}
\thanks{M.G.\ was supported by NSF grant DMS-1262531, EPSRC grant EP/N03189X/1
and a Royal Society Wolfson Research Merit Award.}

\author{Bernd Siebert}
\address{\tiny Department of Mathematics, The Univ.\ of Texas at Austin,
2515 Speedway, Austin, TX 78712, USA}
\email{siebert@math.utexas.edu}
\thanks{Research by B.S.\ was partially supported by NSF grant DMS-1903437}

\date{\today}
\begin{abstract}
We prove a decomposition formula of logarithmic Gromov-Witten invariants in a
degeneration setting. A one-parameter log smooth family $X\arr B$ with singular
fibre over $b_0\in B$ yields a family $\scrM(X/B,\beta)\arr B$ of moduli stacks
of stable logarithmic maps. We give a virtual decomposition of the fibre of this
family over $b_0$ in terms of rigid tropical maps to the tropicalization of
$X/B$. This generalizes one aspect of known results in the case that the fibre
$X_{b_0}$ is a normal crossings union of two divisors. We exhibit our formulas
in explicit examples.
\end{abstract}

\maketitle

\tableofcontents

%=================================================================================
\section{Introduction}

%---------------------------------------------------------------------------------
\subsection{Statement of results}
One of the main goals of logarithmic Gromov--Witten theory is to relate the
Gromov--Witten invariants of a smooth projective variety to invariants of a
degenerate variety $X_0$.

Consider a logarithmically smooth and projective morphism $X \arr B$, with $B$ a
logarithmically smooth curve having a single closed point $b_0\in B$ where the
logarithmic structure is nontrivial. In the language of \cite{KKMS,AK}, this is
the same as saying that the underlying schemes $\uX$ and $\uB$ are provided with
a toroidal structure such that $\uX\arr \uB$ is a toroidal morphism, and
$\underline{\{b_0\}}\subset \uB$ is the toroidal divisor. One defines as in
\cite{GS}, see also \cite{Chen, AC}, an algebraic stack $\scrM(X/B,\beta)$
parameterizing {\em stable logarithmic maps $f: C \arr X$ with discrete data
$\beta = (g,A,u_{p_1},\ldots,u_{p_k})$} from logarithmically smooth curves to
$X$. Here
\begin{itemize}
\item
$g$ is the genus of $C$,
\item
$A$ is the homology class $\uf_*[\uC]$, which we assume is supported on
fibres of $\uX \arr \uB$, and
\item
$u_{p_1},\ldots,u_{p_k}$ are the {\em contact orders} of the marked points with
the logarithmic strata of $X$.
\end{itemize}

Writing $\ul\beta = (g,k,A)$ for the non-logarithmic discrete data, there is a
natural morphism $\scrM(X/B,\beta) \arr \scrM(\uX/\uB,{\ul\beta})$ ``forgetting
the logarithmic structures", which is proper and representable \cite[Thm.~1.1.1]{ACMW14}.
The map $\scrM(X/B,\beta) \arr \scrM(\uX/\uB,{\ul\beta})$ is in
fact finite, see \cite[Cor.~1.2]{Wise-finite}. There is also a natural
morphism $\scrM(X/B,\beta) \arr B$, and we denote its fibre over $b\in B$ by
$\scrM(X_b/b,\beta)$.

Since $X\arr B$ is logarithmically smooth there is a perfect relative
obstruction theory $\bE^\bullet \arr \bL_{\scrM(X/B,\beta)\,/\,\Log_B}$
in the sense of \cite{Behrend-Fantechi}, hence defining a virtual
fundamental class $[\scrM(X/B,\beta)]^\virt$ and logarithmic
Gromov--Witten invariants.

An immediate consequence of the formalism is the following {(this is indicated
after \cite[Thm.~0.3]{GS})}:
\begin{theorem}[Logarithmic deformation invariance]
\label{Th:deformation}
For any point $\{b\}\stackrel{j_b}\into B$ one has
$$j_b^![\scrM(X/B,\beta)]^\virt = [\scrM(X_b/b,\beta)]^\virt.$$
\end{theorem}

This implies, in particular, that Gromov--Witten invariants of $X_b$ agree with
those of $X_0=X_{b_0}$. Now holomorphic curves in $X_0$ come in various
families depending on the intersection pattern with the irreducible components
of $X_0$. Thus one may hope that logarithmic Gromov-Witten invariants similarly
group according to some discrete data reflecting such intersection patterns. The
main result of this paper shows that this is indeed the case, with the
intersection patterns recorded in an interesting and very transparent fashion in
terms of the underlying tropical geometry.

\begin{theorem}[The logarithmic decomposition formula; Theorem~\ref{Thm: Main} below]
\label{Thm: Decomposition}
Suppose the morphism $X_0 \arr b_0$ is logarithmically smooth and $X_0$ is
simple. Then we have the following equality in the Chow group of
$\scrM(X_0/b_0,\beta)$ with coefficients in $\QQ$:
\[
[\scrM(X_0/b_0,\beta)]^\virt = \sum_{\ttau=(\tau,\bA)}
\frac{m_\tau}{|\Aut(\tau)|}\, {j_\ttau}_*[\scrM(X_0,\ttau)]^\virt.
\]
\end{theorem}

See Definition~\ref{Def:simple} for the notion of simple logarithmic structures.
The notations $\scrM(X_0,\ttau)$, $m_\tau$ and $j_\ttau$ are briefly explained
as follows. First, the \emph{tropicalization of $X_0\arr b_0$} defines a
polyhedral complex $\Delta(X_0)$ (\S\ref{ss:logschtrop} and
\S\ref{Sec:basic-gives-universal-tropical}), and $\tau$ stands for a \emph{rigid
tropical map to $\Delta(X)$} (Definition~\ref{Def: Rigid tropical map}). Each
such rigid $\tau$ comes with a multiplicity $m_{\tau} \in \NN$, the smallest
integer such that scaling $\Delta(X)$ by $m_\tau$ leads to a tropical curve with
integral vertices and edge lengths.

The symbol $\bA$ stands for a partition of the curve class $A\in H_2(X)$ into
classes $\bA(v)$, one for each vertex $v$ in the graph underlying $\tau$.

The moduli stack $\scrM(X_0,\ttau)$ is the stack parameterizing basic stable
logarithmic maps to $X_0$ over $b_0$ decorated by $\ttau=(\tau,\bA)$
(Definition~\ref{Def: ttau-marked stable map}). The marking exhibits $\ttau$ as
a degeneration of the tropicalization of any stable logarithmic map in this
moduli stack. The map $j_\ttau:\scrM(X_0,\ttau)\arr \scrM(X_0/b_0,\beta)$
forgets the marking by $\ttau$.

\begin{remark}
In general, the sum over $\ttau$ will be infinite, but because
the moduli space $\scrM(X_0/b_0,\beta)$ is of finite type, all but
a finite number of the moduli spaces
$\scrM(X_0,\ttau)$ will be empty. In practice one uses the
balancing condition \cite[Prop.~1.15]{GS} to control how curves can break
up into strata of $X_0$. This is carried out in some of the examples in
\S\ref{examplessection}.
\end{remark}

Theorems~\ref{Th:deformation} and \ref{Thm: Decomposition} form the first two
steps toward a general logarithmic degeneration formula. In many cases this is
sufficient for meaningful computations, as we show in \S\ref{examplessection}.
These results have precise analogies with results in \cite{Jun2}, as explained
in \S\ref{sec:JunLi}. Theorem~\ref{Th:deformation} is a generalization of
\cite[Lem.~3.10]{Jun2}, while Theorem~\ref{Thm: Decomposition} is a
generalization of part of \cite[Cor.~3.13]{Jun2}, where the notation
${\mathfrak M}(\mathfrak{Y}_1^{rel}\cup \mathfrak{Y}_2^{rel},\eta)$ describes an
object playing the role of our $\scrM(X_0,\ttau)$.

The current paper does not, however, include a description of the moduli stack
$\scrM(X_0,\ttau)$ analogous to that given in the proof of
\cite[Lem.~3.14]{Jun2}. There, the moduli space is described by gluing together
relative stable maps to the individual components of $X_0$. However, in general
this will not be the case: while a curve in $\scrM(X_0,\ttau)$ may be glued
schematically from stable maps to individual components of $X_0$, it is not
possible to do this at the logarithmic level, in the sense that the maps to
individual components of $X_0$ may not be interpretable as relative maps. We
give an example in \S\ref{sec:Cubic} in which $X_0$ has three components meeting
normally, with one triple point. Our example features a log curve contributing
to the Gromov-Witten invariant which has a component contracting to the triple
point, and this curve cannot be interpreted as a relative curve on any of the
three irreducible components of $X_0$.

In fact, a new theory is needed to give a more detailed description of the
moduli spaces $\scrM(X_0,\ttau)$ in terms of pieces of simpler curves. In the
follow-up paper \cite{punctured} we define \emph{stable punctured maps}
admitting negative contact orders to replace the relative curves in Jun Li's
gluing formula. Crucially, we will explain how punctured curves can be glued
together to describe the moduli spaces $\scrM(X_0,\ttau)$.

The results described here are analogous to results of Brett Parker proved in
his category of exploded manifolds. He defines Gromov-Witten invariants in this
category in the series of papers \cite{Parker: reg,Parker:
cmp,Parker: Kuranishi,Parker: vfc}. The analogue of logarithmic
deformation invariance, Theorem~\ref{Th:deformation} above, is proved in
\cite[Thms.~5.20 and 5.22]{Parker: vfc}, while Theorem~\ref{Thm: Decomposition}
is analogous to parts of \cite[Thm.~5.22 and Lem.~7.3]{Parker: vfc}. A gluing
formula in terms of Gromov-Witten invariants of individual irreducible
components of $X_0$ is given in \cite[Thms.~4.7 and 5.2]{Parker: gluing}. The aim
in proving a general gluing formula is a full logarithmic analogue of these
theorems.
\vspace{2ex}

This paper has a somewhat long genesis, with the main ideas contained in draft
versions first presented in a talk by B.S.\ at the conference ``Algebraic,
Analytic, and Tropical Geometry'' in Ein Gedi/Israel in Spring~2013. A first
full version was posted on Q.C.'s website in October~2016. The follow-up paper
\cite{punctured} has furthermore been distributed via M.G.'s website since
March~2017.

Several related works have appeared during this long period of preparation.
The 2016 version has been used in \cite{MandelRuddat}. Concerning the
decomposition formula, the one closest to our point of view is
\cite{KimLhoRuddat}, giving a formula of logarithmic Gromov-Witten invariants of the
central fiber $X_0$ of a degeneration with smooth singular locus in terms of
Gromov-Witten invariants of the reducible components. This paper is a full
logarithmic analogue of Jun Li's formula in \cite{Jun2}, without using expanded
degenerations. This case is considerably simpler than the case with points of
multiplicity greater than~$2$ and in particular does not require the
introduction of punctured Gromov-Witten invariants, see \ref{sec:JunLi}
and \cite{punctured}.

A gluing formula for a special case has also been proved by Tony Yue Yu in his
developing theory of Gromov-Witten invariants in rigid analytic geometry
\cite[Thm.~1.2]{Yu}.

Very recently, Ranganathan has suggested an alternative approach to fully
general gluing formulas for logarithmic Gromov-Witten invariants using expanded
degenerations \cite{Ranganathan}.

The structure of the paper is as follows. In \S\ref{sec:prelim}, we review
various aspects of logarithmic Gromov-Witten theory, with a special emphasis on
the relationship with tropical geometry. We develop tropical geometry in the
setup of generalized cone complexes, introduced in \S\ref{Subsect: Cone
complexes}. While this point of view was present in \cite{GS}, we make it more
explicit here, and in particular discuss tropicalization in a sufficient degree
of generality as needed here. As an application, in \ref{Subsect: tropically
decorated moduli spaces} we introduce the refined moduli spaces $\fM(X_0,\ttau)$
appearing in the decomposition formula. \S\ref{artinfansec} reviews the
notion of Artin fans, an algebraic stack associated to any generalized cone
complex. Our decomposition result is based on a decomposition of the fundamental
class in a moduli space of stable log maps to the Artin fan of $X_0$ over $b_0$.

\S\ref{sec:torictovirtual} proves the main theorem, the decomposition
formula. In \S\ref{Sec:decomposition-toroidal} we first prove a general
decomposition of the fundamental class for a space log smooth over the standard
log point. The main insight in \S\ref{Subsect: stable maps to cX} is that
replacing $X_0$ with its relative Artin fan, the moduli space of stable log maps
becomes unobstructed, hence has a fundamental class that can be decomposed. The
main theorem then follows in \S\ref{Subsect: Proof of Main Thm} by lifting this
decomposition to the virtual level.

The remainder of the paper is devoted to applications. As a preparation,
\S\ref{calculationalsection} closes a gap in the literature, building on work of
Nishinou and Siebert in \cite{NS06}. This concerns the \emph{logarithmic
enhancement problem}, the problem of constructing stable logarithmic maps with a
given usual stable map, previously considered only in special cases. We address
the problem through a two-step process. In the first step, we use the tropical
geometry of the situation to identify a proper, birational, logarithmically
\'etale map --- a logarithmic modification --- which reduces the problem to a
situation where no irreducible component of the domain curve maps into the
singular locus of $X_0$ and maps no node into strata of $X_0$ of codimension
larger than~$1$. The second step is the main result of
\S\ref{calculationalsection}, Theorem~\ref{curveconstruction}, giving the number
of logarithmic enhancements in fully general situations, including non-reduced
$X_0$.

\S\ref{examplessection} employs these formulas in the discussion of a
number of hopefully instructive examples. \S\ref{sec:JunLi} contains the already
announced discussion of our decomposition formula in the traditional situation of
\cite{Jun2}. In \S\ref{sec:Cubic} we retrieve the classical number $12$ of nodal
plane sections of a cubic surface passing through two points via a degeneration
into three $\PP^2$'s, blown up in $0,3$ and $6$ points, respectively. The topic of
\S\ref{pointdegensection} is an interpretation of the imposing of point
conditions in tropical geometry via degenerating scheme-theoretic point
conditions in the trivial product $Y\times\AA^1$. The decomposition formula in
this case (Theorem~\ref{Th:decomposition-points}) provides an alternative view
on tropical map counting with point conditions as in \cite{Mikhalkin05,NS06}.
The final section \S\ref{Subsect: FF2} features an example with two rigid
tropical maps such that only one of them arises as the tropicalization of a
stable log map, but the contribution to the virtual count comes from the other, non-realizable rigid tropical map.

%---------------------------------------------------------------------------------
\subsection{Acknowledgements}
That there are analogies with Parker's work is not an accident: We received a
great deal of inspiration from his work and had many fruitful discussions with
him. We also benefited from discussions with Steffen Marcus, Dhruv
Ranganathan, Ilya Tyomkin, Martin Ulirsch and Jonathan Wise.

%---------------------------------------------------------------------------------
\subsection{Conventions}
\label{Sec:convention}
All logarithmic schemes and stacks we consider here are fine and saturated and
defined over an algebraically closed field $\kk$ of characteristic $0$. We will
usually only consider toric monoids, i.e., monoids of the form $P=P_{\RR}\cap M$
for $M\simeq\ZZ^n$, $P_{\RR}\subset M_{\RR}=M\otimes_{\ZZ} \RR$ a rational
polyhedral cone. For $P$ a toric monoid, we write
\[
P^{\vee}=\Hom(P,\NN),\quad P_\RR^{\vee}= \Hom(P,\RR_{\ge0}),\quad P^*=\Hom(P,\ZZ).
\]
For $Q$ a toric monoid and $\varphi:Q\to R$ a homomorphism to the multiplicative
monoid of the $\kk$-algebra $R$, the notation $\Spec(Q\arr R)$ denotes $\Spec R$
with the log structure induced by $\varphi$. For our conventions
concerning graphs see \S\ref{Par: dual graphs}.

%=================================================================================
\section{Preliminaries}
\label{sec:prelim}

%---------------------------------------------------------------------------------
\subsection{Cone complexes associated to logarithmic stacks}
\label{Subsect: Cone complexes}

%---------------------------------------------------------------------------------

\subsubsection{The category of cones}
We consider the category of rational polyhedral cones, which we denote by
$\Cones$. The objects of $\Cones$ are pairs $\sigma = (\sigma_\RR, N)$ where
$N\simeq\ZZ^n$ is a lattice and $\sigma_\RR\subset N_{\RR}=N\otimes_{\ZZ}\RR$
is a top-dimensional strictly convex rational polyhedral cone. A morphism of
cones $\varphi:\sigma_1\arr \sigma_2$ is a homomorphism
$\varphi:N_1\arr N_2$ which takes ${\sigma_1}_\RR$ into ${\sigma_2}_\RR$.
Such a morphism is a \emph{face morphism} if it identifies ${\sigma_1}_\RR$ with
a face of ${\sigma_2}_\RR$ and $N_1$ with a saturated sublattice of $N_2$. If we
need to specify that $N$ is associated to $\sigma$ we write $N_\sigma$ instead.

%---------------------------------------------------------------------------------
\subsubsection{Generalized cone complexes}
\label{ss:cone-complex}
Recall from \cite[II.1]{KKMS} and \cite{ACP} that a
\emph{generalized cone complex} is a topological space with a presentation as
the colimit of an arbitrary finite diagram in the category $\Cones$ with all
morphisms being face morphisms. If $\Sigma$ denotes a generalized cone complex,
we write $\sigma\in\Sigma$ if $\sigma$ is a cone in the diagram yielding
$\Sigma$, and write $|\Sigma|$ for the underlying topological space. A morphism
of generalized cone complexes $f:\Sigma\arr\Sigma'$ is a continuous map
$f:|\Sigma|\arr |\Sigma'|$ such that for each $\sigma_\RR\in\Sigma$, the induced
map $\sigma\arr |\Sigma'|$ factors through a morphism $\sigma\arr\sigma'\in
\Sigma'$. For a cone $\sigma\in\Cones$, we use the same symbol $\sigma$
to also denote the cone complex of all its faces.

Note that two generalized cone complexes can be isomorphic yet not have the same
presentation. This phenomenon does not occur for so-called
\emph{reduced} presentations, which have the defining property that every face
of a cone in the diagram is in the diagram, and every isomorphism in the diagram
is a self-map. By \cite[Prop.~2.6.2]{ACP} any generalized cone complex
has such a reduced presentation. In this paper we only work with reduced presentations of generalized cone complexes.

%---------------------------------------------------------------------------------
\subsubsection{Generalized polyhedral complexes}
\label{ss:poly-complex}
We can similarly define a \emph{generalized polyhedral complex}, where in the
above set of definitions pairs $(\sigma_\RR,N)$ live in the category
$\mathbf{Poly}$ of rationally defined polyhedra. This is more general than
cones, as any cone $\sigma$ is in particular a polyhedron (usually unbounded).
For example, an affine slice of a fan is a polyhedral complex.

%---------------------------------------------------------------------------------
\subsubsection{The tropicalization of a logarithmic scheme}
\label{ss:logschtrop}
Now let $X$ be a Zariski fs log scheme of finite type. For the generic point
$\eta$ of a stratum of $X$, its characteristic monoid $\overline\cM_{X,\eta}$
defines a dual monoid $(\overline\cM_{X,\eta})^{\vee} :=
\Hom(\overline\cM_{X,\eta}, \NN)$ lying in the group $(\overline\cM_{X,\eta})^* :=
\Hom(\overline\cM_{X,\eta}, \ZZ)$, see \S\ref{Sec:convention}, hence a
dual cone
\begin{equation}
\label{Eqn: sigma_eta}
\sigma_{\eta}:=\big((\overline\cM_{X,\eta})^{\vee}_\RR,\,
(\overline\cM_{X,\eta})^*\big).
\end{equation}

If $\eta$ is a specialization of $\eta'$, then there is a well-defined
generization map $\overline\cM_{X,\eta}\arr\overline\cM_{X,\eta'}$ since
we assumed $X$ is a Zariski logarithmic scheme. Dualizing, we obtain a face
morphism $\sigma_{\eta'}\arr \sigma_{\eta}$. This gives a diagram of
cones indexed by strata of $X$ with face morphisms, and hence gives a
generalized cone complex $\Sigma(X)$. We call this the \emph{tropicalization} of
$X$, following \cite[Appendix~B]{GS}.\footnote{This terminology differs
slightly from that of \cite{Ulirsch}, where the tropicalization is a canonically
defined map from the Thuillier analytification $X^\beth$ of $X$ to the
compactified cone complex. Hopefully this will not cause confusion.} For $\sigma\in\Sigma(X)$ we denote by
\[
X_\sigma\subset X
\]
the closure of the corresponding stratum of $X$, endowed with the reduced
induced scheme structure. We refer to these subschemes with reduced induced
structure as \emph{closed strata} of $X$.

This construction is functorial: given a morphism of log schemes
$f:X\arr Y$, the map $f^{\flat}:f^{-1}\overline{\cM}_Y
\arr\overline{\cM}_X$ induces a map of generalized cone complexes
$\Sigma(f):\Sigma(X)\arr\Sigma(Y)$.

\begin{definition}\cite[Def.~B.2]{GS}
\label{Def:simple}
We say $X$ is \emph{monodromy free} if $X$ is a Zariski log scheme and for every
$\sigma\in\Sigma(X)$, the natural map $\sigma\arr |\Sigma(X)|$ is injective on
the interior of any face of $\sigma$. We say $X$
is \emph{simple} if the map is injective on every $\sigma$.
\end{definition}

Here is an example of a Zariski log structure that is monodromy free, but not
simple. Take $X$ to be the Neron $2$-gon, the fibred sum of two $\PP^1$'s
joined at two pairs of points, Thus $X$ has two irreducible components $X_1,X_2$
and two nodes $q_1,q_2$. Take a log structure $\cM_X$ on $X$ with $\ocM_X$
constant with fibres $\NN^2$ along $X_1$, with fibers $\NN$ on
$X_2\setminus\{q_1,q_2\}$ and with generization maps $\ocM_{X,q_i}=\NN^2\to \NN$
to the generic point of $X_2$ the two projections. See also \cite[Expl.~B.1]{GS}
for another example.

Simplicity is, however, true in the Zariski log smooth case over a trivial log point. Such log schemes can in fact be viewed as toroidal pairs without self-intersections and the statement follows readily from the classical treatment in \cite[p.70--72]{KKMS}:

\begin{proposition}
\label{Prop: simplicity criterion}
Let $X$ be a Zariski log scheme, log smooth over $\Spec \kk$ with the trivial log structure. Then $X$ is simple.
\end{proposition}

As remarked in \cite{GS}, more generally we can define the generalized
cone complex associated with a finite type logarithmic stack $X$, in particular
allowing for logarithmic schemes $X$ in the \'etale topology. In fact, one can
always find a cover $X'\arr X$ in the smooth topology with $X'$ a union
of simple log schemes, and with $X''=X'\times_{X} X'$; then define $\Sigma(X)$
to be the colimit of $\Sigma(X'')\rightrightarrows\Sigma(X')$. The resulting
generalized cone complex is independent of the choice of cover. This process is
explicitly carried out in \cite{ACP} and \cite{Ulirsch-thesis}.

\begin{examples}
\begin{asparaenum}
\item 
If $X$ is a toric variety with the canonical toric logarithmic structure, then
$\Sigma(X)$ is abstractly the fan defining $X$. It is missing the embedding of
$|\Sigma(X)|$ as a fan in a vector space $N_{\RR}$, and should be viewed as a
piecewise linear object.

\item
Let $\kk$ be a field and $X=\Spec(\NN\arr \kk)$ the standard log point with
$\cM_X=\kk^{\times} \times\NN$. Then $\Sigma(X)$ consists of the ray $\RR_{\ge
0}$.

\item
Let $C$ be a curve with an \'etale logarithmic structure with
the property that $\overline\cM_C$ has stalk $\NN^2$ at any geometric point, but
has monodromy of the form $(a,b)\mapsto (b,a)$, so that the pull-back of
$\overline\cM_C$ to an unramified double cover $C'\arr C$ is constant
but $\overline\cM_C$ is only locally constant. Then $\Sigma(C)$ can be described
as the quotient of $\RR_{\ge 0}^2$ by the automorphism $(a,b)\mapsto (b,a)$. If
we use the reduced presentation, $\Sigma(C)$ has three cones, one each of
dimension $0$, $1$ and $2$.
\end{asparaenum}
\end{examples}

%---------------------------------------------------------------------------------
\subsection{Artin fans}
\label{artinfansec}

Let $X$ be a fine and saturated algebraic log stack.
We are quite permissive with algebraic stacks, as delineated in 
\cite[(1.2.4)--(1.2.5)]{LogStack}, since we need to work with stacks with 
non-separated diagonal. An Artin stack logarithmically \'etale over
$\Spec\kk$ is called an {\em Artin fan}.

The logarithmic structure of $X$ is encoded by a morphism $X \arr \Log$
to Olsson's stack $\Log$ of fine log structures, see \cite{LogStack}.
One crucial idea developed in the context of the present paper is a
refinement of the stack $\Log$ by an Artin fan that takes into account the
stratification of $X$ defined by $\ol\cM_X$. Following preliminary notes written
by two of us (Chen and Gross), the paper \cite{AW} introduces a canonical Artin
fan $\cA_X$ associated to a logarithmically smooth fs log scheme $X$. This was
generalized in \cite[Prop.~3.1.1]{ACMW14}:

\begin{theorem}
Let $X$ be a logarithmic algebraic stack over $\Spec\kk$ which is
locally connected in the smooth topology. Then there is an initial strict
\'etale morphism $\cA_X\arr \Log$ over which $X\arr \Log$ factors. Moreover,
the morphism $\cA_X\arr\Log$ is representable by algebraic spaces.
\end{theorem}

Note that $\cA_X$ in the theorem is indeed an Artin fan because $\Log$ is logarithmically \'etale over $\Spec\kk$.

If $X$ is a Deligne-Mumford stack, $\cA_X$ can be
constructed from the cone complex $\Sigma(X)$ as
follows. For any cone $\sigma\subset N_{\RR}$, let $P=\sigma^{\vee}\cap M$ be
the corresponding monoid. We write 
\begin{equation}
\label{APdefinition}
\cA_\sigma=\cA_P:=\big[\Spec\kk[P]/\Spec \kk[P^{\gp}]\big].
\end{equation}
This stack carries the standard toric logarithmic structure induced by
descent from the global chart $P \arr \kk[P]$. Then $\cA_X$ is the colimit
\begin{equation}
\label{Eqn: cover for cA_X}
\cA_X=\varinjlim_{\sigma\in\Sigma(X)} \cA_\sigma,
\end{equation}
in the category of sheaves over $\Log$.

\begin{remark}
Unlike $\Sigma(X)$, the formation of $\cA_X$ is not functorial for all
logarithmic morphisms $Y \arr X$. This is a result of the fact that the morphism
$Y \arr \Log$ is not the composition $Y \arr X \arr \Log$, unless $Y \arr X$ is
strict. Note also that not all Artin fans $\cA$ are of the form $\cA_X$, since
$\cA \arr \Log$ may fail to be representable.
\end{remark}

Our next aim is to prove functoriality of the formation of Artin fans for maps
with Zariski log smooth domains, stated as Proposition~\ref{lem:Zariski-factor}
below. We need two lemmas.

\begin{lemma}
\label{lem:Zariski-cover}
Suppose $X$ is a log smooth scheme over the trivial log point
$\Spec\kk$ and with Zariski log structure. Then $\cA_X$ admits a Zariski open
covering $\{\cA_{\sigma} \subset \cA_X\,|\, \sigma\in\Sigma(X)\}$.
\end{lemma}
\begin{proof}
Since $X$ has Zariski log structure, we may select a covering $\{U \arr X\}$ by
Zariski open sets such that $U \arr \cA_{\sigma_U}$ is the Artin fan of
$U$. By the log smoothness of $X$, the morphism $X \arr \cA_{X}$ is smooth, and
hence the image $\tilde U \subset \cA_X$ of $U$ is an open substack.

It remains to show that $\tilde U$ is the Artin fan of $U$. By \cite[\S2.3
and Definition~2.3.2(2)]{AW}, this amounts to show that $\tilde U$ parameterizes
the connected components of the fibres of $U \arr \Log$. Since both $X \arr
\Log$ and $U \arr \Log$ are smooth morphisms between reduced stacks, it suffices
to consider each geometric point $T \arr \Log$. Since $U \subset X$ is Zariski
open, $U_T = T\times_{\Log}U \subset X_{T} = T\times_{\Log}X$ is also Zariski
open. Thus, for each connected component $V \subset U_T$, there is a unique
connected component $V' \subset X_T$ containing $V$ as a Zariski open dense set.
As the set of connected components of $X_T$ is parameterized by
$T\times_{\Log}\cA_X$, we observe that the set of connected components of $U_T$
is parameterized by the subscheme $T\times_{\Log}\tilde U \subset
T\times_{\Log}\cA_X$.
\end{proof}

\begin{lemma}
\label{lem:local-Zariski-factor}
Suppose $X$ is a log smooth scheme with Zariski log structure and
$\tau\in\Cones$. Then any morphism $X \arr \cA_{\tau}$ has a canonical
factorization through $\cA_{X} \arr \cA_{\tau}$.
\end{lemma}
\begin{proof}
By Lemma~\ref{lem:Zariski-cover}, we may select a Zariski covering $\cC := \{\cA_\sigma \subset \cA_X\}$ of $\cA_X$, hence a Zariski covering $\{U_{\sigma} := \cA_{\sigma}\times_{\cA_X}X \subset X\}$ of $X$. We may assume that if $\sigma' \subset \sigma$ is a face, then $\cA_{\sigma'} \subset \cA_{\sigma} \subset \cA_X$ is also in $\cC$. 

Locally, the morphism $U_{\sigma} \arr \cA_{\tau}$ induces a morphism $\tau^{\vee} \arr \Gamma(U_{\sigma},\ocM_{U_{\sigma}}) = \sigma^{\vee}$, hence a canonical $\phi_{\sigma}: \cA_{\sigma} \arr \cA_{\tau}$ through which $U_{\sigma} \arr \cA_{\tau}$ factors. 

To see the local construction glues, observe that the intersection
$\cA_{\sigma_1}\cap\cA_{\sigma_2}$ of two Zariski charts in $\cC$ is again
covered by elements in $\cC$. It suffices to verify that
$\phi_{\sigma_1},\phi_{\sigma_2}$ agree on $\cA_{\sigma'} \in \cC$ if
$\cA_{\sigma'} \subset \cA_{\sigma_1}\cap\cA_{\sigma_2}$. Taking global
sections, we observe that the composition $\tau^{\vee} \arr
\Gamma(U_{\sigma_i},\ocM_{U_{\sigma_i}}) \arr
\Gamma(U_{\sigma'},\ocM_{U_{\sigma'}}) = (\sigma')^{\vee}$ is independent of $i
= 1,2$ as they are determined by the restriction of $U_{\sigma_i} \arr
\cA_{\tau}$ to the common Zariski open $U_{\sigma'}$. Hence
$\phi_{\sigma_1}|_{U_{\sigma'}} = \phi_{\sigma_2}|_{U_{\sigma'}}$.
\end{proof}

\begin{proposition}
\label{lem:Zariski-factor}
Let $X \arr Y$ be a morphism of log schemes. Suppose $X$ is log smooth with Zariski log structure. Then there is a canonical morphism $\cA_X \arr \cA_Y$ such that the following diagram commutes
\[
\xymatrix{
X \ar[r] \ar[d] & Y \ar[d] \\
\cA_X \ar[r] & \cA_Y
}
\]
\end{proposition}
\begin{proof}
By the claimed uniqueness and \'etale descent, the statement can be checked
\'etale locally on $\cA_Y$. We may then assume $\cA_Y = \cA_{\tau}$ for some
$\tau\in\Cones$, for which the statement is exactly
Lemma~\ref{lem:local-Zariski-factor}.
\end{proof}

Using Proposition~\ref{lem:Zariski-factor} we can also define a relative notion of Artin fan for maps with log smooth domains.

\begin{definition}
\label{Def: relative Artin fan}
The \emph{relative Artin fan} for a morphism $X\arr B$ of log schemes with $X$ log smooth with Zariski log structure is defined as the fibre product
\[
\cX= B\times_{\cA_B} \cA_X.
\]
\end{definition}

Assuming $B$ smooth over a trivial log point, $\cX$ has the following explicit description Zariski-locally. Let $P^\vee\to Q^\vee$ be a map of cones in $\Sigma(X)\to \Sigma(B)$. Then the open embedding $\cA_P\to \cA_X$ from Lemma~\ref{lem:Zariski-cover} induces the open embedding
\[
\big[\Spec \kk[P]/\Spec \kk[P^\gp/Q^\gp]\big]\arr \cX,
\]
and these cover $\cX$. Here $\Spec \kk[P^\gp/Q^\gp]$ acts as the subtorus of
$\Spec\kk[P^\gp]$ defined by the kernel of the map $\Spec \kk[P^\gp]\arr \Spec\kk[Q^\gp]$.
\medskip

While not, strictly speaking, needed for this paper, we end this
subsection with the instructive result that giving a log morphism to the Artin
fan $\cA_X$ of a log scheme $X$ is combinatorial in nature, captured entirely by
the induced map of cone complexes.

\begin{proposition}
\label{Prop: Combinatorial nature of Artin fan}
Let $X$ be a Zariski fs log scheme log smooth over $\Spec\kk$. Then for any fs
log scheme $T$ there is a canonical bijection
\[
\Hom_{\mathrm{fs}}(T,\cA_X)\arr \Hom_{\Cones}(\Sigma(T),\Sigma(X)),
\]
which is functorial in $T$.
\end{proposition}

\begin{proof} 
\textsc{Step I. Description of $\cA_X$.} 
By Lemma~\ref{lem:Zariski-cover}, we may select a Zariski covering $\cC :=
\{\cA_\sigma \subset \cA_X\}$, hence a Zariski covering $\{U_{\sigma} :=
\cA_{\sigma}\times_{\cA_X}X \subset X\}$. We may assume that if $\sigma' \subset
\sigma$ is a face, then $\cA_{\sigma'} \subset \cA_{\sigma} \subset \cA_X$ is
also in $\cC$. Thus $\Sigma(X)$ can be presented by the collection of cones
$\{\sigma\}$ glued along face maps $\sigma' \arr \sigma$. In particular, this
shows that $\Sigma(X)=\Sigma(\cA_X)$. Since $\Sigma$ is functorial, there is
then a map $\Hom(T,\cA_X)\arr \Hom(\Sigma(T),\Sigma(X))$. We need to construct
the inverse.

\textsc{Step II. $T$  {is atomic}.} 
Suppose $T$ has unique closed stratum $T_0$ and a global chart $P\arr \cM_T$
inducing an isomorphism $P\simeq \overline{\cM}_{T,\bar t}$ at some point $\bar
t\in T_0$ --- in the language of \cite[Def.~2.2.4]{AW} the logarithmic scheme $T$
is \emph{atomic}. Then with $\tau:=\Hom(P,\RR_{\ge 0})$, $\Sigma(T)=\tau$. 

Using the presentation of $\Sigma(X)$ described in Step I, a map
$\alpha:\Sigma(T) \arr \Sigma(X)$ has image $\alpha(\tau) \subset
\sigma_i\in \Sigma(X)$ for some $i$. Observe that
$\Hom(T,\cA_{\sigma_i})=\Hom(Q_i,\Gamma(T,\overline{\cM}_T))$ by
\cite[Prop.~5.17]{LogStack}. Now $\Gamma(T,\overline{\cM}_T)=P$, and giving a
homomorphism $Q_i\arr P$ is equivalent to giving a morphism of cones
$\tau\arr \sigma_i$. Thus $\Hom(T,\cA_{\sigma_i})=\Hom(\tau,\sigma_i)$.
In particular, $\alpha$ induces a composed map $T\arr
\cA_{\sigma_i}\subset \cA_X$, yielding the desired inverse map
$\Hom(\Sigma(T),\Sigma(X))\arr \Hom(T,\cA_X)$.

\textsc{Step III. $T$ general.}
In general $T$ has an \'etale cover $\{T_i\}$ by {atomic logarithmic schemes},
and each $T_{ij}:=T_i\times_T T_j$ also has such a {covering $\{T_{ij}^k \}$ by
atomic logarithmic schemes}. This gives a presentation $\coprod
{\Sigma(T_{ij}^k)}\rightrightarrows \coprod\Sigma(T_i)$ of $\Sigma(T)$. In
particular, a morphism of cone complexes $\Sigma(T)\arr \Sigma(X)$ induces
morphisms $\Sigma(T_i)\arr\Sigma(X)$ compatible with the maps
${\Sigma(T_{ij}^k)} \arr \Sigma(T_i),\Sigma(T_j)$. Thus we obtain unique
morphisms $T_i\arr \cA_X$ compatible with the morphisms ${T_{ij}^k}\arr T_i,
T_j$, inducing a morphism $T\arr \cA_X$.
\end{proof}

\begin{example}
Let $X=\AA^1$ with the toric log structure. Then $\cA_X=\cA_\NN =
[\AA^1/\GG_m]$. Given an ordinary scheme $\ul T$, a morphism $\ul f:\ul T\arr
\cA_X$ is equivalent to giving a strict log morphism $T\arr \cA_X$, by endowing
$\ul T$ with the pull-back $f^*\cM_{\AA_\NN}$ of the log structure on $\cA_\NN$.
From this point of view, the universal $\GG_m$-torsor $\cP$ on $\cA_\NN$ agrees
with the $\GG_m$-torsor subsheaf of $\cM_{\cA_\NN}$ defined by the generating
section of $\ol\cM_{\cA_\NN}$. Thus the pull-back log structure $\ul
f^*\cM_{\AA_\NN}$ is given by a line bundle $\cL$ on $\ul T$, the line bundle
with associated torsor $\cL^\times= \ul f^* \cP$, and a homomorphism
$\cL\to\cO_{\ul T}$ of $\cO_{\ul T}$-modules defining the structure morphism, or
its restriction to $\cL^\times$. Conversely, the morphism $\ul f$ from $\ul T$
to the quotient stack $\cA_\NN= [\AA^1/\GG_m]$ can be recovered from
$\cL\to\cO_{\ul T}$ by the associated $\GG_m$-equivariant morphism from the
$\GG_m$-torsor $\Spec_{\ul T}\big(\bigoplus_{d\in\ZZ}\cL^{\otimes -d}\big)$ to
$\AA^1_{\ul T}$.

Thus for an arbitrary log structure $\cM_T$ on $\ul T$, a log morphism $f:(\ul
T,\cM_T)\arr \cA_\NN$ is the same data as the restriction of $\cM_T\to\cO_T$ to a
$\GG_m$-torsor subsheaf $\cL^\times\subset\cM_T$. Indeed, such an isomorphism yields the identification of $\ul f^*\cP$ with a $\GG_m$-torsor subsheaf $\cL^\times\subset \cM_T$, and the property of being a log morphism forces $\ul f$ to be associated to the restriction of the structure morphism $\cM_T\arr \cO_{\ul T}$ to $\cL^\times$.

Now Proposition~\ref{Prop: Combinatorial nature of Artin fan} assumes a log
structure $\cM_T$ on $\ul T$ is already given and then says that the set of log
morphisms $T=(\ul T,\cM_T)\arr \cA_X$ equals the set of morphisms
$\Sigma(T)\to\RR_{\ge 0}=\Sigma(X)$ of cone complexes. Indeed, such a morphism
of cone complexes is equivalent to specifying $\ol m\in \Gamma(T,\ol\cM_T)$, and
then the $\GG_m$-torsor subsheaf $\cL^\times\subset \cM_T$ is simply defined by
the preimage of $\ol m$ under $\cM_T\to\ol\cM_T$.
\end{example}

%---------------------------------------------------------------------------------
\subsection{Stable logarithmic maps and their moduli}
\label{Sect: stable logarithmic maps}

This section reviews the theory of stable logarithmic maps developed in
\cite{GS,Chen,AC}, emphasizing the tropical language from
\cite{GS}. Most references in the following are therefore to \cite{GS}, but of
course all results have analogues in \cite{Chen,AC} under the slightly
stronger assumption of global generatedness of $\ol\cM_X$. Note that the
restriction on global generatedness has been removed in \cite{ACMW14} by base
changing to a refinement of the Artin fan $\cA_X$ of $X$.

%---------------------------------------------------------------------------------
\subsubsection{Definition}
We fix a log morphism $X\arr B$ with the logarithmic structure on $X$
being defined in the Zariski topology. Recall from \cite[Def.~1.6]{GS}:

\begin{definition}
\label{Def: stable log map}
A \emph{stable logarithmic map} $(C/S,\bp, f)$ is a commutative diagram
\begin{equation}
\label{stablelogmapdiagram}
\xymatrix@C=30pt
{C\ar[r]^f\ar[d]_{\pi}&X\ar[d]\\
S\ar[r]&B
}
\end{equation}
where
\begin{asparaenum}
\item[(i)]
$\pi:C\arr S$ is a proper, logarithmically smooth and integral
morphism of log schemes together with a tuple of sections
${\mathbf p}=(p_1,\ldots,p_{k})$ of $\ul{\pi}$ such that every geometric
fibre of $\pi$ is a reduced and connected curve, and if $U\subset \ul{C}$
is the non-critical locus of $\ul{\pi}$ then $\overline{\cM}_C|_U
\simeq \ul{\pi}^*\overline{\cM}_S\oplus \bigoplus_{i=1}^{k}p_{i*}\NN_S$.
\item[(ii)]
For every geometric point $\bar s \arr \ul{S}$, the restriction of
$\ul{f}$ to $\ul{C}_{\bar s}$ together with $\mathbf p$ is an ordinary stable map.
\end{asparaenum}
\end{definition}

%---------------------------------------------------------------------------------
\subsubsection{Basic maps}
The crucial concept for defining moduli of stable logarithmic maps is the notion
of {\em basic} stable logarithmic maps. To explain this in tropical terms, we
begin by summarizing the discussion of \cite[\S1]{GS} where more details are
available. The terminology used in \cite{Chen, AC} is {\em minimal} stable
logarithmic maps.

%---------------------------------------------------------------------------------
\subsubsection{Induced maps of monoids}
Suppose given $(C/S, \bp, f)$ a stable logarithmic map with
$S=\Spec(Q'\arr\kk)$, with $Q'$ an arbitrary sharp fs monoid and $\kk$ an
algebraically closed field. {We will use the convention that a point denoted
$p\in C$ is always a marked point, and a point denoted $q\in C$ is always a
nodal point.} Denoting $\underline{Q}' = \pi^{-1} Q'$, the morphism $\pi^\flat$
of logarithmic structures induces a homomorphism of sheaves of monoids $\psi=
\ol\pi^\flat:\underline{Q}'\arr \overline{\cM}_{C}$. Similarly $f^\flat$ induces
$\varphi= \ol f^\flat:f^{-1}\overline{\cM}_X\arr \overline{\cM}_{C}$.

%---------------------------------------------------------------------------------
\subsubsection{Structure of $\psi$}
The homomorphism $\psi$ is an isomorphism when restricted to the complement of
the special (nodal or marked) points of $C$. The sheaf $\overline{\cM}_C$ has
stalks $Q'\oplus\NN$ and $Q'\oplus_{\NN}\NN^2$ at marked points and nodal points,
respectively. The latter fibred sum is determined by a map 
\begin{equation}
\label{Eq:rho_q}
\NN \longrightarrow Q',\quad 1\longmapsto \rho_q
\end{equation} 
and the diagonal map $\NN \arr\NN^2$, see \cite[Def.~1.5]{GS}.
The map $\psi$ at these special points is given by the inclusion $Q'\arr
Q'\oplus\NN$ and $Q'\arr Q'\oplus_{\NN}\NN^2$ into the first component
for marked and nodal points, respectively.

%---------------------------------------------------------------------------------
\subsubsection{Structure of $\varphi$}
\label{Sec:phi}
For $\bar x\in C$ a geometric point with underlying scheme-theoretic point $x$,
the map $\varphi$ induces maps $\varphi_{\bar x}:P_{x}\arr
\overline{\cM}_{C,\bar x}$
for 
\[
P_x:=\overline{\cM}_{X,\ul f(\bar x)}.
\]
Note that $\overline{\cM}_{X,\ul f(\bar x)}$ is independent
of the choice of $\bar x\arr x$ since the logarithmic structure on $X$ is
Zariski. Following Discussion 1.8 of \cite{GS}, we have the following
behaviour at three types of points on $C$:
\begin{enumeratei}
\item
$x=\eta$ is a generic point, giving a local homomorphism\footnote{A 
homomorphism of monoids $\varphi:P\arr Q$ is \emph{local} if 
$\varphi^{-1}(Q^{\times})=P^{\times}$.}
of monoids
\begin{equation}
\label{Eqn: varphi_eta}
\varphi_{\bar\eta}:P_{\eta}\longrightarrow Q'.
\end{equation}
\item
$x=p$ is a marked point, giving the composition
\begin{equation}
\label{Eqn: Contact order at p}
u_p:P_p\stackrel{\varphi_{\bar p}}{\larr} Q'\oplus\NN\stackrel{\pr_2}{\larr} \NN.
\end{equation}
The element $u_p\in P_p^{\vee}$ is called the \emph{contact order} at 
$p$.
\item
$x=q$ is a node contained in the closures of $\eta_1$, $\eta_2$. 
If $\chi_i:P_q\arr P_{\eta_i}$ are the generization maps
there exists a homomorphism 
\[
u_q:P_q\arr \ZZ,
\]
called \emph{contact order at $q$},
such that 
\begin{equation}
\label{veta1veta2diffeq}
\varphi_{\bar\eta_2}\big(\chi_2(m)\big)-\varphi_{\bar \eta_1}\big(\chi_1(m)
\big)=u_q(m)\cdot \rho_q,
\end{equation}
with $\rho_q\neq 0$ given in Equation (\ref{Eq:rho_q}), see \cite[(1.8)]{GS}.
The maps $\varphi_{\bar\eta}\circ\chi_i$ and $u_q$ are equivalent to providing
the local homomorphism $\varphi_{\bar q}:P_q \arr Q'\oplus_{\NN} \NN^2$.
\end{enumeratei}

The choice of ordering $\eta_1,\eta_2$ for the branches of $C$ containing a node
is called an \emph{orientation} of the node. We note that reversing the
orientation of a node $q$ (by interchanging $\eta_1$ and $\eta_2$) results in
reversing the sign of $u_q$.

%---------------------------------------------------------------------------------
\subsubsection{Dual graphs and combinatorial type}
\label{Par: dual graphs}
In this paper, a graph $G$ consists of a set of vertices
$V(G)$, a set of edges $E(G)$ and a separate set of {\em legs} or {\em
half-edges} $L(G)$, with appropriate incidence relations between vertices and
edges, and between vertices and half-edges. We admit multiple edges, loops and
legs. In order to obtain the correct notion of automorphisms, we also implicitly
use the convention that every edge $E \in E(G)$ of $G$ is a pair of {\em
orientations of $E$} or a pair of {\em half-edges of $E$} (disjoint from
$L(G)$), so that the automorphism group of a graph with a single loop is
$\ZZ/2\ZZ$.

Given a stable logarithmic map $(C/S,\bp,f)$ over a logarithmic point, let $G_C$
be the dual intersection graph of $C$. This is the graph which has a vertex
$v_{\eta}$ for each generic point $\eta$ of $C$, an edge $E_q$ joining
$v_{\eta_1}, v_{\eta_2}$ for each node $q$ contained in the closures of
both $\eta_1$ and $\eta_2$, and where $E_q$ is a loop if $q$ is a
double point in an irreducible component of $C$. Note that an ordering of the
two branches of $\ul C$ at a node gives rise to an orientation on the
corresponding edge. Finally, {$G_C$} has a leg $L_p$ with endpoint $v_{\eta}$
for each marked point $p$ contained in the closure of $\eta$. Occasionally we
view $V(G), E(G)$ and $L(G)$ as subsets of $C$ and then write $x\in C$ for a
vertex, edge or leg of $G$ coresponding to a generic point, node or marked point
of C respectively.

\begin{definition}
\label{Def: type of stable log map}
Let $(C/S,\bp,f)$ be a stable logarithmic map over a logarithmic
point $S=\Spec(Q\arr \kk)$. The \emph{combinatorial type} of $(C/S,\bp,f)$
consists of the following data:
\begin{enumerate}
\item
The dual intersection graph $G=G_C$ of $C$.
\item
The genus function\footnote{This was not part of the combinatorial type as defined
in \cite{GS}, but is included here to agree with the type of a tropical map
below, where it is indispensible.} $\bg: V(G)\arr\NN$ associating to $v\in V(G)$
the genus of the irreducible component $C(v)\subset C$.
\item
The map $\bsigma: V(G)\cup E(G)\cup L(G)\arr \Sigma(X)$ mapping $x\in C$ to
$\big(\ol\cM_{X,f(x)}\big)_\RR^\vee\in\Sigma(X)$.
\item
The contact data $\bu=\{u_p,u_q\}$ at marked points $p$ and nodes $q$ of $C$.
\end{enumerate}
\end{definition}

%---------------------------------------------------------------------------------
\subsubsection{The basic monoid}
Given a combinatorial type of a stable logarithmic map
$(C/S,\bp,f)$, we define a monoid $Q$ by first defining its dual
\begin{equation}
\label{Eqn: Basic monoid}
Q^{\vee}=\left\{ \big((V_{\eta})_{\eta}, (e_q)_q\big)
\in \bigoplus_{\eta} P_{\eta}^{\vee}\oplus \bigoplus_q\NN
\,\bigg|\, \forall q: V_{\eta_2}-V_{\eta_1}=e_qu_q\right\}.
\end{equation}
Here the sum is over generic points $\eta$ of $C$ and nodes $q$ of $C$. Readers
with background in tropical gemetry should recognize this monoid as the moduli
cone of tropical curves of fixed combinatorial type, as will be discussed in
\S\ref{tropicalsubsection}. We then set
\[
Q:=\Hom(Q^{\vee},\NN).
\]
It is shown in \cite[\S1.5]{GS}, that $Q$ is a sharp monoid, fine and saturated
by construction as the dual of a finitely generated submonoid of a free
abelian group. Note also that $Q$ indeed only depends on the combinatorial type
of $(C/S,\bp,f)$.

Given a stable logarithmic map $(C'/S',\bp',f')$ over $S'=\Spec(Q'\arr\kk)$ of
the same combinatorial type, we obtain a canonically defined map
\begin{equation}
\label{candefmap}
Q\arr Q'
\end{equation}
which is most easily defined as the transpose of the map
\[
(Q')^{\vee} \arr Q^{\vee}\subset \bigoplus_{\eta} P_{\eta}^{\vee}
\oplus\bigoplus_q \NN,\quad
m\longmapsto \big((\varphi_{\bar\eta}^t(m))_{\eta}, (m(\rho_q))_q\big),
\]
with $\varphi_{\bar\eta}$ and $\rho_q$ defined in \eqref{Eqn: varphi_eta} and \eqref{Eq:rho_q}, respectively.

\begin{definition}[Basic maps]
Let $(C/S, \bp, f)$ be a stable logarithmic map. We say $f$ is \emph{basic}
if at every geometric point $\bar s$ of $S$, the map $Q\arr Q'=\ol\cM_{S,\bar s}$
from \eqref{candefmap} defined by the restriction $(C_{\bar s}/\bar s, \bp_{\bar s},
f|_{C_{\bar s}})$ is an isomorphism.
\end{definition}

%---------------------------------------------------------------------------------
\subsubsection{Degree data and class}
In what follows, $H_2^+(X)$ denotes a semigroup carrying \emph{degree data} for
curves in $X$, which are locally constant in flat families, such as effective
$1$-cycles on $X$ modulo algebraic or numerical equivalence or, working over
$\CC$, classes in singular homology $H_2(X,\ZZ)$ pairing non-negatively with a
K\"ahler form. We require that the moduli spaces of ordinary stable
maps of fixed curve class, genus and number of marked points are of finite
type.

\begin{definition} 
A \emph{class $\beta$} of stable logarithmic maps to $X$ consists of the following:
\begin{enumeratei}
\item
The data $\ul{\beta}$ of an underlying ordinary stable map, i.e., the genus $g$,
a curve class $A\in H_2^+(X)$, and the number of marked points $k$.
\item
Integral elements $u_{p_1},\ldots,u_{p_k}\in |\Sigma(X)|$.
\footnote{We remark that this definition of contact orders is different than
that given in \cite[Def.~3.1]{GS}. Indeed, the definition given there
does not work when $X$ is not monodromy free, and \cite[Rem.~3.2]{GS}
is not correct in that case. However, \cite[Def.~3.1]{GS} may be used
in the monodromy free case.}
\end{enumeratei}
We say a stable logarithmic map $(C/S,\bp,f)$ \emph{is of class $\beta$} if two
conditions are satisfied. First, the underlying ordinary stable map must be of
type $\ul\beta=(g,A,k)$. Second, define the closed subset $\ul{Z}_i\subset
\ul{X}$ to be the union of strata with generic points $\eta$ such that $u_{p_i}$
lies in the image of $\sigma_{\eta}\arr |\Sigma(X)|$. Then for any $i$ we have
$\im(\ul{f}\circ p_i)\subset\ul{Z}_i$ and for any geometric point $\bar
s\arr\ul{S}$ such that $p_i(\bar s)$ lies in the stratum of $X$ with generic
point $\eta$, there exists $u\in\sigma_\eta=\Hom(\ol \cM_{X,\bar\eta},\NN)$
mapping to $u_{p_i}\in |\Sigma(X)|$ making the following diagram
commute:
\[
\hspace{6ex}
\xymatrix{
\ol\cM_{X,\ul f(p_i(\bar s))}\ar[r]^(.35){\overline{f}^{\flat}}\ar[d]_\chi&
\overline{\cM}_{C,p_i(\bar s)}=\overline{\cM}_{S,\bar s}\oplus\NN\ar[d]^{\pr_2}
\phantom{\hspace{12ex}}\\
\ol\cM_{X,\bar\eta}\ar[r]^u&\NN.
}
\]
Here $\chi$ is the generization map. In particular, $s_i$ specifies the contact
order $u_{p_i}$ at the marked point $p_i(\bar s)$ as defined in
\eqref{Eqn: Contact order at p}. \end{definition} We emphasize that the class
$\beta$ does not specify the contact orders $u_q$ at nodes.

\begin{definition}
Let $\scrM(X/B,\beta)$ denote the stack of basic stable logarithmic maps of
class $\beta$. This is the category whose objects are basic stable logarithmic
maps $(C/S,\bp,f)$ of class $\beta$, and whose morphisms $(C/S,\bp,f)\arr
(C'/S',\bp',f')$ are commutative diagrams
\[
\xymatrix@C=30pt
{C\ar[r]^g\ar[d]&C'\ar[r]^{f'}\ar[d]&X\ar[d]\\
S\ar[r]^h&S'\ar[r]&B
}
\]
with the left-hand square cartesian, $S\arr S'$ strict, and
$f=f'\circ g$, $g\circ{\mathbf p}= {\mathbf p}'\circ h$.
\end{definition}

\begin{theorem}
\label{Thm: Stack of stable log maps}
If $X\arr B$ is proper, then $\scrM(X/B,\beta)$ is a proper
Deligne-Mumford stack. If furthermore $X\arr B$ is logarithmically smooth,
then $\scrM(X/B,\beta)$ carries a perfect obstruction theory, defining
a virtual fundamental class $[\scrM(X/B,\beta)]^{\virt}$ in the 
rational Chow group of $\scrM(X/B,\beta)$.
\end{theorem}

\begin{proof}
Under the given assumption that $X$ is a Zariski log scheme,
\cite[Thm.~2.4]{GS} proves that $\scrM(X/B,\beta)$ is a Deligne-Mumford
stack. Properness was shown in logc.cit.\ under a technical assumption, and in
general in \cite{ACMW14}.

The existence of a perfect obstruction theory when $X\arr B$
is logarithmically smooth was proved in \cite[\S5]{GS}.
\end{proof}

%---------------------------------------------------------------------------------
\subsection{Stacks of pre-stable logarithmic curves}
\label{pre-stablelogcurves}
For the obstruction theory in Theorem~\ref{Thm: Stack of stable log
maps} one works over the Artin stack $\fM_B$ of pre-stable logarithmically
smooth curves defined over $B$. Since this stack will be important later on, let
us briefly recall its construction. First, working over a field $\kk$, there is
a stack ${\mathbf M}$ of pre-stable basic logarithmic curves over $\Spec\kk$,
essentially constructed by F.\ Kato in \cite{FKato}. Endowing $\mathbf M$ with
its basic log structure, the fibre product ${\mathbf M}\times_{\Spec\kk} B$ in
the category of log stacks is a fine log stack. We can then define $\fM_B$ using
Olsson's stack over ${\mathbf M}\times B$:
\[
\fM_B:=\Log_{{\mathbf M}\times B}.
\]
Indeed, an object in this stack is a log scheme $T$ with two morphisms $T\to
{\mathbf M}$ and $T\arr B$. The corresponding pre-stable log smooth curve over
$T$ is the logarithmic pull-back to $T$ of the universal pre-stable curve over
$\mathbf M$.

We also consider the following refinements of $\bM$ introduced in
\cite[Def.~2.6]{Behrend-Manin} and further discussed in \cite[p.603]{Behrend-GW}.
Let $G$ be a graph decorated by a map
\[
\bg: V(G)\arr \NN,
\]
associating to each vertex its \emph{genus}. Then there is an algebraic stack
\begin{equation}
\label{Eqn: M(G,bg)}
\text{$\bM(G,\bg)\quad$ of \emph{$\quad(G,\bg)$-marked pre-stable curves}}
\end{equation}
with objects over a
$B$-scheme $S$ given by
\begin{enumerate}
\item
for each $v\in V(G)$, a family of pre-stable curves $C_v\to S$ of genus $\bg(v)$, together with marked sections $x_L:S\to C_v$ defined by the legs $L\in L(G)$ with $v\in L$,
\item
for each edge $E\in E(G)$ with vertices $v,w$, a pair of marked sections $y_v,
y_w$ of $C_v\to S$, $C_w\to S$, respectively.
\end{enumerate}
All marked sections are required to be mutually disjoint and to have image in
the non-critical locus of $\coprod_v C_v\arr S$. Taking the fibred sum of
$\coprod_v C_v$ along the pairs of marked sections associated to the edges, we
may as well view the objects of $\bM(G,\bg)$ as families of marked nodal curves
\begin{equation}
\label{Eqn: universal curve}
(C\arr S,x);
\end{equation}
from this point of view, each edge $E$ defines a family of nodal
points $y_E:S\to C$ and each vertex a closed embedding $C_v\to C$ of a family of
pre-stable curves of genus $g(v)$ and with image a union of irreducible
components. Thus we have a morphism of algebraic stacks
\begin{equation}
\label{Eqn: M(G,bg)->M}
\mathbf{M}(G,\bg)\arr \mathbf{M},
\end{equation}
turning $\mathbf{M}(G,\bg)$ into a logarithmic algebraic stack by pulling back
the log structure from $\mathbf{M}$. Note that on the level of the underlying
stacks, \eqref{Eqn: M(G,bg)->M} induces the identification of the stack quotient
$[\bM(G,\bg)/ \Aut(G,\bg)]$ with the normalization of a closed substack of
$\bM$, defining the well-known stratified structure of $\bM$. See
\cite[XII,\S10]{ACG} for a detailed discussion. Now define
\begin{equation}
\label{Eqn: fM(G,bg)}
\fM_B(G,\bg):=\Log_{\bM(G,\bg)\times B},\quad
\fC_B(G,\bg):=\fM_B(G,\bg)\times_{\bM(G,\bg)} \bC(G,\bg).
\end{equation}

An important feature of the collection of stacks $\bM(G,\bg)$ and in turn of
$\fM_B(G,\bg)$ is their functorial behaviour under \emph{contraction morphisms}
of decorated graphs
\begin{equation}
\label{Eqn: contraction morphism}
\phi:(G,\bg)\arr (G',\bg'),
\end{equation}
that is, an isomorphism of $G'$ with the graph $G/E_\phi$ contracting a subset
of edges $E_\phi\subset E(G)$ such that\footnote{The right-hand side is
identified in Equation~\eqref{Eq:genus} below as the genus of $\phi^{-1}(v')$.}
\[
\bg'(v')= b_1\big(\phi^{-1}(v')\big)+ \sum_{v\in V(\phi)^{-1}(v')} \bg(v)
\]
holds for all $v'\in V(G')$
\cite[Def.~1.3]{Behrend-Manin}. Here $V(\phi):V(G)\arr V(G')$ is the surjection
on the set of vertices defined by $\phi$, and we have in addition a compatible
inclusion $E(\phi): E(G')\mapright{\simeq} E(G)\!\setminus\! E_\phi\subset
E(G)$ of the sets of edges and a bijection $L(\phi):L(G')\to L(G)$ on the sets of
legs. This notion of morphism captures the behavior of the combinatorial type of
pre-stable curves under generization and is indeed compatible with the finite
maps \eqref{Eqn: M(G,bg)->M} to $\bM$:

\begin{proposition}
\label{Prop: Functoriality of M(G,bg)}
For any contraction morphism $(G,\bg)\arr (G',\bg')$ of genus-decorated graphs,
there are finite unramified morphisms of ordinary stacks $\bM(G,\bg)\arr
\bM(G',\bg')$ and
\[
\fM_B(G,\bg) \arr \fM_B(G',\bg').
\]
\end{proposition}

\begin{proof}
By base change and the definition of the log structures it is enough to prove
the statement for the morphism of stacks underlying $\bM(G,\bg)\arr
\bM(G',\bg')$. In this case the statement follows by iterated application of the
clutching morphisms of \cite[Cor.~3.9]{Knudsen}.
\end{proof}

We emphasize that Proposition~\ref{Prop: Functoriality of M(G,bg)} is purely on the level of stacks with no log structures involved. Incorporating log structures in the picture is more subtle and is part of the gluing formalism developed in \cite{punctured}.

%---------------------------------------------------------------------------------
\subsection{The tropical interpretation}
\label{tropicalsubsection}

The basic monoid $Q$ was originally derived from its tropical interpretation,
which will play an important role here. We review this in our general
setting. Given a stable logarithmic map $(C/S,\bp,f)$, 
we obtain an associated diagram of cone complexes,
\begin{equation}
\label{conecomplexdiagram}
\xymatrix@C=30pt
{
\Sigma(C)\ar[r]^{\Sigma(f)}\ar[d]_{\Sigma(\pi)} & \Sigma(X)\ar[d] \\ 
\Sigma(S) \ar[r] & \Sigma(B).
}
\end{equation}
This diagram can be viewed as giving a family of tropical curves
mapping to $\Sigma(X)$, parameterized by the cone complex $\Sigma(S)$.
Indeed, a fibre of $\Sigma(\pi)$ is a graph and the restriction of
$\Sigma(f)$ to such a fibre can be viewed as a tropical curve mapping
to $\Sigma(X)$. We make this precise.

To avoid difficulties in notation, we shall assume that $X$ is simple
(Definition~\ref{Def:simple}). This is not a restrictive assumption in
this paper since we assume $X$ to be log smooth over the
trivial log point $\Spec\kk$, and as $X$ is assumed to be Zariski in any event, it
follows that $X$ is simple (Proposition~\ref{Prop: simplicity criterion}). We use the reduced
presentation of $\Sigma(X)$ from \S\ref{ss:cone-complex}. Then simplicity implies that if
$\tau,\sigma\in\Sigma(X)$ and the image of $\tau$ in $|\Sigma(X)|$ is a face of
the image of $\sigma$, then there is a unique face map $\tau\arr\sigma$ in the
diagram.

The left-hand vertical arrow of \eqref{conecomplexdiagram} is a family
of abstract tropical curves according to the following definition, cf.\ also
\cite[Def.~3.2]{CaChUW}.

\begin{definition}
\label{tropicalcurvedef}
A \emph{(family of)} tropical curves $(G,\bg,\ell)$ over a cone
$\omega\in\Cones$ is a connected graph $G$ together with a bijection
$L(G)\arr\{1,\ldots,k\}$ (\emph{leg ordering}) and two maps
\[
\bg: V(G)\arr \NN,\quad
\ell: E(G)\arr \Hom(\omega\cap N_\omega,\NN)\setminus\{0\}.
\]
For $v\in V(G)$ and $E\in E(G)$ we call $g(v)$ the \emph{genus} of $v$ and $\ell(E)$ the \emph{length function} of $E$.
\end{definition}

The genus of a family of tropical curves $(G,\bg,\ell)$ is defined by
\begin{equation}
\label{Eq:genus}
|\bg|= b_1(G) + \sum_{v\in V(G)} \bg(v).
\end{equation}
Note that given a tropical curve $(G,\bg,\ell)$ over a cone $\omega$ and
$s\in\omega$ not contained in any proper face, then $s\circ\ell$ assigns a
strictly positive real number to each edge. Together with the convention that
legs are infinite length, $(G,s\circ\ell)$ therefore specifies a metric graph,
reproducing the traditional definition of an abstract tropical curve. Hence our
definition makes precise the notion of a family of abstract tropical curves
parameterized by $\omega\in \Cones$.

\begin{construction}
\label{Constr: Cone complex for tropical curve}
We suppress the genus decoration in the notation $(G,\bg,\ell)$ and conflate
$(G,\bg,\ell)$ with its associated morphism of cone complexes
\begin{equation}
\label{Eqn: Gamma for tropical curve}
\Gamma=\Gamma(G,\ell)\stackrel{\pi_\Gamma}{\larr} \omega,
\end{equation}
constructed as follows. For each $v\in V(G)$ take one copy $\omega_v$ of
$\omega$, while for each $E\in E(G)$ take the cone
\begin{equation}
\label{Eqn: cone for edge}
\omega_E=\big\{ (s,\lambda)\in\omega\times\RR_{\ge0}\,\big|\, \lambda\le \ell(E)(s)\big\}.
\end{equation}
The cone $\omega_E$ has two facets, each isomorphic to $\omega$ via projection to the first factor. The corresponding inclusions
\[
s\longmapsto (s,0),\qquad
s\longmapsto (s,\ell(s)).
\]
define face morphisms $\omega_v,\omega_{v'}\arr \omega_E$ for the two vertices $v,v'$
adjacent to $E$. Note this definition is independent of the chosen labellings
$v,v'$ and works also for graphs with loops. Finally, for each $L\in L(G)$ with
adjacent vertex $v$ take $\omega_L=\omega\times\RR_{\ge0}$ with face morphism
$\omega_v\to \omega_L$ defined by the facet $\omega\times\{0\}\subset
\omega_L$. Then $\Gamma$ is the generalized cone complex defined by this
directed sytem in $\Cones$. The morphism to $\omega$ is defined on each
$\omega_E$ by the projection to the first factor.

By construction, each vertex $v\in V(G)$ defines a section of $\pi_\Gamma: \Gamma\to \omega$ denoted as follows:
\begin{equation}
\label{Eqn: v(s)}
\omega\arr \Gamma,\quad s\longmapsto v(s)\in \omega_v.
\end{equation}
Then for $s\in\omega$ not contained in a proper face, the fibre
${\pi_\Gamma}^{-1}(s)$ is the metric graph $(G,s\circ\ell)$ previously defined.
\end{construction}

It is also not hard to replace individual cones as base spaces for families of tropical curves by cone complexes. See \cite[\S3]{CaChUW} for an elaboration of such ideas.

\begin{definition}
\label{tropmaptosigmaxdef}
A \emph{(family of)} tropical maps (from a tropical curve) to $\Sigma(X)$
over a cone $\omega\in\Cones$ is a tropical curve $(G,\bg,\ell)$ over
$\omega$ (Definition~\ref{tropicalcurvedef}) with associated cone complex
$\Gamma=\Gamma(G,\ell)$ (Construction~\ref{Constr: Cone complex for tropical
curve}), together with a morphism of cone complexes
\[
h:\Gamma\larr \Sigma(X).
\]
\end{definition}

\begin{remark}
\label{Rem: Discrete data tropical map}
There are a number of discrete data that we can extract from a tropical map $h:\Gamma\arr\Sigma(X)$ over $\omega\in\Cones$ which are of importance in the sequel.
\begin{enumerate1}
\item
\emph{Image cones:} For a vertex, edge, or leg $x$ of $G$, let
$\omega_x\in\Gamma$ be the cone associated to $x$. Define
\begin{equation}
\label{Eqn: strata function}
\bsigma: V(G) \cup E(G) \cup L(G) \arr \Sigma(X)
\end{equation}
by mapping $x$ to the minimal cone $\tau\in\Sigma(X)$ containing $h(\omega_x)$.
Note that if $E$ is a leg or edge incident to a vertex $v$, then there is an
inclusion of faces $\bsigma(v)\subset\bsigma(E)$ in the (reduced) presentation
of $\Sigma(X)$.
\item
\emph{Contact orders at edges:} Let $E_q\in E(G)$ be an edge with a chosen order
of vertices $v,v'$ (orientation). Then by the definition of the cone
$\omega_{E_q}$ of $\Gamma$ associated to $E_q$ in \eqref{Eqn: cone for
edge}, the image of $(0,1)\in N_{\omega_{E_q}}= N_\omega\times\RR$ under
$h$ defines $u_q\in N_{\bsigma(E_q)}$ such that in $N_{\bsigma(E_q)}$,
\begin{equation}
\label{Eqn: u_q for tropical curve}
h(v(s))-h(v'(s))=\ell(E_q)(s)\cdot u_q
\end{equation}
holds for any $s\in\omega_{E_q}$. Here $v(s)\in\Gamma$ is the section
of $\Gamma\to\omega$ defined in \eqref{Eqn: v(s)}. Reversing the orientation of
$E_q$ results in replacing $u_q$ by $-u_q$.
\item
\emph{Contact orders at marked points:} Similarly, for a leg $L_p\in L(G)$, the image of $(0,1)\in N_{\omega_{L_p}}= N_\omega\times\RR$ defines $u_p \in N_{\bsigma(L_p)}\cap
\bsigma(L_p)$ with $h(\Int(\omega_{L_p}))\subset
\Int(\bsigma(L_p))$.
\end{enumerate1}
\end{remark}

\begin{definition}
\label{Def: type of families of tropical maps}
\begin{enumerate1}
\item
The \emph{type} of a family of tropical maps $h:\Gamma\arr \Sigma(X)$ over
$Q^\vee_\RR\in\Cones$ is the quadruple $\tau=(G,\bg,\bsigma,\bu)$ consisting of the
associated genus decorated graph $(G,\bg)$, the map $\bsigma$ from \eqref{Eqn:
strata function} recording the strata and the contact orders $\bu=\{ u_p,u_q\}$
as defined in Remark~\ref{Rem: Discrete data tropical map}. Note that we are
suppressing the leg numbering, viewing the set $L(G)$ as identical with
$\{1,\ldots ,k\}$.
\item
For a type $\tau$ of a family of tropical maps, $\Aut(\tau)$ denotes the subset of automorphisms of $G$ commuting with the maps $\bg,\bsigma,\bu$.
\item
Given a type $\tau$ of a family of tropical maps, the associated \emph{basic monoid}
$Q(\tau)$ is the dual of the monoid $Q^\vee$ defined in \eqref{Eqn: Basic
monoid}, depending only on $G,\bsigma$ and $\bu$.
\item
If in addition we have given a map
\[
\bA: V(G)\arr H_2^+(X),
\]
we call $\ttau=(\tau,\bA)$ the \emph{decorated type} of a family of
tropical maps, the pair $(h,\bA)$ a \emph{decorated family of tropical
maps} and
\[
|\bA| = \sum_{v\in V(G)} \bA(v)
\]
the \emph{total curve class} of $\bA$.
\end{enumerate1}
\label{Def: type of tropical map}
\end{definition}

Generalizing \eqref{Eqn: contraction morphism} we have a notion of contraction morphism for (decorated) types of families of tropical maps needed below.

\begin{definition}
\label{Def: Contraction for tropical types}
Let $\tau=(G,\bg,\bsigma,\bu)$ and $\tau'=(G',\bg',\bsigma',\bu')$ be types of
families of tropical maps. A \emph{contraction morphism}
$\tau\to \tau'$ is a contraction morphism $\phi:(G,\bg)\to (G',\bg')$ of
decorated graphs \eqref{Eqn: contraction morphism} with the following additional properties:
\begin{enumeratei}
\item
For all $x\in V(G)\cup E(G)\cup L(G)$ the cone $\bsigma'(\phi(x))\in \Sigma(X)$ is a face of $\bsigma(x)$.
\item
For all $x\in E(G')\cup L(G')$ it holds $\bu'(x)=\bu\big(E(\phi)(x)\big)$.
\end{enumeratei}
Similarly, a contraction morphism $\ttau=(\tau,\bA)\arr \ttau'=(\tau',\bA')$ of
decorated types of families of tropical maps is a contraction morphism
$\tau\arr\tau'$ such that $\bA'(v')=\sum_{v\in V(\phi)^{-1}(v')} \bA(v)$ holds
for all $v'\in V(G')$.
\end{definition}

%---------------------------------------------------------------------------------
\subsubsection{Families of tropical curves from logarithmically smooth curves.}
Now suppose 
\[
S=\Spec(Q\arr\kk)
\]
for some monoid $Q$ and $(C/S,\mathbf{p})$ is a family of marked log smooth curves, as in Definition~\ref{Def: stable log map},(i).

\begin{proposition}
\label{Prop: tropicalization of domain curves}
The tropicalization
\begin{equation}
\label{Eqn: tropicalization of domain curve}
\Sigma(\pi):\Sigma(C)\larr \Sigma(S)=Q^{\vee}_{\RR}
\end{equation}
of $(C/S,\mathbf{p})$ naturally has the structure of a family of tropical curves
$(G,\bg,\ell)$ over $Q^\vee_\RR$.
\end{proposition}

\begin{proof}
Take for $G$ the dual intersection graph of $C$.
If $\eta$ is a generic point of $C$, then $\omega_{\eta}=Q^{\vee}_{\RR}$ and
$\Sigma(\pi)|_{\omega_{\eta}}$ is the identity. Thus each fibre of
$\Sigma(\pi)|_{\omega_{\eta}}$ is a point $v$. We take the weight $\bg(v) =
g(C(v))$, the geometric genus of the component $C(v)$ with generic point $\eta$.
The cone of $\Sigma(C)$ defined by a node $q$ of $C$ is
\[
\omega_q=\Hom(Q\oplus_{\NN} \NN^2,\RR_{\ge 0})=Q_{\RR}^{\vee} 
\times_{\RR_{\ge 0}} \RR^2_{\ge 0},
\]
where the maps $Q^{\vee}_{\RR}\arr \RR_{\ge 0}$ and $\RR^2_{\ge 0} \arr \RR_{\ge
0}$ are given by evaluation at $\rho_q\in Q\setminus\{0\}$ and by $(a,b)\mapsto
a+b$, respectively. The projection $\RR^2_{\ge0}\arr \RR_{\ge0}$ to, say, the
first factor, defines an isomorphism
\[
Q_{\RR}^{\vee}  \times_{\RR_{\ge 0}} \RR^2_{\ge 0}
\larr \big\{ (m,\lambda)\in Q_\RR^\vee\times\RR_{\ge0}\, \big|\,
\lambda\le m(\rho_q)\big\}.
\]
Thus defining $\ell(E_q)= \rho_q$, we have a canonical isomorphism $\omega_q\simeq
\omega_{E_q}$ with $\omega_{E_q}$ defined in \eqref{Eqn: cone for edge}.
For a marked point $p_i\in C$, we have $\omega_{p_i}=Q^{\vee}_{\RR} \times
\RR_{\ge 0}$, and $\Sigma(\pi)|_{\omega_{p_i}}$ is the projection onto the first
component, again compatible with the definition of $\Gamma=\Gamma(G,\ell)$ in
Construction~\ref{Constr: Cone complex for tropical curve}.

\end{proof}

%---------------------------------------------------------------------------------
\subsubsection{Families of tropical maps to $\Sigma(X)$ from stable logarithmic maps.}
\label{Sec:maps-to-tropical-maps}
We continue working over a logarithmic point $S=\Spec(Q\arr\kk)$ and assume in
addition given an fs log scheme $X$, which is simple in the sense of
Definition~\ref{Def:simple}.

\begin{proposition}
\label{Prop: log maps define tropical maps}
The tropicalization of a stable logarithmic map $(C/S,\mathbf{p},f)$ over the
logarithmic point $S=\Spec(Q\arr\kk)$ defines a family of tropical maps to
$\Sigma(X)$ over $Q_\RR^\vee$.
\end{proposition}

\begin{proof}
In view of Proposition~\ref{Prop: tropicalization of domain curves} the statement follows readily from the definitions.
\end{proof}

\begin{remark}
An element $x\in V(G)\cup E(G)\cup L(G)$ corresponds to a point $x\in C$
--- either a generic point, a double point, or a marked point. The cone
$\bsigma(x)$ introduced in Remark~\ref{Rem: Discrete data tropical map} is
\[
\bsigma(x)=(P_x)^{\vee}_{\RR}=\Hom(\overline{\cM}_{X,\ul f(\bar x)},\RR_{\ge 0})
\in \Sigma(X),
\]
for any geometric point $\bar x$ mapping to $x$. With this identification of
cones understood, it is a matter of unravelling the definitions that the other
discrete data introduced in Remark~\ref{Rem: Discrete data tropical map}, the
contact orders $u_{L_p},u_{E_q}$, agree with $u_p,u_q$ defined in \S\ref{Sec:phi}.
Note in particular how \eqref{Eqn: u_q for tropical curve} appears as the
tropical manifestation of \eqref{veta1veta2diffeq}. Thus the type of the
tropicalization of a stable logarithmic map, as a family of tropical maps
(Definition~\ref{Def: type of families of tropical maps}), agrees with its
combinatorial type from Definition~\ref{Def: type of stable log map}.
\end{remark}

%---------------------------------------------------------------------------------
\subsubsection{Traditional tropical maps --- the relative situation}
\label{Subsect: Traditional tropical geometry}
A situation of particular interest arises when working over the standard log
point $b_0=\Spec\big(\NN\to \kk[\NN]\big))$. Then all generalized cone complexes
come with a morphism $\pi$ to $\Sigma(b_0)=\RR_{\ge0}$. Taking the fiber of
$\pi$ over $1\in\RR_{\ge0}$ then produces a generalized polyhedral complex as
introduced in \S\ref{ss:poly-complex}. Conversely, let $\pi:\Sigma\to \RR_{\ge 0}$
be a map of generalized cone complexes such that no maximal cone of $\Sigma$
maps to $0\in\RR_{\ge 0}$. Then $\Sigma$ and $\pi$ can be recovered from the
generalized polyhedral complex $\pi^{-1}(1)$ by replacing each polyhedron
$\sigma=(\sigma_\RR,N)$ by the closure of $\RR_{\ge0}(\sigma\times\{1\})$ in
$N_\RR\times\RR$.

If $X$ is a finite type logarithmic stack over the standard log point $b_0$ with associated tropicalization $\pi:\Sigma(X)\arr \Sigma(B_0)=\RR_{\ge0}$, we now write
\[
\Delta(X)= \pi^{-1}(1) \subset \Sigma(X)
\]
for the associated polyhedral complex.

In particular, this discussion applies to the logarithmic scheme $X_0$ and
logarithmically smooth morphism $X_0\arr b_0$ from the main theorem in this
paper, Theorem~\ref{Thm: Decomposition}. Let $(C/S,\bp,f)$ be a stable log map
to $X_0$ with $S=\Spec(Q\to \kk)$ a log point as in
\ref{Sec:maps-to-tropical-maps}, \emph{but now coming with a map to $b_0$}. Let
$\pi_S:Q^\vee\to\NN$ be the tropicalization of $S\arr b_0$. Then the family of
tropical maps $\Sigma(X)\arr \Sigma(X_0)$ over $Q^\vee$ carries the same
information as its restriction to the fiber over $1\in\RR_{\ge 0}$, a family of
maps from metric graphs to $\Delta(X)$ parameterized by the polyhedron
$\pi_S^{-1}(1)\subset Q^\vee$.

The transition from cone complexes to polyhedral complexes provides the
link to more traditional tropical language. In the remainder of this paper we
use cone complexes for most of the general results and polyhedral complexes for
explicit computations. With regards to using both cones and polyhedra as
parameter spaces for families of tropical maps, note that there is no conflict
of language: A family of tropical maps to $\Sigma(X_0)$ over a cone $\sigma$ can
be viewed as a family of maps of metric graphs to $\Sigma(X_0)$
interpreted as a polyhedral complex, now parameterized by $\sigma$ as a
polyhedron.\footnote{It is worthwhile pointing out that the transition
from polyhedral complexes to cone complexes can be subtle \cite{recession
cones}. This is not an issue here since we always have an underlying description
in terms of cone complexes.}

As a matter of notation, we indicate the transition from cone complexes to
polyhedral complexes by overlining. Thus a family of tropical maps
$h:\Gamma\arr\Sigma(X_0)$ over a cone $\sigma$ with a map $\pi_S:
\sigma\arr\RR_{\ge0}$ induces the family of tropical maps
\begin{equation}
\label{Eqn: ol functor}
\ol h:\ol \Gamma\arr \ol\Sigma(X)=\Delta(X)
\end{equation}
over the polyhedron $\ol\sigma=\pi_S^{-1}(1)$.

%---------------------------------------------------------------------------------
\subsubsection{Basic maps and tropical universal families}
\label{Sec:basic-gives-universal-tropical}
Basicness of a stable logarithmic map $(C/S,\bp,f)$ over a logarithmic
point can then be recast as follows.

\begin{proposition}
\label{Prop: basicness versus universally tropical}
Let $(C/S,\bp,f)$ be a stable logarithmic map over a logarithmic point
$S=\Spec(Q\arr \kk)$ and $\tau$ its combinatorial type (Definition~\ref{Def:
type of stable log map}). Then $(C/S,\bp,f)$ is basic if and only if the
family of tropical maps in Proposition~\ref{Prop: log maps define tropical maps} is
universal among families of tropical maps to $\Sigma(X)$ of type $\tau$.
\end{proposition}

\begin{proof}
The definition of the dual of the basic monoid $Q^{\vee}$ precisely encodes the
data of a family of tropical maps to $\Sigma(X)$ over $\sigma=\RR_{\ge 0}$
of type $\tau$ (Definition~\ref{Def: type of tropical map}). Indeed, let
$G_C$ be the dual intersection graph of $C$ from \S\ref{Par: dual graphs}, with
vertices $v_\eta$, edges $E_q$ and legs $L_p$. Then a tuple $((V_{\eta})_{\eta},
(e_q)_q)\in \Int(Q^{\vee}_{\RR})$ specifies a family of
tropical maps
\[
h:\Gamma(G,\ell) \arr \Sigma(X)
\]
over $\RR_{\ge 0}$ of the given type, by defining $\ell(E_q)=e_q$ and
$h|_{\omega_v}$ by mapping $1\in \RR_{\ge0}=\omega_{v_\eta}$ to $V_{\eta}\in
\Sigma(X)$. The type also determines $h$ on each leg $L_p$. It is shown in
\cite[Prop.~1.9]{GS} that if one such tropical map to $\Sigma(X)$ of a certain
type exists then there exists one over $Q_\RR^\vee$; moreover, any other
tropical map of the same type, say over $\sigma\in\Cones$, is obtained from this
one by pull-back via a homomorphism $\sigma\to Q^\vee$.
\end{proof}

\begin{remark}
\label{Rk:added-cones}
Note that if $S$ is not a log point, the diagram \eqref{conecomplexdiagram}
still exists, but the fibres of $\Sigma(\pi)$ may not be the expected ones.
In particular, if $\bar s$ is a geometric point of $S$, there is a functorial
diagram
\[
\xymatrix@C=30pt
{\Sigma(C_{\bar s})\ar[r]\ar[d]&\Sigma(C)\ar[d]\\
\Sigma(\bar s)\ar[r] &\Sigma(S)
}
\]
but this diagram need not be Cartesian due to monodromy in the family $S$. For
example, it is easy to imagine a situation where $C_{\bar s}$ has two
irreducible components and two nodes for every geometric point $\bar s$, but the
nodal locus of $C\arr S$ is irreducible, as there is monodromy interchanging the
two nodes. Then a fibre of $\Sigma(C)\arr \Sigma(S)$ may consist of two vertices
joined by a single edge, while a fibre of $\Sigma(C_{\bar s})\arr \Sigma(\bar
s)$ will have two vertices joined by two edges. Similarly, there may be
monodromy interchanging irreducible components, hence a fibre of
$\Sigma(C)\arr\Sigma(S)$ may have fewer vertices than $C_{\bar w}$ has
irreducible components. This issue can be resolved by redefining moduli
of tropical curves as stacks, following \cite{CaChUW}.
\end{remark}

%---------------------------------------------------------------------------------
\subsubsection{Decorated tropical maps from stable logarithmic maps}
\label{Sec: decorations tropical map}

In the situation of Proposition~\ref{Prop: log maps define tropical maps}, the
tropical map $h:\Sigma(C)\to\Sigma(X)$ comes with the natural decoration
\begin{equation}
\label{Eqn: decorated type log map}
\bA: V(G)\arr  H_2^+(X),\quad
v\longmapsto \big[f(C(v))\big].
\end{equation}
Here $C(v)\subset C$ is the irreducible component corresponding to the vertex
$v$ and $[f(C(v))]$ is the class of $f(C(v))$ in $H_2^+(X)$. The decoration by
curve classes is compatible with the contraction morphisms of decorated graphs
(Definition~\ref{Def: Contraction for tropical types}) defined by generization:

\begin{lemma}
\label{Lem: generization and decorations}
Let $(C/S,\bp,f)$ be a stable logarithmic map to $X$ over some logarithmic
scheme $S$ and $(\tau_{\bar s},\bA_{\bar s})$ with $\tau_{\bar s}= (G_{\bar
s},\bg_{\bar s},\bsigma_{\bar s},\bu_{\bar s})$ its decorated type at the
geometric point $\bar s\to S$ according to Definition~\ref{Def: type of tropical
map} and \eqref{Eqn: decorated type log map}. Then if $\bar s,\bar s'\to S$ are
two geometric points with $\bar s$ a generization of $\bar s'$, the induced map
\[
(\tau_{\bar s'},\bA_{\bar s'})\arr (\tau_{\bar s},\bA_{\bar s})
\]
is a contraction morphism (Definition~\ref{Def: Contraction for tropical types}).
\end{lemma}

\begin{proof}
\cite[Lem.~1.11]{GS} says that $\tau_{\bar s'}\to \tau_{\bar s}$ is a contraction morphism. To check the statement on the curve classes, recall that the preimage of $v\in G_{\bar s}$ in $G_{\bar s}'$ consists of those vertices $v'\in G_{\bar s'}$ with $C_{\bar s'}(v')$ contained in the closure of $C_{\bar s}(v)$. Since this closure defines a flat family of curves, invariance of classes in $H_2^+(X)$ in flat families then implies
\[
\big[f(C_{\bar s}(v))\big] = \sum_{v'\mapsto v} \big[f(C_{\bar s'}(v'))\big],
\]
as claimed.
\end{proof}

%---------------------------------------------------------------------------------
\subsection{Stacks of stable logarithmic maps marked by tropical types}
\label{Subsect: tropically decorated moduli spaces}

We now put ourselves in the situation of the main result in this paper,
Theorem~\ref{Thm: Decomposition}, and assume $X_0\arr b_0$ is logarithmically
smooth and $X_0$ is simple. In particular, curve classes are understood to take values in $H_2^+(X_0)$.

Similar to $\bM(G,\bg)$, we can now define stacks of stable logarithmic maps to
$X_0$ over $b_0$ with restricted decorated types of tropicalizations.

\begin{definition}
\label{Def: ttau-marked stable map}
Let $\ttau=(G,\bg,\bsigma,\bu,\bA)=(\tau,\bA)$ be the decorated type of a
tropical map as defined in Definition~\ref{Def: type of tropical map}. A
\emph{marking by $\ttau$} of a stable logarithmic map $(C/S,\bp,f)$ to $X_0$
over a logarithmic base scheme $S$ over $b_0$ is the following data:
\begin{enumerate}
\item
An isomorphism of $\ul C/\ul S$ with a $(G,\bg)$-marked pre-stable curve
\eqref{Eqn: universal curve}.
\item
The restriction of $\ul f$ to the closed subscheme $Z\subset \ul C$ (a
subcurve or nodal or punctured section of $C$) defined by $x\in V(G)\cup
E(G)\cup L(G)$ factors through $X_{\bsigma(x)}\subset X_0$.
\item
For each geometric point $\bar s\arr S$ with decorated type $\ttau_{\bar
s}=(G_{\bar s},\bg_{\bar s},\bsigma_{\bar s},\bu_{\bar s},\bA_{\bar s})$ of
$(C/S,\bp,f)$, the morphism $(G_{\bar s},\bg_{\bar s})\arr (G,\bg)$ of decorated
graphs from (1) defines a morphism
\[
\ttau_{\bar s}=(\tau_{\bar s},\bA_{\bar s})\arr \ttau=(\tau,\bA)
\]
of decorated types of tropical maps. In particular, there is an associated localization map
\[
\chi_{\tau\tau_{\bar s}}: Q_{\tau_{\bar s}}\arr  Q_\tau
\]
of the corresponding basic monoids.
\item
In the situation of (3),
the preimage $\cK_{\ttau,\bar s}\subset \cM_{S,\bar s}$ of $Q_\tau\setminus\{0\}$ under the composition
\[
\cM_{S,\bar s}\arr \overline \cM_{S,\bar s}= Q_{\tau_{\bar s}}\stackrel{\chi_{\tau\tau_{\bar s}}}{\arr} Q_\tau
\]
maps to $0$ under the structure morphism $\cM_{S,\bar s} \arr \cO_{S,\bar s}$.
\end{enumerate}
\end{definition}

\begin{remark}
Definition~\ref{Def: ttau-marked stable map} calls for some explanations. The
isomorphism in (1) just identifies a contraction of the dual intersection graph
of each geometric fiber of $\ul C\to \ul S$ with a fixed genus-decorated graph
$(G,\bg)$, in a way compatible with generization. Then (2) and (3) ask that the
decorated graphs associated to geometric fibers of the stable log map
$(C/S,\bp,f)$ are refinements of the decorated type $(\tau,\bA)$. Condition~(4)
is maybe the least obvious. It effectively takes the reduction of the moduli
space in unobstructed situations, or on a virtual level later on. We could, in
fact, omit Condition~(4) at the expense of taking reductions in some formulas
below, e.g.\ in $\fM(\cX_0,\tau)$ in Corollary~\ref{Cor: decomposition of
fM(cX_0,beta')}. 
\end{remark}

Given a decorated type $\ttau=(\tau,\bA)$ of tropical maps, we define
\begin{equation}
\label{Eqn: stack of decorated log maps}
\scrM(X_0,\ttau)
\end{equation}
as the stack with objects over a scheme $S$ basic
stable logarithmic maps $(C/S,\bp,f)$ over $b_0$ marked by the decorated type
$\ttau$. We emphasize $\scrM(X_0,\ttau)$ is a moduli space of stable maps \emph{over $b_0$}, but we suppress $/b_0$ in the notation for simplicity. Similarly, we henceforth write $\scrM(X_0,\beta)$ instead of $\scrM(X_0/b_0,\beta)$.
\smallskip

For later use let us also show here that the monoid ideals in Definition~\ref{Def: ttau-marked stable map},(4) define a coherent sheaf of ideals \cite[Prop.~II.2.6.1]{Ogus} in $\cM_{\scrM(X_0,\ttau)}$.

\begin{lemma}
\label{Lem: cK}
For each decorated type $\ttau$ of tropical maps, there exists a unique coherent
sheaf of ideals $\cK_\ttau\subset \cM_{\scrM(X_0,\ttau)}$ with stalks
$\cK_{\ttau,\ol s}$ as defined in Definition~\ref{Def: ttau-marked stable
map},(4).
\end{lemma}

\begin{proof}
The statement follows by \cite[Prop.~II.2.6.1,(2)]{Ogus} since $\cK_{\bar s}$ is defined by a monoid ideal in a chart.
\end{proof}

Let $\beta=(g,A,u_{p_1},\ldots,u_{p_k})$ with $g=|\bg|$, $A=|\bA|$, $k=|L(G)|$.

\begin{proposition}
\begin{enumerate}
\item
The stack $\scrM(X_0,\ttau)$ is a proper Deligne-Mumford stack.
\item
The morphism $\scrM(X_0,\ttau)\arr \scrM(X_0,\beta)$ is finite and unramified.
\end{enumerate}
\label{Prop: scrM(X_0,ttau)}
\end{proposition}

\begin{proof}
By (1) in Definition~\ref{Def: ttau-marked stable map}, we have a morphism
of stacks
\[
\scrM(X_0,\ttau)\arr \scrM(X_0,\beta)\times_\bM \bM(G,\bg) = \scrM(X_0,\beta)\times_\fM \fM(G,\bg).
\]
Condition~(2) in Definition~\ref{Def: ttau-marked stable map} defines a closed
substack of the fibre product on the right-hand side. Prescribing the contact
orders $u_p, u_q$ at $p\in L(G)$, $q\in E(G)$ and the curve classes for the
subcurves of $C$ defined by each $v\in V(G)$ imposes locally constant
conditions, hence select a union of connected components of this closed
substack. Thus $\scrM(X_0,\ttau)$ is isomorphic to a closed substack of the
algebraic stack $\scrM(X_0,\beta)\times_\bM \bM(G,\bg)$, proving (1). The second
statement follows since $\fM(G,\bg)\arr \fM$ is finite and unramified
(Proposition~\ref{Prop: Functoriality of M(G,bg)}).
\end{proof}

%=================================================================================
\section{From toric decomposition to virtual decomposition}
\label{sec:torictovirtual}

Throughout this section, denote by $b_0=
(\Spec\kk,\kk^\times\oplus\NN)$ the standard log point over $\kk$. We also fix a
logarithmically smooth and projective morphism $X_0\to b_0$ of log
schemes.

%---------------------------------------------------------------------------------
\subsection{Decomposition in the log smooth case}
\label{Sec:decomposition-toroidal}

The decomposition formula is based on the following simple fact in toric
geometry. Let $\pi: W\arr \AA^1$ be a morphism of toric varieties with
$\Sigma_\pi:\Sigma_W\to\Sigma_{\AA^1}$ the corresponding morphism of fans,
defined by a homomorphism $N\arr N_{\AA^1}$ of co-character lattices.
We identify $\Sigma_W$ with the cone complex $\Sigma(W)$ associated to $W$ with
its toric log structure, by forgetting the embedding of $|\Sigma_W|$ into
$N_\RR$, and similarly for $\Sigma_{\AA^1}$. For a ray $\gamma\in \Sigma_W$
denote by $D_\gamma\subset W$ the corresponding toric divisor and by
$m_\gamma\in\NN$ the generator of the image of
\[
\ZZ\simeq N_\gamma\stackrel{\Sigma(\pi)}{\larr} N_{\AA^1}\simeq\ZZ.
\]

\begin{proposition}
\label{prop:obvious}
We have the following equality of Weil divisors on $W$:
\[
\pi^*(\{0\}) = \sum_{\gamma} m_\gamma D_\gamma.
\]
\end{proposition}

\begin{proof}
The map $\Sigma(\pi):N\arr\ZZ$ defines a monomial function $z^m$, $m\in
\Hom(N,\ZZ)$ on $W$. It is standard that the order of vanishing of $z^m$ on the
divisor $D_{\gamma}$ is the value of $m$ on the generator of $\gamma\cap N_{\gamma}$.
But this value is precisely $m_{\gamma}$, giving the result.
\end{proof}

Proposition~\ref{prop:obvious} can equivalently be stated as a decomposition of
the fundamental class of $W_0=\pi^{-1}(0)$. Our decomposition theorem is based
on the generalization of this statement to a log smooth morphism $W_0 \arr b_0$
of logarithmic algebraic stacks locally of finite type. Note first that in this
situation, $W_0$ is locally pure-dimensional by log smoothness over $b_0$. Thus
it makes sense to define the fundamental cycle $[W_0]$ as locally finite formal
linear combination of locally top-dimensional integral substacks.

Next, to define the multiplicities $m_\gamma$, consider the morphism of
generalized cone complexes $\Sigma(W_0) \to \Sigma(b_0)$ associated to $W_0 \arr
b_0$ as defined after Proposition~\ref{Prop: simplicity criterion}. We have
$\Sigma(b_0) \simeq \RR_{\geq 0}$ with the lattice $N_{b_0} \simeq \ZZ$. Working
in charts, there is still a correspondence between rays $\gamma \in \Sigma(W_0)$
and integral substacks $W_\gamma \subset W_0$, now locally of
top dimension. Note that if $\sigma\in\Sigma(W_0)$ and
$\gamma\to\sigma$ is a morphism in $\Sigma(X)$, then the pull-back of
$W_\gamma$ to a chart for $W_0$ at a geometric point of the stratum
$W_0(\sigma)$ is contained in the union of all toric divisors for rays
$\gamma'\subset \sigma$ with $\gamma\simeq \gamma'$ in $\Sigma(W_0)$. Hence
$W_\gamma$ may not be locally irreducible if $\Sigma(W_0)$ has cones with
self-identifications. But since we work with cone complexes with reduced
presentations, such rays $\gamma,\gamma'\subset \sigma$ define the same
one-dimensional cone in $\Sigma(W_0)$. These rays may be identified by
self-maps of $\sigma$ or simply correspond to several maps $\sigma^\vee\to
\gamma^\vee$ defined by generization in $\ocM_{W_0}$.

For a ray $\gamma$ with integral lattice $N_{\gamma}$, we have $\gamma\cap
N_{\gamma}\simeq \NN$, and the homomorphism $\ZZ\simeq N_\gamma \arr
N_{b_0}\simeq\ZZ$ is multiplication by an integer $m_\gamma$.

For the following statement recall also the notion of idealized log
structures and idealized log smoothness from \cite[III.1.3 and IV.3]{Ogus}. In a
nutshell, this notion is designed to treat strata of logarithmic spaces, by
adding sheaves of ideals $\cK\subset \cM_X$ defining these strata as part of the
data.

\begin{corollary}
\label{Cor:decomposition-toroidal}
Let $\pi:W_0\arr b_0$ be a log smooth morphism locally of finite type from a
logarithmic algebraic stack to the standard log point $b_0$. Denote
by $[W_0]$ the fundamental cycle of $W_0$, well-defined since $W_0$ is locally
pure-dimensional. Then the following formula holds
\[
[W_0] = \sum_{\gamma} m_\gamma [W_\gamma]
\]
in the group of locally top-dimensional algebraic cycles on $W_0$ \cite{Kresch}.
The sum runs over the one-dimensional cones in the generalized cone complex
$\Sigma(W_0)$ of $W_0$.

Moreover, $W_\gamma$ is idealized log smooth over $b_0$ for some sheaf of ideals
$\cK_\gamma\subset \cM_{W_\gamma}$.
\end{corollary}
\begin{proof}
The claimed equality of cycles can be checked on a cover by smooth charts. We may thus assume that $W_0$ is covered by a neat chart, that is, that we have a commutative
diagram
\[
\xymatrix{
V\ar[d]_{\pi_V} & U \ar[d]\ar[r]^{h}\ar[l]& W_0 \ar[d]^\pi\\
\AA^1 & \Spec\kk\ar@{=}[r]\ar[l]& b_0 
}
\]
where
\begin{inparaenum}
\item
$h$ is an \'etale surjection, 
\item
$\Spec\kk\arr \AA^1$ is the inclusion of the origin and $g:U\to
\Spec\kk\times_{\AA^1} V= \pi_V^{-1}(0)$ is smooth, 
\item
$V$ is the affine toric variety $\Spec\kk[\sigma^\vee\cap N^*]$ defined by
$(\sigma_\RR,N)\in \Sigma(W_0)$ and $\pi_V:V\to\AA^1$ is a toric morphism.
\end{inparaenum}
Thus we have
\[
h^*[W_0]=[U]=g^*([V_0])
\]
via flat pull-back, where $V_0=\pi^{-1}(0)$. Now Proposition~\ref{prop:obvious}
describes $[V_0]$ in terms of the toric divisors $D_{\gamma'}\subset V$ defined by
the rays $\gamma'\subset \sigma$. Thus
\begin{equation}
\label{Eqn: [W_0] in a chart}
h^*[W_0]= \sum_{\gamma'\subset\sigma} m_{\gamma'} g^*(D_{\gamma'}),
\end{equation}
with $m_{\gamma'}$ the generator of the image of $\ZZ\simeq N_{\gamma'}\to
N_{\AA^1}=\ZZ$. Each such $\gamma'$ defines a one-dimensional cone
$\gamma\in\Sigma(W_0)$ with $m_\gamma=m_{\gamma'}$. Moreover, for two different rays
$\gamma',\gamma''\subset\sigma$, the geometric generic points of $D_{\gamma'}$,
$D_{\gamma''}$ map to the same geometric generic point of $W_0$ if and only if
there exists a one-dimensional cone $\gamma\in\Sigma(W_0)$ and morphisms
$\gamma\to\gamma'$ and $\gamma\to\gamma''$. Since $\Sigma(x)$ is the colimit of
such $\sigma$ appearing in neat charts of $W_0$, the equality~\eqref{Eqn: [W_0] in a
chart} in a chart verifies the claimed equation of cycles.

The claim on idealized log smoothness of $W_\gamma$ follows from the local
description as a union of toric strata and the criteria in
\cite[IV.3.1.21 and IV.3.1.22]{Ogus}.
\end{proof}

%---------------------------------------------------------------------------------
\subsection{Logarithmic maps to the relative Artin fan \texorpdfstring{$\cX_0$}{X0}}
\label{Subsect: stable maps to cX}

To lift the decomposition result Corollary~\ref{Cor:decomposition-toroidal} to
the moduli space $\scrM(X_0,\beta) =\scrM(X_0/b_0,\beta)$ of stable
logarithmic maps in Theorem~\ref{Thm: Decomposition}, we factor the map
$\scrM(X_0,\beta) \to \fM_{b_0}$ forgetting the logarithmic map to $X_0$ via an
intermediate log stack that is log \'etale over $\fM_{b_0}=\fM_B\times_B b_0$.
This intermediate log stack is the stack $\fM(\cX_0,\beta')$ of basic
logarithmic maps to the relative Artin fan $\cX_0 = b_0\times_B \cX$ of $X_0$
over $b_0$ (Definition~\ref{Def: relative Artin fan}). Since curve classes do
not make sense on $\cX_0$, we have no stability in $\fM(\cX_0,\beta')$ and
\[
\beta'=(g,u_{p_1},\ldots,u_{p_k})
\]
only keeps the genus and the contact orders at the marked points from
$\beta=(g,A,u_{p_1},\ldots,u_{p_k})$. The point is that $\fM(\cX_0,\beta')$ is
pure-dimensional, has unobstructed deformations and captures the tropical
geometry of the situation, while the decomposition according to
Corollary~\ref{Cor:decomposition-toroidal} has a simple tropical interpretation
on this stack.

\begin{proposition}
\begin{enumerate}
\item
The stack of basic logarithmic maps $\fM(\cX_0,\beta')$ to $\cX_0$ over $b_0$ is algebraic.
\item
The morphism $\fM(\cX_0,\beta)\arr \fM_{b_0}$ forgetting the logarithmic map to
$\cX_0$ is strict and \'etale. 
\end{enumerate}
\label{fM(cX_0)}
\end{proposition}

\begin{proof}
Let $\fC_{b_0}$ denote the universal curve over $\fM_{b_0}$.
By openness of basicness, $\fM(\cX_0,\beta')$ is an open substack of
$\Hom_{\fM_{b_0}}(\fC_{b_0}, \fM_{b_0}\times_{b_0} \cX_0)$. This Hom-stack is algebraic by \cite[Cor.~1.1.1]{Wise-minimality}, proving (1).

For (2), the morphism  $\fM(\cX_0,\beta')\arr \fM_{b_0}$ is strict by definition.
Since $\cA_X\to \cA_B$ is logarithmically \'etale, it follows that $\cX_0$ is logarithmically \'etale over $b_0$. Now \cite[Prop.~3.2]{AW} implies that $\fM(\cX_0,\beta')\to \fM_{b_0}$ is logarithmically \'etale.
\end{proof}

Note that Proposition~\ref{fM(cX_0)},(2) also shows that
$\fM(\cX_0,\beta')$ is log smooth over $b_0$, because $\fM_{b_0}$ is, and that
the obstruction theory of $\scrM(X_0,\beta)$ over $\fM_{b_0}$ induces an
obstruction theory for $\scrM(X_0,\beta)$ over $\fM(\cX_0,\beta')$.

\begin{remark}
\label{Rem: tropical moduli spaces}
Implicit in the discussion in Proposition~\ref{Prop: basicness versus
universally tropical} applied with $X=\Spec\kk$ and in
Remark~\ref{Rk:added-cones} is the fact that log-smoothness of $\bM$ can be used
to relate the moduli space of abstract tropical curves to
the tropicalization of $\bM$, properly interpreted as a stacky cone complex
\cite{CaChUW} -- see precise statement in \cite[Theorem 3.14]{Ulirsch-teichmuller}. In view of Proposition~\ref{fM(cX_0)},(2) we can now similarly
relate the moduli space of tropical maps to $\Sigma(X_0)=\Sigma(\cX_0)$ of class
$\beta'$ to the stacky cone complex associated to $\fM(\cX_0,\beta')$. While we
do not develop the details of this picture here, it should be clear that this
interpretation is at the basis of many arguments in this paper.
\end{remark}
\medskip

We also need the $\tau$-marked refinements $\fM(\cX_0,\tau)$ of
$\fM(\cX_0,\beta')$, similar to $\scrM(X_0,\ttau)$ for $\scrM(X_0,\beta)$.
Omitting the curve class, $\tau$ is now a type of tropical map to $\Sigma(X_0)$ of
total genus $g$ and with $k$ legs (Definition~\ref{Def: type of tropical map},(1)). Then
\[
\fM(\cX_0,\tau)
\]
is defined as in Definition~\ref{Def: ttau-marked stable map} with $\cX_0$
replacing $X_0$ and disregarding the curve classes in Condition~(3).
Analogous to $\cK_\ttau$ for $\scrM(X_0,\ttau)$ constructed in Lemma~\ref{Lem:
cK}, we have a sheaf of ideals
\begin{equation}
\label{Eqn: cK_tau}
\cK_\tau\subset \cM_{\fM(\cX_0,\tau)}.
\end{equation}

We first observe the following analogue of Proposition~\ref{Prop: scrM(X_0,ttau)}.

\begin{proposition}
\begin{enumerate}
\item
The stack $\fM(\cX_0,\tau)$ is algebraic.
\item
The morphism $\iota_\tau:\fM(\cX_0,\tau)\arr \fM(\cX_0,\beta')$ forgetting the
marking by $\tau$ is finite and unramified.
\end{enumerate}
\label{Prop: fM(cX_0,tau) -> fM(cX_0,beta)}
\end{proposition}

\begin{proof}
The proof is identical to the proof of Proposition~\ref{Prop: scrM(X_0,ttau)}.
\end{proof}
\medskip

We are now in position to apply Corollary~\ref{Cor:decomposition-toroidal} to
$\fM(\cX_0,\beta')\to b_0$. The key is the description of the components
$W_\gamma$ in this corollary in terms of rigid tropical maps.

\begin{definition}
\label{Def: Rigid tropical map}
A family of tropical maps $h:\Gamma\to\Sigma(X_0)$ of type $\tau$ is
\emph{rigid} if the corresponding basic monoid $Q(\tau)$ from
Definition~\ref{Def: type of tropical map},(3) is isomorphic to $\NN$.
\end{definition}

In the language of polyhedral complexes, being rigid is equivalent to saying
that the restriction $\ol h: \ol\Gamma\arr \Delta(X)$ of $h$ to the fiber over
$1\in \RR_{\ge 0}=\Sigma(b_0)$ cannot be deformed as a map of generalized
polyhedral complexes. In other words, as a traditional tropical map, any
deformation of $\ol h$ keeping the combinatorial data (i.e.\ of constant type)
is trivial.
\smallskip

The following decomposition of the Artin stack $\fM(\cX_0,\beta')$ according to
rigid tropical curves is the main result of this section.

\begin{theorem} (Virtual Decomposition.)
\label{Thm: Virtual decomposition}
For each irreducible component $W_\gamma$ of $\fM(\cX_0,\beta')$ according to
Corollary~\ref{Cor:decomposition-toroidal} there exists a unique type $\tau$
of a rigid tropical map such that $W_\gamma$ is an irreducible component of
the image of the finite map $\iota_\tau:\fM(\cX_0,\tau)\arr\fM(\cX_0,\beta')$
from Proposition~\ref{Prop: fM(cX_0,tau) -> fM(cX_0,beta)}.

In particular, $\fM(\cX_0,\tau)$ with the sheaf of ideals $\cK_\tau\subset
\cM_{\fM(\cX_0,\tau)}$ from \eqref{Eqn: cK_tau} is idealized
\end{theorem}

\begin{proof}
The logarithmic stack $\fM(\cX_0,\beta')$ is logarithmically smooth over $b_0$ by Proposition~\ref{fM(cX_0)} and since $\fM_{b_0}/b_0$ is logarithmically smooth. Up to a smooth factor, the map
\[
\fM(\cX_0,\beta')\arr b_0
\]
is locally given by base change to the central fibre of the map of toric varieties $\Spec \kk[Q] \arr \Spec \kk[\NN]$ with $Q$ the basic monoid of a tropical map to $\Sigma(X)$ of some type $\tau'$ and $\NN\to Q$ induced by the structure map
\[
\Sigma(\pi): \Sigma(X)\arr \Sigma(B)=\RR_{\ge0}.
\]
Locally the subschemes $W_\gamma$ are defined by the toric divisors in
$\Spec\kk[Q]$, which are in bijection to extremal rays in $Q_\RR^\vee$. Each
extremal ray defines a rigid tropical map, say of type $\tau$. Any
localization map of the associated basic monoids $Q_{\tau'}\arr Q_\tau=\NN$ is
the contraction of the codimension one face dual to the one-dimensional cone in
$Q_{\tau'}^\vee$ defined by $\tau$ By the definition of $\cK_\tau$, the monoid
ideal defining the corresponding toric prime divisor agrees with the ideal in
$Q_{\tau'}$ given by $\cK_\tau$. Since this description is compatible with the
restriction of charts, the first statement follows.

Corollary~\ref{Cor:decomposition-toroidal} also shows that $W_\gamma\to b_0$ is
idealized log-smooth. The corresponding sheaf of ideals has just been
checked to agree with $\cK_\tau$ locally along $W_\gamma$. Since
$\fM(\cX_0,\beta')\to b_0$ is log \'etale, $W_\gamma\to b_0$ is even idealized
log \'etale.
\end{proof}

\begin{corollary}
\label{Cor: decomposition of fM(cX_0,beta')}
We have the following equality of top-dimensional algebraic cycles in the pure-dimensional algebraic stack $\fM(\cX_0,\beta')$:
\[
\big[\fM(\cX_0,\beta')\big] = \sum_\tau m_\tau\cdot \big[\iota_\tau(\fM(\cX_0,\tau))\big].
\]
The sum is over all types $\tau$ of rigid tropical maps to $\Sigma(X)$ and
$m_\tau\in\NN\setminus\{0\}$ is the projection of the generator of the dual
basic monoid $Q_\tau^\vee\simeq\NN$ to $\Sigma(b_0)=\RR_{\ge 0}$.
\end{corollary}

\begin{proof}
The statement merely spells out the definition of the multiplicities $m_\tau$ in
Corollary~\ref{Cor:decomposition-toroidal}.
\end{proof}

%=================================================================================
\subsection{Proof of the Decomposition Theorem}
\label{Subsect: Proof of Main Thm}

To prove the Main Theorem, Theorem~\ref{Thm: Decomposition}, it remains to apply
the virtual bivariant machinery developed by Costello \cite{CO} and Manolache
\cite{Mano}. We need two lemmas.

\begin{lemma}
\label{Lem: degree of fM_tau to fM}
The degree of the finite map
\[
\iota_\tau:\fM(\cX_0,\tau)\arr \iota_\tau\big(\fM(\cX_0,\tau)\big)\subset\fM(\cX_0,\beta')
\]
from Proposition~\ref{Prop: fM(cX_0,tau) -> fM(cX_0,beta)},(2) over any irreducible component of the image is $|\Aut(\tau)|$.
\end{lemma}

\begin{proof}
The description of the smooth cover of $\fM(\cX_0,\beta')$ given in the proof of
Theorem~\ref{Thm: Virtual decomposition} shows that each geometric generic
point $\Spec K\to \fM(\cX_0,\beta')$ of $\iota_\tau\big(\fM(\cX_0,\tau)\big)$
is a basic logarithmic map to $\cX_0$ over $b_0$, defined over $K$ and
with basic monoid $Q(\tau)=\NN$ and tropical type isomorphic to $\tau$. Thus a
geometric generic point of $\fM(\cX_0,\tau)$ is a basic logarithmic map
$(C/S,\bp,f)$ to $\cX_0$ over a standard logarithmic point $S=\Spec(\NN\arr K)$.
Writing $\tau=(G,\bg,\bsigma,\bu)$, the fibre of $\iota_\tau$ over $(C/S,\bp,f)$
is an isomorphism of the dual intersection graph of $C$ with $G$ identifying
$\bg,\bsigma,\bu$ with the genera, strata and contact orders of $(C/S,\bp,f)$.
The statement now follows by observing that the automorphism group $\Aut(\tau)$
of the decorated graph $\tau$ acts simply transitively on this set of
isomorphisms of graphs.
\end{proof}

As an intermediate object we define the \emph{stack of basic stable logarithmic
maps marked by a tropical type $\tau$} by 
\begin{equation}
\label{Eqn: scrM_tau(X_0)}
\scrM_\tau(X_0,\beta):= \fM(\cX_0,\tau)\times_{\fM(\cX_0,\beta')}\scrM(X_0,\beta).
\end{equation}
Compared to $\scrM(X_0,\ttau)$, this stack keeps the total curve class $A$ from
$\beta=(g,A,u_{p_1},\ldots,u_{p_k})$, but drops the restriction on the distribution of $A$ to the subcurves given by the vertices.

For the following statement recall that $\scrM(X_0,\ttau)$ is the
stack defined in \eqref{Eqn: stack of decorated log maps} of basic stable log
maps over $b_0$ marked by the decorated type $\ttau$ and
\[
j_\ttau: \scrM(X_0,\ttau)\arr \scrM(X_0,\beta)
\]
is the morphism forgetting the marking.

\begin{lemma}
\label{Lem: decomposition of scrM according to curve class}
Let $\tau=(G,\bg,\bsigma,\bu)$ be the type of a tropical map to $\Sigma(X_0)$ and $\beta=(g,A,u_{p_1},\ldots,u_{p_k})$. Then we have the decomposition
\[
\scrM_\tau(X_0,\beta) = \coprod_{\bA}\scrM\big(X_0,\ttau\big),
\]
where the sum is over all
$\bA: V(G)\to H_2^+(X_0)$ with $|\bA|=A$ and $\ttau=(\tau,\bA)$.
\end{lemma}

\begin{proof}
The result follows since the map $\bA: V(G)\to H_2^+(X_0)$ of curve classes is
locally constant on $\scrM_\tau(X_0,\beta)$.
\end{proof}

Before stating the Main Theorem, we note that $\scrM_\tau(X_0,\beta)$ inherits a
perfect obstruction theory\footnote{$\fE$ is the gothic letter ``E''.} $\fE_\tau$
over $\fM(\cX_0,\tau)$ from the perfect obstruction theory $\fE$ of
$\scrM(X_0,\beta)$ over $\fM(\cX_0,\beta')$ by base change by $\iota_\tau:
\fM(\cX_0,\tau)\to \fM(\cX_0,\beta')$. Restricting to the open substacks
$\scrM(X_0,\ttau)\subset \scrM_\tau(X_0,\beta)$ in Lemma~\ref{Lem: decomposition
of scrM according to curve class}, we also have an obstruction theory
$\fE_\ttau$ on $\scrM(X_0,\ttau)$. If $\tau$ is rigid, $\fM(\cX_0,\tau)$ is
pure-dimensional of the same dimension as $\fM(\cX_0,\beta')$. Thus we have
virtual fundamental classes
\[
[\scrM(X_0,\beta)]^\virt,\quad [\scrM_\tau(X_0,\beta)]^\virt,\quad [\scrM(X_0,\ttau)]^\virt
\]
on the moduli spaces $\scrM(X_0,\beta)$, $\scrM_\tau(X_0,\beta)$ and $\scrM(X_0,\ttau)$.
\medskip

Here is our main theorem, stated as Theorem~\ref{Thm: Decomposition} in the introduction.

\begin{theorem}
\label{Thm: Main}
For any $\beta=(g,A,u_{p_1},\ldots,u_{p_k})$ we have the equality
\[
[\scrM(X_0,\beta)]^\virt = \sum_{\ttau=(\tau,\bA)}  \frac{m_\tau}{|\Aut(\tau)|}\,
{j_\ttau}_*[\scrM(X_0,\ttau)]^\virt
\]
in the Chow group of the underlying stack $\ul\scrM(X_0,\beta)$ with
coefficients in $\QQ$. The sum is over all isomorphism classes of decorated
types of rigid tropical maps $\ttau=(G,\bg,\bsigma,\bu,\bA)=(\tau,\bA)$ of total
genus $|\bg|=g$, total curve class $|\bA|=A$ and $|L(G)|=k$.
\end{theorem}

\begin{proof}
By Corollary~\ref{Cor: decomposition of fM(cX_0,beta')} and Lemma~\ref{Lem: degree of fM_tau to fM} we can write the fundamental class of $\fM(\cX_0,\beta')$ as
\begin{equation}
\label{Eqn: dec of f[M(cX_0,beta')]}
[\fM(\cX_0,\beta')]= \sum_\tau \frac{m_\tau}{|\Aut(\tau)|}\,
 {\iota_\tau}_*[\fM(\cX_0,\tau)].
\end{equation}
For each $\tau$, compatibility of virtual pull-back with push-forward
\cite[Thm.~4.1,(3)]{Mano} applied to the cartesian square
\[
\xymatrix{
\scrM_\tau(X_0,\beta) \ar[d]_q&\hspace{-27pt}= {\displaystyle\coprod_\bA} \scrM(X_0,(\tau,\bA))\ar[r]^(.57){j_\tau}&\scrM(X_0,\beta)\ar[d]^p\\
\fM(\cX_0,\tau) \ar[rr]^{\iota_\tau} &&\fM(\cX_0,\beta')
}
\]
yields
\[
p_\fE^!{\iota_\tau}_*[\fM(\cX_0,\tau)] =
{j_\tau}_* q_{\fE_\tau}^! [\fM(\cX_0,\tau)] =
{j_\tau}_* [\scrM_\tau(X_0,\beta)]^\virt.
\]
Moreover, from Lemma~\ref{Lem: decomposition of scrM according to curve class} and the definition of $\fE_\ttau$ by restriction of $\fE_\tau$, it holds
\[
[\scrM_\tau(X_0,\beta)]^\virt = \sum_\bA [\scrM(X_0,(\tau,\bA))]^\virt.
\]
Plugging the last two equalities into \eqref{Eqn: dec of f[M(cX_0,beta')]} now
gives the desired result:
\begin{eqnarray*}
[\scrM(X_0,\beta)]^\virt &=&
p_\fE^![\fM(\cX_0,\beta')] \ =\ \sum_\tau \frac{m_\tau}{|\Aut(\tau)|}\,
p_\fE^!{\iota_\tau}_*[\fM(\cX_0,\tau)]\\
&=& \sum_\tau \frac{m_\tau}{|\Aut(\tau)|}\,
{j_\tau}_* [\scrM_\tau(X_0,\beta)]^\virt\\
&=&  \sum_{\ttau=(\tau,\bA)} \frac{m_\tau}{|\Aut(\tau)|}\,
 {j_\ttau}_*[\scrM(X_0,\ttau)]^\virt.
\end{eqnarray*}
\end{proof}

%=================================================================================
\section{Logarithmic modifications and transversal maps}
\label{calculationalsection}

There is a general strategy which is often useful for constructing stable
logarithmic maps. This is the most powerful tool we have at our disposal at the
moment; eventually, the hope is that gluing technology will replace this
construction. However, we expect it to be generally useful, as
illustrated by the examples in the next section.

Suppose we wish to construct a stable logarithmic map to $X/B$, and as usual $X$
logarithmically smooth with a Zariski logarithmic structure over one-dimensional
$B$ with logarithmic structure induced by $b_0\in B$. Suppose further we wish
the stable logarithmic map to map into the fibre $X_{0}$ over $b_0$.
Generalizing a method introduced in \cite{NS06}, this construction is
accomplished by the following two-step process: (1) Apply a logarithmic
modification\footnote{A logarithmic modification is a proper, birational and log \'etale morphism \cite{FKato2}.} of $X$ to reduce to a transverse situation. (2) Study logarithmic enhancements in the transverse case.

%---------------------------------------------------------------------------------
\subsection{Logarithmic modifications}
First, we will choose a logarithmic modification $h:\tilde X\arr X$. The
modification $h$ is chosen to accommodate a situation at hand --- in our
applications the datum of a rigid tropical map. 

Given a modification $h$, \cite{AW} constructed a morphism
$\scrM(h):\scrM(\tilde X/B)\arr \scrM(X/B)$ of moduli stacks of basic
stable logarithmic maps, satisfying
\[
\scrM(h)_*([\scrM(\tilde X/B)]^{\virt})=[\scrM(X/B)]^{\virt}.
\]
The construction of $\scrM(h)$ is as follows. Given a stable logarithmic map
$\tilde f:\tilde C/S\arr \tilde X/B$, one obtains on the level of schemes
the stabilization of $h\circ \tilde f$, i.e., a factorization of $h\circ \tilde
f$ given by 
\[
\underline{\tilde C}/\underline{S}\stackrel{g}{\longrightarrow}
\underline{C}/\underline{S}\arr \underline{X}
\]
such that $\underline{C}/\underline{S}\arr \underline{X}$ is a stable map. One
gives $\underline{C}$ the logarithmic structure $\cM_{C}:=g_*\cM_{\tilde C}$,
and with this logarithmic structure one obtains a factorization of $h\circ
\tilde f$ through $C$ at the level of log schemes, giving $f:C/S \arr X/B$.
Note that this is one of the rare occasions where push-forward of
logarithmic structures behaves well. If $\tilde f$ was basic, there is no
expectation that $f$ is basic, but by \cite[Prop.~1.22]{GS} there is a unique
basic map with the same underlying stable map of schemes such that the above
constructed $f$ is obtained by pull-back from the basic map. This
yields the map $\scrM(h)$.

%---------------------------------------------------------------------------------
\subsection{Transverse maps, logarithmic enhancements, and strata}
Second, if we have a stable map to $\underline{X}_0$ which interacts
sufficiently well with the strata, we will compute in
Theorem~\ref{curveconstruction} the number of log enhancements of this curve.
This generalizes a key argument of Nishinou and Siebert in \cite{NS06}. There
are two differences: our degeneration $X \arr B$ is only logarithmically smooth
and not necessarily toric; and the fibre $X_0$ is not required to be reduced.
Not requiring $X_0$ to be reduced makes the situation more complex and perhaps
explains why it was avoided in the past; we hope our treatment here will find
further uses. The precise meaning of ``interacting well with logarithmic strata"
is as follows:

\begin{definition}[{Transverse maps and constrained points}]
\label{transversaldef}
Let $X\arr B$ be a logarithmically smooth morphism over $B$ one-dimensional
carrying the divisorial logarithmic structure $b_0\in B$ as usual. Let
$X_0^{[d]}$ denote the union of the codimension $d$ logarithmic
strata of $X_0$. Suppose $\ulf:\underline{C}/\Spec
\kk\arr \underline{X}_0$ is a stable map. We say that $\ulf$ is a
\emph{transverse map} if the image of $\ulf$ is contained in $X_0^{[0]}\cup
X_0^{[1]}$, and $\ulf^{-1}(X_0^{[1]})$ is a finite set.

We call a node $q\in \ulC$ a {\em constrained node} if $\ulf(q)\in X^{[1]}_0$
and otherwise it is a {\em free node}. Similarly a marked point $x\in \ulC$ with
$\ulf(x)\in X^{[1]}_0$ is a {\em constrained marking}, otherwise it is a {\em
free marking}.
\end{definition}

The term ``transverse map" is shorthand for ``a map meeting strata in a
logarithmically transverse way".
\medskip

\paragraph{\bf Cones and strata in the transverse setting} For the rest of this
section strata of higher codimension are irrelevant and we henceforth assume
$X_0= X_0^{[0]} \cup X_0^{[1]}$. Then $\Sigma(X_0)$ is a purely two-dimensional
cone complex, with rays in bijection with the irreducible components of $X_0$.
There are two types of two-dimensional cones: first, there is one cone for each
component of the double locus $X_0^{[1]}$; second, there is one cone for each
other component of $X_0^{[1]}$, forming a smoth divisor in the regular locus of $X_0$.

\paragraph{\bf Logarithmic enhancement of a map} We codify what it means to take
a stable map and endow it with a logarithmic structure:

\begin{definition}
Let $X \arr B$ be as above and $\ulf: \uC \arr \uX_0$ a stable map. A {\em
logarithmic enhancement} $f: C \arr X$ is a stable logarithmic map whose
underlying map is $\ulf$. Two logarithmic enhancements $f_1,f_2$ are {\em
isomorphic enhancements} if there is an isomorphism between $f_1$ and $f_2$
which is the identity on the underlying $\ulf$. Otherwise we say they are {\em
non-isomorphic} or {\em distinct enhancements}.
\end{definition}

\paragraph{\bf Discrete invariants in the transverse case}

\begin{notation}
\label{Not: Discrete invariants transverse case}
Let $\ulf: \ul C/\Spec\kk \arr \ul X_0$ be a transverse map and $x\in \ul C$ a
closed point with $\ulf(x)$ contained in a stratum $S\subset X_0^{[1]}$ and let
$\eta\in \ul C$ be a generic point with $x\in\cl(\eta)$. We now assoicate a number of invariants to the pair $(\eta,x)$, all related to the rank two toric monoid $P_x=\ocM_{X,\ulf(x)}$. Denote by $m_{\eta,x}\in P_x$ the generator of the
kernel of the localization map $P_x\arr \ocM_{X,\ulf(\eta)}\simeq\NN$ and by
$m'_{\eta,x}\in P_x$ the generator of the other extremal ray. Denote by
$n_{\eta,x},n'_{\eta,x}\in P_x^\vee$ the dual generators of the extremal rays of
$P_x^\vee$, satisfying $\langle n_{\eta,x},m_{\eta,x}\rangle=0$. A third
distinguished element $\rho_x\in P_x$ is defined by pulling back the generator of
$\Gamma(B,\ocM_B)=\NN$ under the log morphism $X\arr B$.
\end{notation}

For the following discussion denote by $\ell(m)$ the integral length of an
element $m\in M\otimes_\ZZ\QQ$, that is, for $m\neq0$ the maximum of
$\alpha\in\QQ_{>0}$ with $\alpha^{-1}\cdot m\in M$, while
$\ell(0)=0$.

\begin{definition}
\begin{enumerate}
\item
The \emph{index} of $x\in\ul C$ or of the stratum $S\subset X_0^{[1]}$
containing $\ulf(x)$ is the index of the sublattices in $P_x^\gp$ or in $P_x^*$
generated by $m_{\eta,x},m'_{\eta,x}$ and $n_{\eta,x},n'_{\eta,x}$,
respectively, that is,
\[
\ind(S)= \ind_x = \langle n_{\eta,x},m'_{\eta,x}\rangle = \langle n'_{\eta,x}, m_{\eta,x}\rangle. 
\]
For a constrained node $x=q$, the \emph{length} $\lambda(q)=\lambda(S)\in\QQ$ is the
integral length of the interval $\rho_q^{-1}(1)$ when viewing $\rho_q$ as a map
$P^*_q\otimes_\ZZ\QQ \to\QQ$.
\item
If $\eta\in \ul C$ is a generic point with $x\in\cl(\eta)$, denote by
$w_{\eta,x}\in \NN\setminus\{0\}$ the  {local} intersection number of $\ul
f|_{\cl(\eta)}$  {at $x$} with $S$ inside the irreducible component of $X_0$ containing
$\ulf(\eta)$. 
\end{enumerate}
\label{Def: Discrete invariants transverse case}

When the choice of $x$ and $\eta$ is understood we write $m_1=m_{\eta,x}$,
$m_2=m'_{\eta,x}$, $n_1=n_{\eta,x}$, $n_2=n'_{\eta,x}$, $\rho_x\in P_x$ and
$w_1=w_{\eta,x}$.
\end{definition}

\paragraph{\bf Relations between discrete invariants}
\begin{lemma}
\label{Lem: multiplicities versus rho}
In the situation of Definition~\ref{Def: Discrete invariants transverse case}
denote by $\mu_1$ the multiplicity of the irreducible component of $\ul X_0$
containing $\ulf(\eta)$. If the stratum $S\subset X_0^{[1]}$ is contained in
two irreducible components of $\ul X_0$, denote by $\mu_2$ the multiplicity of
the other component and otherwise define $\mu_2=0$. 
\[
(1)\ \ \mu_i= \langle n_i,\rho_x\rangle.\quad
(2)\ \ \ind_x \cdot\rho_x =\mu_2 m_1+\mu_1 m_2.\quad
(3)\ \ \lambda(q)=\frac{\ell(\rho_q)\cdot \ind_q}{\mu_1\mu_2}.
\]
In particular, if $X_0$ is reduced then $\mu_i\in\{0,1\}$ for all $i$ and\, $\ind_x\cdot \rho_x= m_1+m_2$, $\lambda(q)=\ell(\rho_q)\cdot \ind_q$.
\end{lemma}

\begin{proof}
For (1) note that since $n_i\in P_x^\vee$ is a primitive vector with $\langle
n_i,m_i\rangle =0$, the pairing with $n_i$ computes the integral distance from
the face $\NN \cdot m_i$ of $P_x$. Now \'etale locally, the log smooth
morphism $X\arr B$ is the composition of a smooth map with $\Spec
\kk[P_x]\arr \Spec\kk[t]$ defined by sending $t$ to $z^{\rho_x}\in \kk[P_x]$. Hence the
multiplicity $\mu_i$ equals the integral distance of $\rho_x$ to $\NN\cdot m_i$,
that is, the image of $\rho_x$ under the quotient map $P_x\to P_x/\NN
m_i\simeq\NN$.

For (2), since the sublattice of $P^\gp$ generated by $m_1, m_2$ is of index
$\ind_x$, there are $a_1,a_2\in\ZZ$ with $\ind_x\cdot\rho_x= a_1 m_1+a_2 m_2$.
Pairing with $n_1$ and using~(1) and the definition of $\ind_x$ yields
\[
\ind_x\cdot \mu_1= \ind_x\cdot\langle n_1,\rho_x\rangle = a_2 \langle n_1,m_2\rangle
= a_2\cdot\ind_x.
\]
This shows $a_2=\mu_1$, and similarly $a_1=\mu_2$, yielding the claim.

To prove (3) note that (1) implies
\[
\langle \mu_2 n_1,\rho_q\rangle=\mu_1\mu_2 =
\langle \mu_1 n_2,\rho_q\rangle.
\]
Hence $\rho_q: P_q^*\to \NN$ maps both $\mu_2 n_1$ and $\mu_1 n_2$ to
$\mu_1\mu_2$. Since $\lambda(q)$ is defined as the integral length of
$\rho_q^{-1}(1)$, we see that $\mu_1\mu_2\cdot\lambda(q)$ equals the integral
length of $\mu_2 n_1-\mu_1 n_2$. Choosing an isomorphism of $P_q$ with
\[
\ZZ^2\cap\big(\RR_{\ge 0}\cdot (1,0) +\RR_{\ge 0}\cdot (r,s)\big)
\]
with $r,s>0$ pairwise prime and $\rho_q$ mapping to $(a,c)$, then
\[
m_1=(1,0),\quad m_2=(r,s),\quad
\mu_1= c,\quad \mu_2= as-cr,\quad \ind_q=s.
\]
In the dual lattice $P_q^*\simeq\ZZ^2$ we have $n_1=(0,1)$, $n_2=(s,-r)$ and
$\mu_2 n_1-\mu_1 n_2= s\cdot(-c,a)$ has integral length $\ind_q\ell(\rho_q)$.
Thus $\lambda(q)= \ind_q\ell(\rho_q)/\mu_1\mu_2$ as claimed.
\end{proof}

\paragraph{\bf Necessary conditions for enhancement}
As we now show, the data listed in Definition~\ref{Def: Discrete
invariants transverse case} determine the discrete invariant $u_x\in P_x^\vee$
at each special point $x\in \ul C$. Recall that
Equation~\eqref{veta1veta2diffeq} characterizing $u_q$ implies $\langle u_q,
\rho_x\rangle=0$. To fix the sign of $u_q$ we use the convention that $\chi_1$
in the defining equation is the generization map to $\eta$. Similarly, for each
marked point $p$, it holds $\langle u_p, \rho_p \rangle=0$ by definition of
$u_p$. We now deduce a number of necessary conditions for a logarithmic
enhancement of a transverse stable map to exist.

\begin{proposition}
\label{Prop: type of transverse stable log map}
Let $f:C\arr X$ be a logarithmic enhancement of a transverse stable map $\ul
f:\ul C\arr \ul X_0$. Let $\eta\in\ul C$ be a generic point and $x\in\cl(\eta)$.
If $\ulf(x)\in X_0^{[1]}$ then following Definition~\ref{Def: Discrete invariants
transverse case} write $m_1=m_{\eta,x}$, $m_2=m'_{\eta,x}$, $n_1=n_{\eta,x}$,
$n_2=n'_{\eta,x}$, $\rho_x\in P_x$ and $w_1=w_{\eta,x}$.\\[1ex]
\noindent
I) {\bf (Node)} If $x=q$ is a constrained nodal point of $\ul C$, then the second generic
point $\eta'$ of $C$ with $x\in\cl(\eta')$ maps to a different irreducible
component of $X_0$ than $\eta$. Moreover, with $w_2=w_{\eta',x}$ the following
holds:
\begin{enumerate}
\item
$\displaystyle u_q= \frac{1}{\ind_q}\cdot(w_1 n_2-w_2 n_1)$.
\item
$u_q(m_1)=w_1$, $u_q(m_2)=-w_2$.
\item
$\mu_1 w_2=\mu_2 w_1$.
\item
The integral length of $u_q$ equals $\displaystyle \ell(u_q)= \frac{\mu_2
w_1\lambda(q)}{\ind_q} = \frac{w_1}{\mu_1}\ell(\rho_q)$.
\end{enumerate}
If $x=q$ is a free node then $u_q=0$.\\[2ex]
\noindent
II) {\bf (Marked point)} If $x$ is a smooth point of $\ul C$, then $\ulf(x)$ is contained in only
one irreducible component of $X_0$. Moreover, if $x=p$ is a marked point then
$u_p=0$ in the free case, while in the constrained case the following holds.
\begin{enumerate}
\item
$w_1$ is a multiple of $\ind_p$.
\item
$\displaystyle u_p= \frac{w_1}{\ind_p} n_2$.
\end{enumerate}
\end{proposition}

\begin{proof}
\emph{Setup for  (I)}. Let $C$ be defined over the log point
$S=\Spec(Q\arr \kk)$. For any generic point $\eta\in\ul C$, there is a
commutative square
\[\begin{CD}
\NN\simeq P_\eta= \ol\cM_{X_0,\ulf(\eta)}@>{\ol f^\flat_\eta}>>\ol\cM_{C,\eta}\\
@AAA @AAA\\
\NN\simeq\ol\cM_{B,b_0} @>>> Q.
\end{CD}\]

\emph{Free node.} In the case of a free
node, both generic points $\eta,\eta'\in \ul C$ containing $q$ in their closure
map to the same irreducible component of $X_0$. Thus $u_q=0$ by the defining
equation~\eqref{veta1veta2diffeq}.

\emph{Image components of constrained node.} Let now $x=q$ be a constrained node. Since the generization map
$\chi_\eta:P_q\arr P_\eta$ is a localization of fine monoids there exists $m\in
P_q\setminus\{0\}$ with $\chi_\eta(m)=0$. Then also $\ol f^\flat_q(m)$ is a
non-zero element in $\ol\cM_{C,q}$ with vanishing generization at $\eta$. But
$\ol\cM_C$ has no local section with isolated support at $q$. Hence
$\chi_{\eta'}(m)\neq0$, which implies that the two branches of $C$ at $q$
map to different irreducible components of $X_0$.

\emph{Computations for a constrained node.} (1) follows from (2) by pairing both sides with $m_1$, $m_2$ since these
elements generate $P_q\otimes_\ZZ\QQ$. We now prove (2). Since $u_q$ is
preserved under base-change, we may assume $C$ is defined over the standard log
point $\Spec(\NN\to\kk)$. Then $\ocM_{C,q}\simeq S_e$ for some
$e\in\NN\setminus\{0\}$ with $S_e$ the submonoid of $\ZZ^2$ generated by
$(e,0),(0,e),(1,1)$, see e.g.\ \cite[\S1.3]{GS}. The generator $1\in\NN$ of the
standard log point maps to $(1,1)$, while a chart at $q$ maps $(e,0)$ to a
function restricting to a coordinate on one of the two branches of $C$, say on
$\cl(\eta)$, while vanishing on the other. Similarly, $(0,e)$ restricts to a
coordinate on $\cl(\eta')$. By transversality we conclude
\[
\ol f_q^\flat(m_1)=w_1\cdot(e,0),\quad \ol f_q^\flat(m_2)=w_2\cdot(0,e).
\]
Equation~\eqref{veta1veta2diffeq} defining $u_q$ says
\begin{equation}
\label{Eq: GS1.8}
\chi_2\circ\ol f_q^\flat-\chi_1\circ \ol f_q^\flat= u_q\cdot e,
\end{equation}
with $\chi_i: S_e\to\NN$ the generization maps. With our presentation, $\chi_1$
and $\chi_2$ are induced by the projections $S_e\subset\ZZ^2\to\ZZ$ to the
second and first factors, respectively. Hence
\begin{eqnarray*}
\big(\chi_2\circ\ol f_q^\flat-\chi_1\circ \ol f_q^\flat\big) (m_1)&=& w_1\cdot e\\
\big(\chi_2\circ\ol f_q^\flat-\chi_1\circ \ol f_q^\flat\big) (m_2)&=& -w_2\cdot e,
\end{eqnarray*}
showing~(2).

(3) is obtained by evaluating (1) on $\rho_q$:
\[
0=\ \ind_q\cdot \langle u_q, \rho_q\rangle
\ =\ w_1\langle n_2,\rho_q\rangle - w_2\langle n_1,\rho_q\rangle
\ =\ w_1\mu_2-w_2\mu_1.
\]

For (4) observe from (1) that $\ind_q\cdot u_q$ is the vector connecting the
extremal elements $w_2n_1$ and $w_1 n_2$ of $P_q^\vee$. Thus $\ind_q\cdot \ell(u_q)$ equals
the integral length of $\rho_q^{-1}(h)$ for $h=\langle w_1 n_2,\rho_q\rangle = \mu_2
w_1= \mu_1 w_2= \langle w_2 n_1,\rho_q\rangle$. This length equals $h\cdot \lambda(q)$,
yielding the stated formula. This finishes the proof of (I).
\smallskip

\emph{Marked point.} Turning to (II), let $x\in \ul C$ be a smooth point with
$\ulf(x)\in X_0^{[1]}$ and again assume without restriction $C$ is defined over
the standard log point. If $s_u\in \cM_{X,\ulf(x)}$ is a lift of $m_1$, then by
transversality, $f_x^\flat(s_u)\in \cM_{C,x}$ maps under the structure
homomorphism $\cM_{C,x}\arr \cO_{C,x}$ to $z^{w_1}$, with $z$ a local coordinate
of $\ul C$ at $x$. Thus $x=p$ is a marked point, $\ocM_{C,x}=\NN^2$ and
\[
\ol f_p^\flat: P_p\longrightarrow \NN^2
\]
maps $m_1$ to $(0,w_1)$. Here we are taking the morphism $C\arr \Spec(\NN\to\kk)$
to be defined by $\NN\to\NN^2$, $1\mapsto (1,0)$. Moreover, by compatibility of
$f_p^\flat$ with the morphism of standard log points that $C$ and $X_0$ are
defined over, $\ol {f_p}^\flat(\rho_p)=(b,0)$ for some $b\in\NN\setminus\{0\}$.
Thus by Lemma~\ref{Lem: multiplicities versus rho},(2),
$\rho_p=\frac{\mu_1}{\ind_p} m_2$ spans an extremal ray of $P_p$. In particular,
$\ulf(p)$ is contained in only one irreducible component of $X_0$ and
$u_p(m_2)=0$. Thus
\[
u_p(m_1)=w_1 = \frac{w_1}{\ind_p} \langle n_2,m_1\rangle,\quad
u_p(m_2)=0= \frac{w_1}{\ind_p} \langle n_2,m_2\rangle.
\]
This shows (2), which implies (1) since $n_2$ is a primitive vector.

Finally, at a free marked point $p\in\ul C$, commutativity over the standard log point again readily implies $u_p=0$.
\end{proof}

\begin{remark}
\label{Rem: Conditions for reduced case}
If $X_0$ is reduced, then in Proposition~\ref{Prop: type of transverse
stable log map},(I) there is a well-defined contact order $w=w_1=w_2$ of $ \ul f$ with the double locus, and the formulas simplify to
\[
u_q=\frac{w}{\ind_q}\big(n_2-n_1\big),\quad \ell(u_q)= w\ell(\rho_q).
\]
\end{remark}

\paragraph{\bf Transverse pre-logarithmic maps}
Summarizing the necessary conditions of Proposition~\ref{Prop: type of
transverse stable log map}, we are led to the following definition.

\begin{definition}
\label{transprelogdef}
Let $X\arr B$ be as above, and let $\ulf:\underline{C}/
\Spec\kk\arr\underline{X}_0$ be a transverse map. We say $\ulf$ is a
\emph{transverse pre-logarithmic map} if any $x\in \ul C$ with $\ulf(x)\in
X_0^{[1]}$ is a special point and if in the notation of Proposition~\ref{Prop:
type of transverse stable log map} the following holds.\\[1ex] 
(I)  {\bf (Constrained node)}~If $x=q$ is a
constrained node then the two branches of $C$ at $q$ map to different
irreducible components of $X_0$. In addition, $\mu_1 w_2=\mu_2 w_1$ and the \emph{reduced
branching order}
\begin{equation}
\label{prelogconditions}
\ol w_q:= \frac{w_i}{\mu_i}\ell(\rho_q), \quad i=1,2
\end{equation}
is an integer.\\[1ex]
(II) {\bf (Constrained marking)}~If $x=p$ is a constrained marking then $\ul f(x)$ is a smooth point of
$X_0$ and $w_1/\ind_p\in\NN$.
\end{definition}

Note that if a logarithmic enhancement of $\ul f$ exists, then by
Proposition~\ref{Prop: type of transverse stable log map} the reduced branching
order $\ol w_q$ agrees with $\ell(u_q)$. Note also that in the case of
reduced $X_0$, we have $\ell(\rho_q)=1$ and all $\mu_i=1$, and hence $\ol
w_q=w_1=w_2$.

\begin{definition}[Base order]\label{Def:base order}
For a transverse pre-logarithmic map $\ulf:\underline{C}/
\Spec\kk\arr\underline{X}_0$ define its \em{base order} $b\in\NN$ to be
the least common multiple of the following natural numbers: (1)~all
multiplicities of irreducible components of $X_0$ intersecting $\ulf(\ul C)$ and
(2)~for each constrained node $q\in\ul C$ the quotient $\mu_1 w_2/\gcd(\ind_q,
\mu_1 w_2)$, notation as in Proposition~\ref{Prop: type of transverse stable log
map}.
\end{definition}

\begin{theorem}
\label{Thm: BasicEnhancementsNecCond}
Let $X\arr B$ be as above, and let $\ulf:\underline{C}/
\Spec\kk\arr\underline{X}_0$ be a transverse map. Suppose that there is
an enhancement of $\ulf$ to a basic stable logarithmic map $f:C/S\arr
X/B$. Then 
\begin{enumerate}
\item
$\ulf$ is a 
transverse pre-logarithmic map.

\item
The combinatorial type of $f$ is uniquely determined up
to possibly a number of marked points $p$ with $u_p=0$, and
the basic monoid $Q$ is
\[
Q=\NN\oplus \bigoplus_{q \text{\ \rm a free node}} \NN.
\]

\item
The map $S=\Spec(Q\arr\kk)\arr B$ induces the map
$\overline\cM_{b_0}=\NN\arr Q$ given by $1\mapsto (b,0,\ldots,0)$, where
the integer $b\in\NN$ is the base order of $\ulf$.
\end{enumerate}
\end{theorem}

\begin{proof}
(1) and (2) follow readily from Proposition~\ref{Prop: type of transverse stable
log map}. For (3), recall that the basic monoid $Q$ is dual to the monoid
$Q^{\vee} \subset Q_{\RR}^{\vee}$, the latter being the moduli space of tropical
maps $h:\Gamma\arr \Sigma(X)$ of the given combinatorial type, and
$Q^{\vee}$ consists of those tropical maps whose edge lengths are integral and
whose vertices map to integral points of $\Sigma(X)$.

If $\eta$ is a generic point of $\ulC$, denote by $\mu_{\eta}$ the multiplicity
of the irreducible component of $(X_0)_{\red}$ in $X_0$ containing $\ulf(\eta)$.
Thus the induced map $\NN\arr P_{\eta}\simeq\NN$ coming from the structure map
$X\arr B$ is multiplication by $\mu_{\eta}$. Write $\rho:\Sigma(X)\arr\Sigma(B)$
for the tropicalization of $X\arr B$. The restriction of $\rho$ to the ray
$\Hom(P_\eta,\RR_{\ge0})$ of $\Sigma(X)$ corresponding to the irreducible
component of $X_0$ containing $\ulf(\eta)$ is multiplication by $\mu_{\eta}$.
Thus given a tropical map $h:\Gamma\arr\Sigma(X)$ with vertex $v_\eta$ for
$\eta\in\ul C$ and $b$ the image of $\rho\circ h$ in $\Sigma(B)$, we see that
$h(v_{\eta})$ is integral if and only if $\mu_{\eta}|b$.

The edges of $\Gamma$ corresponding to free nodes have arbitrary length
independent of $\mu$. But an edge corresponding to a constrained node $q$
must have length 
\begin{equation}
\label{Eqn: e_q}
e_q=b\frac{\lambda(q)}{\ell(u_q)}=b\frac{\ind_q}{\mu_1 w_2}.
\end{equation}
This must also be integral for $h$ to represent a point in $Q^{\vee}$.
Thus the map $\Sigma(S)\arr \Sigma(B)$ must be given by
$(\alpha, (\alpha_q)_q)\mapsto b\alpha$ where $b$ is as given in the
statement of the theorem. Dually, we obtain the stated description of
the map $S\arr B$.
\end{proof}

%---------------------------------------------------------------------------------
\subsection{Existence and count of enhancements of transverse pre-logarithmic maps}
\label{Par: Torically transverse counts}

We now turn to count the number of logarithmic enhancements of a transverse
stable map $\ul f:\ul C\arr \ul X_0$. Denote by $\cM:= \ul f^*\cM_{X_0}$ the
pull-back log structure on $\ul C$ and by $\cM^{\text{Zar}}$ the corresponding
sheaf of monoids in the Zariski topology, noting that the log structure on $X_0$
is assumed to be defined in the Zariski topology.
\smallskip

\paragraph{\bf The torsor of roots} The count of logarithmic enhancements
involves a torsor $\cF$ under a sheaf of finite cyclic groups $\cG$ on a finite
topological space encoding compatible choices of roots of elements occurring in
the construction of logarithmic enhancements. The following discussion
is trivial if $X_0$ is reduced and can be skipped by the reader only interested
in this case. Given a transverse map $\ulf:\ul C/\Spec\kk\arr \ul X_0$, the
finite topological space consists of the set of constrained nodes $q\in\ul C$
and generic points $\eta\in\ul C$. As basis for the topology we take the sets
$U_\eta= \{\eta\}\cup \big\{q\in\cl(\eta)\}$ and $U_q=\{q\}$ (which is opposite
to the topology as a subset of $\ul C$). Let $\rho\in \Gamma(\ul C,\cM)$ be the
preimage of a generator $\rho_0$ of $\cM_{B,b_0}$, that we assume fixed in this
subsection. The stalks at a constrained node $q\in\ul C$ and at a generic point
$\eta\in\ul C$ are various roots of the germs $\rho_x$ of $\rho$:
\begin{eqnarray*}
\cF_q&=&\big\{ \sigma_q\in\cM_q^{\text{Zar}}\,\big|\, \sigma_q^{\ell(\rho_q)}=\rho_q\big\},\\
\cF_\eta&=&\big\{ \sigma_\eta\in\cM_\eta^{\text{Zar}}\, \big|\, \sigma_\eta^{\mu_\eta}
=\rho_\eta\big\}.
\end{eqnarray*}
We note that any of these sets may be empty, as Example~\ref{example:no-lift}
below shows. In such case we do not define a sheaf $\cF$ and declare
$|\Gamma(\cF)| = \emptyset$ in what follows. Otherwise we define the sheaf $\cF$
as follows. For $q\in\cl(\eta)$, a choice $\sigma_\eta$ with
$\sigma_\eta^{\mu_\eta}= \rho_\eta$ determines a unique
$\sigma_{\eta,q}\in\cF_q$ with restriction to $\eta$ equal to
$\sigma_\eta^{\mu_\eta/\ell(\rho_q)}$. Note that by Lemma~\ref{Lem:
multiplicities versus rho},(1) we have $\mu_\eta/ \ell(\rho_q) \in\NN$. We
define the generization map $\cF_\eta\to\cF_q$ by mapping $\sigma_\eta$ to
$\sigma_{\eta,q}$. Observe that a different choice of $\rho$ leads to an
isomorphic sheaf $\cF$. 

Replacing the elements $\rho_q$ and $\rho_\eta$ in the definition of $\cF$ by
the element $1$, we obtain a sheaf $\cG$ of abelian groups, for which $\cF$ is
evidently a torsor.
\smallskip

\paragraph{\bf Global sections of $\cG $ and $\cF$} General theory \cite[Tag
03AH]{stacks-project}, or direct computation, implies that the \emph{set} of
global sections $\Gamma(\cF)$ is a \emph{pseudo-torsor} for the \emph{group}
$G:= \Gamma(\cG)$. Here $G$ is computed as the kernel of the sheaf-axiom
homomorphism
\begin{equation}
\label{Eqn: compatibility group}
\partial: \prod_{\eta\in\ul C}\ZZ/\mu_\eta\longrightarrow \prod_{q\in\ul C}\ZZ/\ell(\rho_q),\quad
\partial(\big(\zeta_\eta\big)_\eta) :=  \Big(\zeta_{\eta(q)}^{\mu_{\eta(q)}/\ell(\rho_q)}\cdot \zeta_{\eta'(q)}^{-\mu_{\eta'(q)}/\ell(\rho_q)}\Big)_q.
\end{equation}
Here $\eta(q),\eta'(q)$ are the generic points of the two adjacent branches of a
constrained node $q\in\ul C$, viewed in the \'etale topology. The notation
implies a chosen order of branches. Multiplication of $\sigma_q$ by $\zeta_q$
and of $\sigma_\eta$ by $\zeta_\eta$ describes the natural action of $G= \Gamma(\cG)$ on
$\Gamma(\cF)$. Note that if $X_0$ is reduced then all $\mu_i=1$ and $G$ is the trivial group.

\begin{lemma}
\label{Lem: G-action}
If $\Gamma(\cF)\neq\emptyset$ the action of $G$ on $\Gamma(\cF)$ is simply
transitive. In particular, it then holds $\big|\Gamma(\cF)\big|= \big|G\big|$.
If the dual intersection graph of $C$ is a tree or if $X_0$ is reduced
then $\Gamma(\cF)\neq\emptyset$.
\end{lemma}

\begin{proof}
Simple transitivity is the fact that $\Gamma(\cF)$ is a pseudo-torsor for $G$.

If $X_0$ is reduced then $\mu_\eta=1$ for all $\eta$ and $\Gamma(\cF)=\prod_q
\cF_q$ is non-empty. If $C$ is rational we can construct a section by inductive
extension over the irreducible components. Indeed, if $\sigma_q\in\cF_q$ and
$\eta$ is the generic point of the next irreducible component, we can define
$\sigma_\eta$ as any $\mu_\eta/\ell(\rho_q)$-th root of the restriction of
$\sigma_q$ to $\eta$. By the definition of $\cF$ this choice then also defines
$\sigma_{q'}$ for all other $q'\in\cl(\eta)$.
\end{proof}

\begin{example}\label{example:no-lift}
Here is a simple example with $\Gamma(\cF)=\emptyset$, in fact
$\cF_\eta=\emptyset$ for the unique point $\eta$ in our space. Let $\ul X\arr \ul
B=\AA^1$ be an elliptically fibred surface with $\ul X_0\subset \ul X$ a
$b$-fold multiple fibre with smooth reduction. Endow $\ul X$ and $\ul B$ with
the divisorial log structures for the divisors $\ul X_0\subset\ul X$ and
$\{0\}\subset \ul B$. Then the generator $\ol\rho_0\in\ol\cM_{B,0}$ maps to $b$
times the generator $\ol\sigma\in\Gamma(\ul X_0,\ol \cM_{X_0})=\NN$. The
preimage of $\ol\sigma$ under $\cM_{X_0^\red}\arr \ol\cM_{X_0^\red}$ is the
torsor with associated line bundle the conormal bundle $N^{\vee}_{X_0^\red|X}$.
This conormal bundle is not trivial, but has order $b$ in $\pic(\ul X_0^\red)$.
Thus there exists no section $\sigma_\eta$ with $\sigma_\eta^b$ extending to a
global section $\rho$ of $\cM_{X_0}$ lifting $\ol\rho= b\cdot\ol\sigma$.
\end{example}

The following statement generalizes and gives a more structural proof of
\cite[Prop.~7.1]{NS06}, which treated a special case with reduced central
fibre.

\begin{theorem} 
\label{curveconstruction}
Suppose given $X\arr B$ as above, and let
\[
\ulf:(\underline{C},p_1,\ldots,p_n)/\Spec\kk\arr\underline{X}_0
\]
be a transverse pre-logarithmic map. Suppose further that the marked points
$\{p_i\}$ include all points of $\ulf^{-1}(X_0^{[1]})$ mapping to non-singular
points of $(X_0)_{\red}$. 

Then there exists an enhancement of $\ulf$ to a basic stable logarithmic
map if and only if $\Gamma(\cF)\neq\emptyset$, in which case the number of pairwise non-isomorphic enhancements is
\[
\frac{|G|}{b} \prod_q \ol w_q.
\]
Here $G$ is as in \eqref{Eqn: compatibility group}, the integer $b$ is the base
order (Definition~\ref{Def:base order}), and the product is taken over the
reduced branching orders \eqref{prelogconditions} at constrained nodes.

If $X_0$ is reduced there is no obstruction to the existence of an
enhancement and the count is $b^{-1}\prod_q \ol w_q$ and $\ol w_q=w_1=w_2$
either of the two contact orders in Definition~\ref{Def: Discrete invariants
transverse case} and Proposition~\ref{Prop: type of transverse stable log map}.
\end{theorem}

\begin{proof}
\textsc{Counting rigidified objects.}
We are going to count diagrams of the form
\begin{equation}
\label{Eqn: counting diagram}
\begin{CD}
C=(\ul C,\cM_C)@>{f}>>X_0=(\ul X_0,\cM_{X_0})\\
@V{\pi}VV @VV{p}V\\
\Spec(Q\arr \kk) @>>{g}> B=(\ul B,\cM_B),
\end{CD}
\end{equation}
with $p$ and $\ulf$ given by assumption and $g$ determined by $b$ as in
Theorem~\ref{Thm: BasicEnhancementsNecCond},(3) uniquely up to isomorphism. For the final
count we will divide out the $\ZZ/b$-action coming from the automorphisms
of $\Spec(Q\to\kk)\arr B$.

\textsc{Simplifying the base.} By Theorem~\ref{Thm: BasicEnhancementsNecCond} we have $Q=\NN\oplus\bigoplus_{\text{free
nodes}}\NN$ and the map $\NN=\ocM_{B,b_0}\arr Q$ is the inclusion of the first
factor multiplied by the base order $b$. Pulling back by any fixed sharp map
$Q\to\NN$ replaces the lower left corner by the standard log point
$O^\dagger=\Spec(\NN\arr \kk)$. To be explicit, we take $Q\to\NN$ to restrict to
the identity on each summand. Since this map $Q \arr \NN$ is surjective we do not introduce automorphisms or ramification. The universal property of basic objects guarantees that the number of liftings is not changed.

The composition $\ol\cM_{B,b_0}\arr Q\arr \NN$ is
then multiplication by $b$. We have now arrived at a counting problem over a
standard log point. Note also that the given data already determines \eqref{Eqn:
counting diagram} at the level of ghost monoids, that is, the data determines the sheaf $\ol\cM_C$, and maps $\ol
f^\flat: \ol\cM= \ul f^*\ol\cM_{X_0}\arr \ol\cM_C$ and $\ol\pi^\flat:
\ol\cM_{O^\dagger}\arr \ol\cM_C$, uniquely.

\textsc{Pulling back the target monoid.}
Pull-back yields the two log structures $\cM=\ul f^*\cM_{X_0}$ and
$\pi^*\cM_{O^\dagger}$ on $\ul C$. Recall our choice of generator $\rho_0$ of
$\cM_{B, b_0}$ and its pull-back $\rho\in \Gamma(\ul C,\cM)$, introduced earlier in \S\ref{Par: Torically transverse counts}. For later use let
also $\tau_0\in \cM_{O^\dagger,0}$ be a generator with $g^\flat(\rho_0)=
\tau_0^b$. Then for any log smooth structure $\cM_C$ on $\ul C$ over $O^\dagger$
we have a distinguished section $\tau=\pi^\flat(\tau_0)$.

\textsc{Simplifying the target monoid.}
Now define $\cM'_C$ as the fine monoid sheaf given by pushout of these two
monoid sheaves over $\ul\pi^*\ul g^*\cM_B$:
\[
\cM'_C= \ul f^*\cM_{X_0}\oplus^{\text{fine}}_{\ul\pi^* \ul g^*\cM_B} \pi^*\cM_{O^\dagger}.
\]
Since $\ocM_{B,b_0} = \NN$, by \cite[(4.4)(ii)]{KKato} $X_0\arr B$ is an integral morphism.
Hence the pushout $\cM'_C$ in the category
of fine monoid sheaves agrees with the ordinary pushout. In particular, the structure
morphisms of $X_0$ and $O^\dagger$ define a structure morphism $\alpha'_C: \cM'_C\to
\cO_C$. 

\textsc{Restating the counting problem.} Classifying diagrams~\eqref{Eqn: counting diagram} amounts to finding an
fs log structure $\cM_C$ on $\ul C$ together with a morphism of monoid sheaves
\[
\phi:\cM'_C\longrightarrow \cM_C
\]
compatible with $f^\sharp$ and such that the composition $\pi^*\cM_{O^\dagger}\arr \cM_C$  of $\phi$ with
$\pi^*\cM_{O^\dagger}\arr \cM'_C$ is log smooth. 

We will soon see that $\bar\phi:(\ocM'_C)^\gp\longrightarrow \ocM_C^\gp$
necessarily decomposes into the quotient by some finite torsion part and the
inclusion of a finite index subgroup. 

Lifting the quotient morphism to $\cM'_C$ leads to the factor $|G|$, while the
finite extension of the resulting log structure to $\cM_C$ receives a
contribution by the reduced branching order $\ol w_q$ from each constrained
node. Note that $\ol w_q = \ell(u_q)$ by Proposition~\ref{Prop: type of
transverse stable log map},(I)(4); it is in this form that it appears in the
proof.

\textsc{The ghost kernel.} To understand the torsion part to be divided out, note
that $(\ol{\cM'_{C,x}})^\gp$ at $x\in\ul C$ equals $P_x^\gp\oplus_\ZZ \ZZ$ with
$1\in\ZZ$ mapping to $\ol\rho_x\in P_x^\gp$ and to the base order $b\in\ZZ$,
respectively. Since $\ell(\ol\rho_x)$ divides the multiplicities of some
irreducible components of $\ul X_0$, Theorem~\ref{Thm:
BasicEnhancementsNecCond},(3) implies $b/\ell(\ol\rho_x)\in\NN$. If $\ol\rho_x$
has integral length $\ell(\ol\rho_x)>1$, then $\big(\ol\rho_x/
\ell(\ol\rho_x),-b/ \ell(\ol\rho_x)\big)$ is a generator of the torsion subgroup
$((\ol{\cM'_{C,x}})^\gp)_\tor$, which has order $\ell(\ol\rho_x)$. This element
has to be in the kernel of the map to the torsion-free monoid $\ol\cM_C$. 

\textsc{The $|G|$ embodiments of the ghost image.} The interesting
fact is that the lift of $\big(\ol\rho_x/\ell(\ol\rho_x),-b/
\ell(\ol\rho_x)\big)$ to $(\cM'_{C,x})^\gp$ is only unique up to an
$\ell(\ol\rho_x)$-torsion element in $\cO_{\ul C,x}^\times$, that is, up to an
$\ell(\ol\rho_x)$-th root of unity $\zeta_x\in\kk^\times$. Explicitly, the lift
is equivalent to a choice $\sigma_x\in\cM_x$ with
$\sigma_x^{\ell(\ol\rho_x)}=\rho_x$ by taking the torsion subsheaf in
$(\cM'_C)^\gp$ generated by $(\sigma_x,\tau_x^{-b/\ell(\ol\rho_x)}\big)$. The
quotient by this subsheaf means that we upgrade the relation
$f^\flat(\rho_x)=\tau_x^b$ coming from the commutativity of \eqref{Eqn: counting
diagram} to $f^\flat(\sigma_x)= \tau_x^{b/\ell(\ol\rho_x)}$. 

To define this
quotient of the monoid $\cM'_C$ globally amounts to choosing the roots $\sigma_x$ of
$\rho_x$ compatibly with the generization maps, leading to a global section of
the sheaf $\cF$ introduced directly before the statement of the theorem. For
this statement note that for $x=\eta$ a generic point, $\ell(\ol\rho_\eta)$
equals the multiplicity $\mu_\eta$ of the irreducible component of $\ul X_0$
containing $\ul f(\eta)$. 

\textsc{The quotient is a logarithmic structure.} Assume now
$\sigma\in\Gamma(\cF)$ has been chosen and denote by $\cM''_C$ the quotient of
$\cM'_C$ by the corresponding torsion subgroup of $(\cM'_C)^\gp$. Since
$\alpha'_C(\sigma_x)=\alpha'_C(\tau_x)=0$, the homomorphism $\alpha'_C$ descends
to the quotient, thus defining a structure homomorphism $\alpha''_C: \cM''_C\to
\cO_C$.

\textsc{The log structure $\cM_C$ is determined at smooth points.} Note
that the map $(\pi^*\cM_{O^\dagger})_\eta\to\cM''_{C,\eta}$ is an
isomorphism and hence we must have $\cM_{C,\eta}=\cM''_{C,\eta}$. The log
structure $\cM_C$ is then also defined at each marked point $p\in\ul C$ by
adding a generator of the maximal ideal in $\cO_{\ul C,p}$ as an additional
generator to $\cM_{C,p}$. It is also clear that $\cM''_{C,p}\arr \cM_{C,p}$
exists and is determined by the corresponding map at $\eta$ and by $f^\sharp$.

\textsc{The log structure $\cM_C$ is determined at free nodes.} At a free node $q$
we have $\ul f(q)\in (X_0)_\reg$ and hence there is a unique specialization
$\sigma_q\in \cM_q$ of $\sigma_\eta$ for the two generic points $\eta\in \ul C$
with $q\in\cl(\eta)$. The log structure $\cM_C$ on $C$ is then determined by
$f^\flat_\eta$ and by the universal log structure $\cM_C^\circ$ of $\ul C$ as
follows. Let $x,y\in \cO_{C,q}$ be coordinates of the two branches of $\ul C$ at
$q$ in the \'etale topology. Then there exist unique lifts
$s_x,s_y\in\cM^\circ_{C,q}$ such that $s_x\cdot s_y$ is the pull-back of a
generator $\epsilon_q$ of the $q$-th factor in the universal base log structure
$\Spec \big(\bigoplus_{\text{nodes of \ul C}} \NN\arr \kk\big)$. Our choice of
pull-back $O^\dagger\arr \Spec\big(Q\arr \kk)$ turns $\epsilon_q$ into
$\lambda\tau_0^{e_q}$ for some $\lambda\in\kk^\times$ and $e_q\in\NN$ determined
by basicness as in \eqref{Eqn: e_q}. Thus $\cM_{C,q}$ is generated by $s_x,s_y$
and $\tau_q$ with single relation $s_x\cdot s_y=\lambda\tau_q^{e_q}$ and mapping
to $x,y$ and $0$ under the structure homomorphism, respectively. The morphism
$f^\flat:\cM_q\arr \cM_{C,q}$ factors over $\pi^*\cM_{O^\dagger}$ and is
therefore completely determined by $f^\flat_\eta(\sigma_q)=\tau_q^{b/\mu_\eta}$.

\textsc{Constrained nodes: Study of the image log structure $\cM''_C$.}
It remains to extend $\cM''_C$ to the correct log structure at each constrained
node $q\in\ul C$. On the level of ghost sheaves we have the following situation,
where we include the above description of the kernel for completeness.

\begin{proposition}
\label{Prop: Index of the image}
The homomorphism of abelian groups
\[
P_q^\gp\oplus\ZZ\longrightarrow \ol\cM_{C,q}^\gp,\quad
(m,k)\longmapsto \ol f^\flat_q(m)+ k\cdot\ol\tau_q,
\]
has kernel generated by $\big(\ol\rho_q/\ell(\ol\rho_q),-b/\ell(\ol\rho_q)\big)$ and
cokernel a cyclic {group} of order $\ell(u_q)$.
\end{proposition}

\begin{proof}
The kernel is described in the discussion above. Indeed, if $(m,k)\in
P_q^\gp\oplus\ZZ$ lies in the kernel then $\ol f^\flat_q(m)\in
\ZZ\cdot\ol\tau_q$. Because $\ol f^\flat_q$ is injective and $\ol
f^\flat_q(\ol\rho_q)=\ol\tau_q^b$ we conclude that $(m,k)$ is proportional to
$(\ol\rho_q, -b)$. The stated element is a primitive element of this
one-dimensional subspace.

For the determination of the cokernel observe that the composition
\[
P_q^\gp\stackrel{\ol f^\flat_q}{\longrightarrow} \ol\cM_{C,q}^\gp
\longrightarrow \ol\cM_{C,q}^\gp/\ZZ\ol\tau_q\simeq \ZZ
\]
equals $u_q$ up to sign. Indeed, the quotient by $\ZZ\ol\tau_q$ maps the two
generators of extremal rays of $\ol\cM_{C,q}$ to $\pm 1\in\ZZ$. Hence $m_i\in
P_q^\gp$ maps to $\pm w_i$, which by Proposition~\ref{Prop: type of transverse
stable log map}, I,(2) agrees with $\pm u_q(m_i)$. The order of the cokernel now
agrees with the greatest common divisor of the components of $u_q$, that is,
with $\ell(u_q)$.
\end{proof}

Once again we follow \cite[\S1.3]{GS} and denote by $S_e\subset \ZZ^2$ the submonoid generated by $(e,0),(1,1),(0,e)$, for $e\in\NN\setminus\{0\}$. Up to a choice of ordering of extremal rays there is a canonical isomorphism
\begin{equation}
\label{Eqn: ocM_q}
\ol\cM_{C,q} \stackrel{\simeq}{\longrightarrow} S_{e_q}.
\end{equation}
Using Proposition~\ref{Prop: Index of the image} we can now determine the saturation of $\ol\cM''_{C,q}$. For readability we write $\ell=\ell(u_q)$.

\begin{corollary}
\label{Cor: Saturation of M''}
Using the description \eqref{Eqn: ocM_q}, the saturation of $\ol\cM''_{C,q}$ equals $S_{\ell(u_q)e_q}\subset S_{e_q}$.
\end{corollary} 

\begin{proof}
By construction of $\cM''_C$, the image of the homomorphism in Proposition~\ref{Prop: Index of the image} equals $(\ol\cM''_{C,q})^\gp$. In the notation of \eqref{Eqn: ocM_q},
the statement now follows from the fact that by the proposition, the image has index
$\ell(u_q)$ in $\ol\cM_{C,q}^\gp$ and $(e_q,0)\in S_{e_q}$. Hence $(\ell e_q,0)\in (\ol\cM''_{C,q})^\gp$, which together with $(1,1)\in (\ol\cM''_{C,q})^\gp$ generates $S_{\ell e_q}\subset S_{e_q}$.

The saturation is then computed by taking all integral points in the real cone
in $\ol\cM_{C,q}^\gp\otimes_\ZZ \RR$ spanned by $\ol\cM''_{C,q}$.
\end{proof}

\textsc{Constrained nodes: Extending the log structure.}

In this step we extend the log structure to the saturation of $\cM''_C$, described in Corollary~\ref{Cor: Saturation of M''}.

\begin{lemma}
\label{Lem: Extension to the saturation}
The log structure $\alpha'':\cM''_C\to\cO_C$ extends uniquely to the saturation $(\cM''_C)^{\mathrm{sat}}$.
\end{lemma}

\begin{proof}
We continue to write $\ell=\ell(u_q)$.
The saturation can at most be non-trivial at a constrained node $q$. By
Corollary~\ref{Cor: Saturation of M''} we have an isomorphism
$(\ol\cM''_{C,q})^{\mathrm{sat}}\simeq S_{\ell e_q}$. The definition of the
weights $w_i$ implies $(w_1 e_q,0),(0,w_2 e_q)\in \ol\cM''_{C,q}$, for the
appropriate ordering of the branches of $C$ at $q$. As a sanity check, notice
that $w_i$ divides $\ell$ by Proposition~\ref{Prop: type of transverse
stable log map},(I)(2). Let $\beta_q: \ol\cM''_{C,q}\arr \cO_{C,q}$ be
the composition of a choice of splitting $\ol\cM''_{C,q}\arr
\cM''_{C,q}$ and the structure morphism $\cM''_{C,q}\arr \cO_{C,q}$. 
Then $\beta_q\big((w_1 e_q,0)\big)$ vanishes at $q$ 
to
order $w_1$ on one branch of $C$ and $\beta_q\big((0,w_2 e_q)\big)$ vanishes to
order $w_2$ on the other branch. Thus, \'etale locally there exist generators
$x,y\in \cO_{C,q}$ for the maximal ideal at $q$ with
\[
\beta_q\big((w_1 e_q,0)\big)=x^{w_1},\quad
\beta_q\big((0,w_2 e_q)\big)=y^{w_2}.
\]
Thus any extension $\beta_q^{\mathrm{sat}}$ of $\beta_q$ to a chart for
$(\cM_{C,q}'')^{\mathrm{sat}}$ has to fulfill
\begin{equation}\label{Eqn: chart for saturation}
\beta_q^{\mathrm{sat}}\big((\ell e_q,0)\big)=\zeta_1\cdot x^{\ell},\quad
\beta_q^{\mathrm{sat}}\big((0,\ell e_q)\big)=\zeta_2\cdot y^{\ell},
\end{equation}
with $\zeta_i\in\kk$, $\zeta_i^{w_1/\ell}=1$. On the other hand, the $\zeta_i$
are uniquely determined by compatibility of $\beta_q^{\mathrm{sat}}$ with
$\beta$ at the generic points of the two branches of $C$ at $q$ since $(\ell
e_q,0),(0,\ell e_q)\in\ol\cM_{C,q}$. Conversely, with this choice of the
$\zeta_i$, the equations \eqref{Eqn: chart for saturation} provide the requested
extension of $\alpha''$.
\end{proof}

Finally, we extend $\cM''_C$ to a log structure of a log smooth curve over the
standard log point. The situation is largely the same as with admissible covers,
{see, e.g., \cite[\S3]{Mochizuki}}.

\begin{lemma}
\label{lem: From cM'' to cM}
Up to isomorphism of log structures over the standard log point, there are
$\ell=\ell(u_q)$ pairwise non-isomorphic extensions $\alpha_q:\cM_{C,q}\to
\cO_{C,q}$ of the image log structure $\alpha''_q:\cM''_{C,q}\to\cO_{C,q}$ at
the constrained node $q$ to a log structure of a log smooth curve.
\end{lemma}

\begin{proof}
Let
\[
\beta_q^{\mathrm{sat}}: S_{\ell e_q}\longrightarrow \cO_{C,q}
\]
be a chart for the log structure $(\cM''_C)^{\mathrm{sat}}$ at $q$. The task is
to classify extensions to a chart $\tilde\beta_q: S_{e_q}\arr \cO_{C,q}$ up to
isomorphisms of induced log structures. Similar to the reasoning in
Lemma~\ref{Lem: Extension to the saturation}, in terms of coordinates
$x,y\in\cO_{C,q}$ with
\[
\beta_q^{\mathrm{sat}}\big((\ell e_q,0)\big)= x^{\ell},\quad
\beta_q^{\mathrm{sat}}\big((0,\ell e_q)\big)= y^{\ell},
\]
we have to define 
\begin{equation}
\label{Eqn: tilde beta}
\tilde\beta_q\big(( e_q,0)\big)= \zeta_1\cdot x,\quad
\tilde\beta_q\big((0, e_q)\big)= \zeta_2\cdot y,
\end{equation}
with $\zeta_i\in\kk$, $\zeta_i^\ell=1$. Dividing out isomorphisms amounts to
working modulo $\varphi\in\Hom(S_{e_q},\ZZ/\ell)$ with $\varphi\big(
(1,1)\big)=1$. In other words, we can change $\zeta_1,\zeta_2$ by
$\zeta\zeta_1,\zeta^{-1}\zeta_2$ for any $\ell$-th root of unity $\zeta$. This
leaves us with $\ell$ pairwise non-isomorphic extensions of the log structure at
$q$.
\end{proof}

\textsc{Counting non-rigidified lifts.} For the final count we need to divide out
the action of $\ZZ/b$ by composition with automorphisms of $\Spec(Q\to\kk)$ over
$B$. The stated count follows once we prove that this action is free. The action
changes $\tau_0$ to $\zeta\cdot\tau_0$ for $\zeta$ a $b$-th root of unity. For
this change to lead to an isomorphic log structure $\cM_C$ requires
$\zeta_1\zeta_2\in\kk^\times$ in \eqref{Eqn: tilde beta} to be unchanged at any
constrained node $q\in\ul C$. This shows $\zeta^{e_q}=1$ for all $q$. Similarly,
for the map $\cM_{X_0,\ulf(\eta)}\arr \cM_{C,\eta}$ to stay unchanged relative
$\cM_{B,b_0}\arr \cM_{O^\dagger,0}$ requires $f^\flat_\eta(\sigma_\eta)
=\tau_\eta^{b/\mu_\eta}$ to stay unchanged. Thus also $\zeta^{b/\mu_\eta}=1$ for
all generic points $\eta\in \ul C$. But by Theorem~\ref{Thm:
BasicEnhancementsNecCond} the base order $b$ is the smallest natural number with
all $e_q= b\cdot \ind_q/\mu_2 w_1$ and all $b/\mu_\eta$ integers. Thus the $e_q$
and $b/\mu_\eta$ have no common factor. This shows that $\zeta^{e_q}=1$ and
$\zeta^{b/\mu_\eta}=1$ for all $q, \eta$ implies $\zeta=1$. We conclude that the
action of $\ZZ/b$ is free as claimed.
\end{proof}

\begin{remark}
The obstruction to the existence of a logarithmic enhancement in Theorem
\ref{curveconstruction} can be interpreted geometrically as follows. 

Let $\bar \mu$ be a positive integer and $\ul{\tilde{B}}\arr {\ul B}$ be
the degree $d$ cyclic cover branched with ramification index $d$
over $b_0$. Let $\ul{\bar{X}}=\ul{X}\times_{\ul B}\ul{\tilde{B}}$, and let
$\ul{\tilde{X}}\arr \ul{\bar{X}}$ be the normalization, giving a family
$\ul{\tilde{X}}\arr \ul{\tilde{B}}$. It is a standard computation that
the inverse image of a multiplicity $\mu$ irreducible component of $\ul{X}_0$ in
$\ul{\tilde X}$ is a union of irreducible components of $\ul{\tilde X}_0$, each
with multiplicity $\mu/\gcd(\mu,d)$.

At the level of log schemes, in fact $\bar{X}$ carries a fine but not saturated
logarithmic structure via the description $\bar{X}=X\times_B \tilde{B}$ in the
category of fine log schemes, while $\tilde X$ carries an fs logarithmic
structure via the description $\tilde{X}=X\times_B\tilde{B}$ in the category of
fs logarithmic structures. Here $\tilde{B}$ carries the divisorial logarithmic
structure given by $\tilde b_0\in \tilde B$, the unique point mapping to $b_0$.

Similarly, the central fibres are related as follows. The map $\tilde B
\arr B$ induces a morphism on standard log points $\tilde b_0\arr
b_0$ induced by $\NN\arr \NN, 1\mapsto d$ for some integer $d$. Then $\bar X_0
=X_0\times_{b_0}\tilde b_0$ in the category of fine log schemes, and $\tilde
X_0=X_0\times_{b_0}\tilde b_0$ in the category of fs log schemes.

Given a transverse pre-logarithmic map $\ul{f}:\ul{C}/\Spec\kk\arr
\ul{X}$, take the integer $d$ above to be the positive integer $b$ given by Theorem~\ref{Thm:
BasicEnhancementsNecCond},(3). Then one checks readily that $\ul{f}$ has a
logarithmic enhancement if and only if there is a lift
$\ul{\tilde{f}}:\ul{C}\arr \ul{\tilde{X}}_0$ of $\ul{f}$. Indeed, if
$\ul{f}$ has a logarithmic enhancement $f:C/S\arr X_0$ with $S$ carrying
the basic log structure, the morphism $S\arr b_0$ factors through $\tilde
b_0$ by the description of Theorem~\ref{Thm: BasicEnhancementsNecCond}. Thus the
universal property of fibred product gives a morphism $\tilde f:C\arr
\tilde X_0$. Conversely, given a lift, it follows again from the definition of
$b$ in Theorem~\ref{Thm: BasicEnhancementsNecCond} that the multiplicity $\mu$
of any irreducible component of $\ul{X}_0$ meeting $\ul{f}(\ul{C})$ divides $b$.
So the multiplicity of any component of $\ul{\tilde{X}}_0$ meeting
$\ul{\tilde{f}}(\ul{C})$ is $1$ and by shrinking $\ul{X}_0$ we can assume that
$\ul{\tilde{X}}_0$ is reduced. One also checks that the reduced branching order
$\overline{w}_q$ associated to a node $q$ is the same for $\ul{f}$ and
$\ul{\tilde{f}}$, and thus $\ul{\tilde{f}}$ is still transverse pre-logarithmic.
By Lemma~\ref{Lem: G-action} and Theorem~\ref{curveconstruction},
$\ul{\tilde{f}}$ has a logarithmic enhancement, and then the composed morphism
$C\mapright{\tilde f}\tilde{X}_0 \mapright {} X_0$ gives the desired logarithmic
enhancement of $\ul{f}$.
\end{remark}

%=================================================================================
\section{Examples}
\label{examplessection}

We will now study explicit examples of the decomposition formula for a
logarithmically smooth morphism $X\to B$. We mostly use the traditional tropical
language of polyhedral complexes and metric graphs discussed in \S\ref{Subsect:
Traditional tropical geometry}.

%---------------------------------------------------------------------------------
\subsection{The classical case}
\label{sec:JunLi}

Suppose $X\arr B$ is a simple normal crossings degeneration with $X_0=Y_1\cup
Y_2$ a reduced union of two irreducible components, with $Y_1\cap Y_2=D$ a
smooth divisor in both $Y_1$ and $Y_2$. In this case, $\Sigma(X)=(\RR_{\ge
0})^2$ and the map $\Sigma(X)\arr\Sigma(B)=\RR_{\ge 0}$ is given by
$(x,y)\mapsto x+y$. Thus $\Delta(X)$ admits an affine-linear isomorphism with
the unit interval $[0,1]$, see Figure~\ref{Fig:DeltaX}.

\begin{figure}[htb]
\input{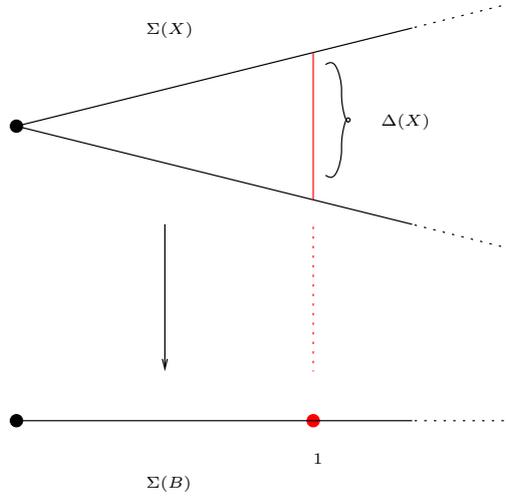}
\caption{The cones $\Sigma(X)$ and $\Sigma(B)$ and the interval $\Delta(X)$}
\label{Fig:DeltaX}
\end{figure}

\begin{proposition}
In the above situation, let $f:\Gamma\arr\Delta(X)$ be a decorated
tropical map. Then $f$ is rigid if and only if every vertex
$v$ of $\Gamma$ maps to the endpoints of $\Delta(X)$ and every 
edge of $\Gamma$ surjects onto $\Delta(X)$.
\end{proposition}

Note that necessarily
every leg of $\Gamma$ is contracted, as $\Delta(X)$ is compact.

\begin{proof}
First note that if an edge $E_q$ is contracted, then $u_q=0$ and
the length of the edge is arbitrary. By changing the length, one sees
$f$ is not rigid, see Figure~\ref{Fig:not-rigid} on the left.

Next, suppose $v$ is a vertex with $f(v)$ lying in the interior of $\Delta(X)$.
Identifying the latter with $[0,1]$, we can view $u_q\in \ZZ$ for any $q$. Let
$E_{q_1},\ldots,E_{q_r}$ be the edges of $\Gamma$ adjacent to $v$ with lengths
$\ell_1,\ldots,\ell_r$, oriented to point away from $v$. We can then write down
a family $f_t$ of tropical maps, $t$ a real number close to $0$, with $f=f_0$,
$f_t(v')=f(v')$ for any vertex $v'\not=v$, and $f_t(v)=f(v)+t$. In doing so, we
also need to modify the lengths of the edges $E_{q_i}$, as indicated in
Figure~\ref{Fig:not-rigid} on the right. Any unbounded edge attached to $v$ is
contracted to $f_t(v)$. So $f$ is not rigid. Thus if $f$ is rigid, we see that
all vertices of $\Gamma$ map to endpoints of $\Delta(X)$, and any compact edge
is not contracted, hence surjects onto $\Delta(X)$. The converse is clear.

\begin{figure}[htb]
\input{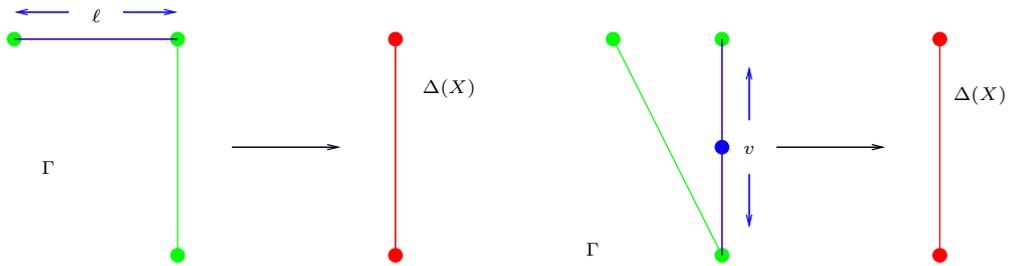}
\caption{A graph with a contracted bounded edge or an interior vertex is not rigid.}
\label{Fig:not-rigid}
\end{figure}
\end{proof}

A choice of decorated rigid tropical map in this situation is then exactly
what Jun Li terms an \emph{admissible triple} in \cite{Jun2}. Indeed, by
removing $f^{-1}(1/2)$ from $\Gamma$, one obtains two graphs (possibly
disconnected) $\Gamma_1, \Gamma_2$ with legs and what {Jun Li} terms
\emph{roots} (the half-edges mapping non-trivially to $\Delta(X)$). The weights
of a root, in Li's terminology, coincide with the absolute value of the
corresponding $u_q$. The set $I$ in the definition of admissible triple
indicates which labels occur for unbounded edges mapping to, say, $0\in
\Delta(X)$. An illustration is given in Figure~\ref{junlicase2}.

\begin{figure}[htb]
\input{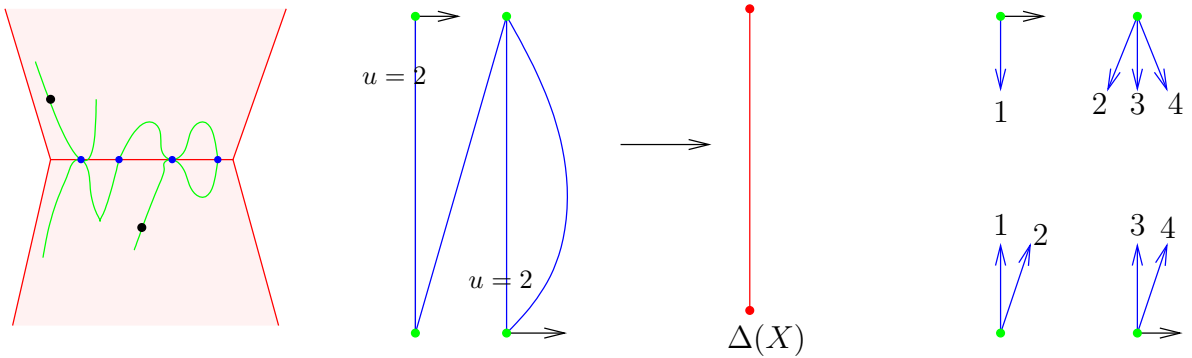}
\caption{A rigid tropical map is depicted with four edges and
two legs, the latter corresponding to marked points with contact order
$0$. The corresponding admissible triple of Jun Li is depicted on the
right, with roots corresponding to half-edges and legs corresponding
to the legs of the original graph. The half-edges marked $1$ and $3$ have $u=2$.  }
\label{junlicase2}
\end{figure}

We emphasize that our virtual decomposition of the moduli space of stable
logarithmic maps in terms of rigid tropical maps does not depend on
transversality. Already in this simple situation, the tropicalization of a basic
stable logarithmic map parameterizes a family of tropical maps with several
rigid limits, one for each facet of the basic monoid. The main result of this
paper refines the virtual counting problem in providing a count for each such
choice of rigid limit. This count applies even in more general situations where
the vertices of the tropical curve do not necessarily map to vertices of the
polyhedron associated to the target, as the next section shows. Note also that
this case has been carried out in detail and with somewhat different notation
after distribution of a first version of this paper in~\cite{KimLhoRuddat}.

%---------------------------------------------------------------------------------
\subsection{Rational curves in a pencil of cubics}
\label{sec:Cubic}
It is well-known that if one fixes $8$ general points in $\PP^2$, the pencil of
cubics passing through these $8$ points contains precisely $12$ nodal rational
curves. Blowing up 6 of these 8 points, we get a cubic surface we denote $X_1'
\subset \PP^3$, and the enumeration of $12$ nodal rational cubics translates to
the enumeration of $12$ nodal plane sections of $X_1'$ passing through the
remaining two points $p_1,p_2$.

We will give here a non-trivial demonstration of the decomposition formula by
degenerating the cubic surface to a normal crossings union {$H_1\cup H_2\cup
H_3$} of three blown-up planes.

%---------------------------------------------------------------------------------
\subsubsection{Degenerating a cubic to three planes} 

Using  coordinates
$x_0,\ldots,x_3$ on $\PP^3$, 
consider a smooth cubic surface $X_1'\subset \PP^3$ with equation
$$f_3(x_0,x_1,x_2,x_3)+x_1x_2x_3=0.$$ 
We then have a family $X'\arr B=\AA^1$ given by $X'\subset \AA^1\times
\PP^3$ defined by $tf_3+x_1x_2x_3=0$. The fibre $X'_0$ is the union of three planes $H'_1\cup
H'_2\cup H'_3$. Pick two sections $p_1,p_2:B\arr X'$ such that $p_i(0)\in
H'_i$. This can be achieved by choosing two appropriate points on the base locus $f_3(x_0,x_1,x_2,x_3)=x_1x_2x_3=0$. 

%---------------------------------------------------------------------------------
\subsubsection{Resolving to obtain a normal crossings family} 
The total space of $X'$ {is not a normal crossings family: it has} $9$ ordinary
double points over $t=0$, assuming $f_3$ is chosen generally: these are the
points of intersection of the singular lines $H_i'\cap H_j'$ with $f_3=0$. {One
manifestation is the fact that $H_i'$ are Weil divisors which are not Cartier.}
By blowing up $H'_1$ followed by $H'_2$, we resolve the ordinary double points.
We obtain a family $X\arr B$, which is normal crossings, hence logarithmically
smooth, in a neighbourhood of $t=0$, as depicted on the left in
Figure~\ref{cubiccentralfibre}. Denote by $H_i$ the proper transform of $H_i'$.

We identify $\Sigma(X)$  with $(\RR_{\ge 0})^3$, so that $\Delta(X)$
is identified with the standard simplex $\{(x_1,x_2,x_3)\,|\,x_1+x_2+x_3=1,
x_1,x_2,x_3\ge 0\}$, as depicted on the right in Figure~\ref{cubiccentralfibre}.

\begin{figure}[htb]
\input{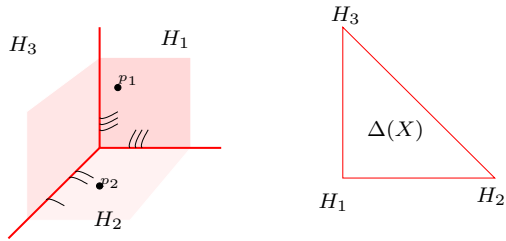}
\caption{The left-hand picture depicts $X_0$ as a union of three copies
of $\PP^2$, blown up at $6$, $3$ or $0$ points. The right-hand picture
depicts $\Delta(X)$.}
\label{cubiccentralfibre}
\end{figure}

%---------------------------------------------------------------------------------
\subsubsection{Limiting curves: triangles}
Since the limit of plane curves on $X'_t = X_t$ should be a plane curve on
$X'_0$, limiting curves on $X_0$ would map to plane sections of $X_0'$ through
$p_1,p_2$. This greatly limits the possible limiting curves --- in particular
the image in each of $H_i'$ is a line.

{\emph{General triangles do not occur.}}
It is easy to see that a plane section of $X'_0$ passing through $p_1, p_2$
whose proper transform in $X_0$ is a triangle of lines cannot be the image of a
stable logarithmic curve $C\arr X_0$ of genus zero. Indeed, there would be a
smooth point of $\ul C$ mapping to $(X_0)_\sing$, contradicting
Proposition~\ref{Prop: type of transverse stable log map},(II).

\emph{Triangles through double points.}
On the other hand, consider the total transform of a triangle in $X_0'$ passing
through $p_1$, $p_2$, and one of the $9$ ordinary double points of $X'$. 
The resulting curve will be a cycle of $4$ rational curves, one of the curves
being part of the exceptional set of the blowup of $H_1'$ and $H_2'$. We can
partially normalize this curve at the node contained in the smooth part of
$X_0$, getting a stable logarithmic curve of genus $0$. See
Figure~\ref{cubiccase1} for one such case.

\emph{Tropical picture.} We depict to the right the associated rigid tropical
curve. Here the lengths of each edge are $1$, and the contact data $u_q$ take
the values $(-1,1,0)$, $(0,-1,1)$ and $(1,0,-1)$. This accounts for $9$ curves.

\begin{figure}[htb]
\input{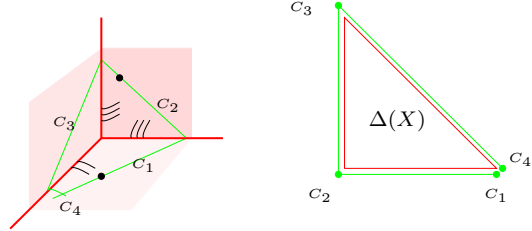}
\caption{Proper transform of a triangle through a double point. The curve is normalized where $C_1$ and $C_4$ meet.}
\label{cubiccase1}
\end{figure}

\emph{Logarithmic enhancement and logarithmic unobstructedness.}
Note that the above curves are transverse pre-logarithmic curves, and by
Theorem~\ref{curveconstruction}, each of these curves has precisely one basic
logarithmic enhancement. Since the curve is immersed it has no automorphisms.
One can use a natural absolute, rather than relative, obstruction theory to
define the virtual fundamental class, which is governed by the logarithmic normal
bundle. In this case each curve is unobstructed: since it is transverse with
contact order 1, the logarithmic normal bundle coincides with the usual normal
bundle. The normal bundle restricts to $\cO_{\PP^1},\cO_{\PP^1}(1),\cO_{\PP^1}(1)$,
and $\cO_{\PP^1}(-1)$ on the respective four components $C_1,C_2,C_3$ and $C_4$,
hence it is non-special. We note that this does not account for the incidence
condition that the marked points land at $p_i$. This can be arranged, for
instance, using \eqref{Eq:point-condition} in \ref{Sec:VFC-points}.

It follows that indeed each of these nine curves contributes
precisely once to the desired Gromov-Witten invariant.

%---------------------------------------------------------------------------------
\subsubsection{Limiting curves: the plane section through the origin}
The far more interesting case is when the plane section of $X_0'$ passes
through the triple point. Then one has a stable map from a union of
four projective lines, with the central component contracted to the 
triple point, see Figure~\ref{cubiccase2} on the left.

\begin{figure}[htb]
\input{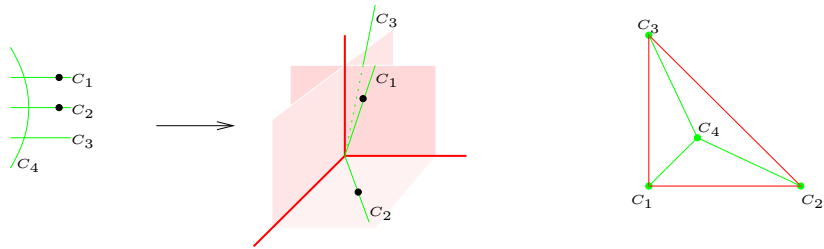}
\caption{A curve mapping to a plane section through the origin, and its tropicalization.}
\label{cubiccase2}
\end{figure}

There is in fact a one-parameter
family $\ul{W}$ of such stable maps, 
as the line in $H_3$ is unconstrained and can be chosen to be any element
in a pencil of lines. Only one member of this family lies in a plane, and we will see below that indeed only one member of the family admits a logarithmic enhancement.

\emph{Tropical picture.}
To understand the nature of such a logarithmic curve, we first analyze the
corresponding tropical map. The image of such a map will be as depicted in
Figure~\ref{cubiccase2} on the right, with the central vertex corresponding to
the contracted component landing somewhere in the interior of the triangle.
However, the tropical balancing condition must hold at this central vertex, by
\cite[Prop.~1.14]{GS}. From this one determines that the only possibility
for the values of $u_q$ are $(-2,1,1)$, $(1,1,-2)$ and $(1,-2,1)$, all lengths
are $1/3$, and the central vertex is $(1/3,1/3,1/3)$. The multiplicity
of this rigid tropical map $\Gamma$ acccording to Corollary~\ref{Cor:
decomposition of fM(cX_0,beta')} then is $m_{\Gamma}=3$.

%---------------------------------------------------------------------------------
\subsubsection{Logarithmic enhancement using a logarithmic modification}
We now show that only one of the stable maps in the family $\ul{S}$ has a
logarithmic enhancement. To do so, we use the techniques of
\S\ref{calculationalsection}, first refining $\Sigma(X)$ to obtain a logarithmic
modification of $X$. The subdivision visible in Figure~\ref{cubiccase2} gives a
refinement of $\Sigma(X)$, the central star subdivision of $\Sigma(X)$. This
corresponds to the ordinary blow-up $h:\tilde X\arr X$ at the triple
point of $X_0$. We may then identify logarithmic curves in $\tilde X$ and use
the induced morphism $\scrM(\tilde X/B) \arr \scrM(X/B)$.

\emph{Lifting the map to $\tilde X_0$.} The central fibre $\tilde X_0$ is now as
depicted in Figure~\ref{cubicblownup}. We then try to build a transverse
pre-logarithmic curve in $\tilde X$ lifting one of the stable maps of
Figure~\ref{cubiccase2}. Writing $C=C_1\cup C_2\cup C_3\cup C_4$, with $C_4$ the
central component, we map $C_1$ and $C_2$ to the lines $L_1$ and $L_2$
containing the preimages of $p_1$ and $p_2$, respectively, as depicted in
Figure~\ref{cubicblownup}, while $C_3$ maps to some line $L_3$ in $H_3$. On the
other hand, by \eqref{prelogconditions} in the definition of transverse
pre-logarithmic maps, $C_4$ must map to the exceptional $\PP^2=E$, which is of
multiplicity~$3$, in such a way that it is triply tangent to $\partial E$
precisely at the points of intersection with $L_i$, $i=1,2,3$. 

\emph{Uniqueness of liftable map.} We claim that there is precisely one such
map, necessarily with image containing a curve of degree $3$ in the exceptional
$\PP^2$, with image as depicted in Figure~\ref{cubicblownup}. The number of
transverse pre-logarithmic maps can be determined by considering linear series
as follows. The three contact points on $C_4 \simeq\PP^1$ can be taken to be
$0,1$ and $\infty$, and the map $C_4 \arr \PP^2$ corresponds, up to a choice of
basis, to the unique linear system on $\PP^1$ spanned by the divisors $3\{0\}$,
$3\{1\}$ and $3\{\infty\}$. Since these points map to the coordinate lines, the
choice of basis is limited to rescaling the defining sections. The choice of
scaling of the defining sections results in fixing the images of $0$ and $1$,
and the image point of $\infty$ is then uniquely determined.

\begin{figure}[htb]
\input{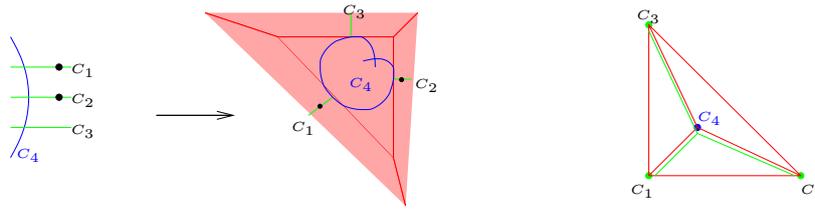}
\caption{The lifted map. The middle figure is only a sketch: the nodal cubic curve $C_4$ meets each of the visible coordinate lines at one point with multiplicity 3. Moreover, these three points are collinear.}
\label{cubicblownup}
\end{figure}

This determines uniquely the point $L_3\cap E$. In particular, the line $L_3$ is
determined. Thus we see that there is a unique transverse prelogarithmic map
$\ulf:C\arr\tilde \uX_0$ such that $\ul{h}\circ \ulf$ lies in the family
$\ul{S}$ of stable maps to $X$.

\emph{Logarithmic enhancement.} Since the curve is rational,
Theorem~\ref{curveconstruction} assures the existence of a logarithmic
enhancement. Only the exceptional component is non-reduced, of multiplicity
$\mu=3$ and for each node $q\in C$ we have $\ind_q=1$ and $\ol w_q=1/1=3/3=1$.
Hence $b=3$, $G=\ZZ/\mu=\ZZ/3$ and the count of Theorem~\ref{curveconstruction}
gives
\[
\frac{|G|}{b}\prod_q\ol w_q= \frac{3}{3}\cdot 1^3=1
\]
basic log enhancement of this transverse prelogarithmic curve. This gives one more
basic stable logarithmic map $h\circ f$. 

\emph{Unobstructedness.} Once again we check that $h\circ f$ is unobstructed,
{if one makes use of an absolute {obstruction} theory}: the logarithmic normal
bundle has degree 0 on each line, hence degree 1 on $C_4$, and is non-special.
Again the map has no automorphisms, which accounts for $1$ curve, with
multiplicity $3$, because $m_{\Gamma}=3$. Hence the final accounting
according to Theorem~\ref{Thm: Main} is
\[
9\,+ \,3\times 1 \ \ =\ \  12,
\]
which is the desired result.

%---------------------------------------------------------------------------------
\subsubsection{Impossibility of other contributions}
Note our presentation has not been thorough in ruling out other possibilities
for stable logarithmic maps, possibly obstructed, contributing to the total. For
example, $\ul{S}$ includes curves where $L_3$ falls into the double point locus
of $X_0$, but a more detailed analysis of the tropical possibilities rules out a
possible log enhancement. We leave it to the reader to confirm that we have
found all possibilities.

%---------------------------------------------------------------------------------
\subsection{Degeneration of point conditions}
\label{pointdegensection}

We now consider a situation which is common in applications of tropical
geometry; this includes tropical counting of curves on toric varieties
\cite{Mikhalkin05,NS06}. We fix a pair $(Y,D)$ where $Y$ is a variety over
a field $\kk$ and $D$ is a reduced Weil divisor such that the divisorial
logarithmic structure on $Y$ is logarithmically smooth over the trivial point
$\Spec\kk$. We then consider the trivial family
\[
X=Y\times\AA^1\arr \AA^1= B,
\]
where now $X$ is given the divisorial logarithmic structure with respect to the
divisor $(D\times B)\cup (Y\times \{0\})$.

%---------------------------------------------------------------------------------
\subsubsection{Evaluation maps and moduli}
Fix a type $\beta$ of stable logarithmic maps to $X$ over $B$, getting a moduli
space ${\mathscr M}(X/B,\beta)$. We assume that the curves of type $\beta$ have
$n$ marked points $p_1,\ldots,p_n$ with $u_{p_i}=0$ --- and possibly some
additional marked points $x_1,\ldots,x_m$ with non-trivial contact orders with
$D$. Given a stable map $(C/S,{\bf x},\bp,f)$, a priori for each
$i\in\{1,\ldots,n\}$ we have an evaluation map $\ev_i:(\ul S,p_i^*\cM_C)\arr X$
obtained by restricting $f$ to the section $p_i$. Noting that $u_{p_i}=0$, the
map $\ev_i^{\flat}:(f\circ p_i)^{-1}\overline{\cM}_X\arr
\overline{\cM}_S\oplus\NN$ factors through $\overline{\cM}_S$, and thus we have
a factorization $\ev_i:(\ul S,p_i^*\cM_C)\arr S \arr X$. In a slight abuse of
notation we write $\ev_i$ for the morphism $S\arr X$ also, and thus obtain a
morphism
\[
\ev:{\mathscr M}(X/B,\beta)\arr X^n:=
X\times_B X\times_B\times\cdots\times_B X.
\]
If we choose sections $\sigma_1,\ldots,\sigma_n:B\arr X$, we obtain a map
\[
\sigma:=\prod_{i=1}^n \sigma_i:B\arr X^n.
\]
This allows us to define \emph{the moduli space of curves passing through the given sections,}
\[
\scrM(X/B,\beta,\sigma):=\scrM(X/B,\beta)\times_{X^n} B,
\]
where the two maps are $\ev$ and $\sigma$.\footnote{Recall that all fibred
products are in the category of fs log schemes.}

%---------------------------------------------------------------------------------
\subsubsection{Virtual fundamental class on $\scrM(X/B,\beta,\sigma)$}
\label{Sec:VFC-points}
We note that the moduli space  $\scrM(X/B,\beta,\sigma)$ of curves passing through the given sections carries a virtual fundamental class.
The perfect obstruction theory is defined by
\begin{equation}
\label{Eq:point-condition}
\fE^{\bullet} = \big(R\pi_*[f^*\Theta_{X/B}\arr \bigoplus_{i=1}^n
(f^*\Theta_{X/B})|_{p_i(S)}]\big)^{\vee},
\end{equation}
for the stable map $(\pi:C\arr S, {\bf x}, \bp, f)$. Here the map of sheaves
above is just restriction. See \cite[\S4]{punctured} for a detailed discussion
of how to impose logarithmic point conditions on a virtual level,
cf.\ also \cite[Prop.~A.1]{Bryan-Leung} for an earlier study in
ordinary Gromov-Witten theory.

%---------------------------------------------------------------------------------
\subsubsection{Choice of sections and $\Delta(X)$} We can now use the techniques of previous sections to produce a virtual
decomposition of the fibre over $b_0=0$ of $\scrM(X/B,\beta,\sigma)
\arr B$. However, to be interesting, we should in general choose
the sections to interact with $D$ in a very degenerate way over $b_0$.
In particular, restricting to $b_0$ (which is now the standard log point),
we obtain maps
\[
\sigma_i:b_0 \arr Y^\dagger,
\]
where $Y^\dagger = Y \times O^\dagger$ is the product with the standard log point. Note that 
\[
\Sigma(Y^\dagger)=\Sigma(X)=\Sigma(Y)\times \RR_{\ge 0},
\]
with $\Sigma(X)\arr\Sigma(B)$ the projection to the second factor.
So $\Delta(X)=\Sigma(Y)$ and $\Sigma(\sigma_i):\Sigma(B)\arr\Sigma(X)$
is a section of $\Sigma(X)\arr\Sigma(B)$ and hence is determined
by a point $P_i\in \Delta(X)$, necessarily rationally defined.

%---------------------------------------------------------------------------------
\subsubsection{Tropical fibred product}
We wish to understand the fibred product
\[
\scrM(X/B,\beta,\sigma):=\scrM(X/B,\beta)\times_{X^n} B
\]
at a tropical level. We observe

\begin{proposition}
\label{tropicalproduct}
Let $X,Y$ and $S$ be fs log schemes, with morphisms $f_1:X\arr S$,
$f_2:Y\arr S$. Let $Z=X\times_S Y$ in the category of fs log schemes,
$p_1,p_2$ the projections. Suppose $\bar z\in Z$ with $\bar x= p_1(\bar z)$,
$\bar y=p_2(\bar z)$, and $\bar s=f_1(p_1(\bar z))=f_2(p_2(\bar z))$. Then
\[
\Hom(\overline{\cM}_{Z,\bar z},\NN)=
\Hom(\overline{\cM}_{X,\bar x},\NN)\times_{\Hom(\overline{\cM}_{S,\bar s},\NN)}
\Hom(\overline{\cM}_{Y,\bar y},\NN)
\]
and
\[
\Hom(\overline{\cM}_{Z,\bar z},\RR_{\ge 0})=
\Hom(\overline{\cM}_{X,\bar x},\RR_{\ge 0})\times_{\Hom(\overline{\cM}_{S,\bar s},\RR_{\ge 0})}
\Hom(\overline{\cM}_{Y,\bar y},\RR_{\ge 0}).
\]
\end{proposition}

\begin{proof}
The first statement follows immediately
from the universal property of fibred product applied to maps $\bar
z^{\dagger}\arr Z$, where $\bar z^{\dagger}$ denotes a geometric point
$\bar z$ with standard logarithmic structure. The second statement then follows
from the first.
\end{proof}

%---------------------------------------------------------------------------------
\subsubsection{Tropical moduli space}
We now provide a simple interpretation for the tropicalization of
$S:=\scrM(X/B,\beta,\sigma)$. If $\bar s\in S$ is a geometric point,
let $Q$ be the basic monoid associated with $\bar s$ as a stable logarithmic map
to $X$. Then by Proposition~\ref{tropicalproduct}, we have
\[
\Hom(\overline{\cM}_{S,\bar s},\RR_{\ge 0})
=\Hom(Q,\RR_{\ge 0})\times_{\prod_i \Hom(P_{p_i},\RR_{\ge 0})} \RR_{\ge 0}.
\]
Here as usual $P_{p_i}=\overline{\cM}_{X,\ul f(p_i)}$, which here equals
$\ol\cM_{Y,\ul f(p_i)}\oplus\NN$. The maps defining the fibred product are as
follows. The map $\Hom(Q,\RR_{\ge 0})\arr \prod_i\Hom(P_{p_i},\RR_{\ge 0})$ can
be interpreted as taking a tropical map $\Gamma\arr
\Sigma(X)=\Sigma(Y)\times\RR_{\ge0}$ to the point of
$\Hom(P_{p_i},\RR_{\ge 0})$ which is the image of the contracted edge
corresponding to the marked point $p_i$. The map $\RR_{\ge
0}\arr\prod_i\Hom(P_{p_i},\RR_{\ge 0})$ is $\prod_i \Sigma(\sigma_i)$ and hence
takes $1$ to $((P_1,1),\ldots,(P_n,1)$.

This yields:

\begin{proposition}
Let $m\in\Delta(S)$, and let $\Gamma_C=
\Sigma(\pi)^{-1}(m)$. Then $\Sigma(f):\Gamma_C\arr \Delta(X)$
is a tropical map with the unbounded edges $E_{p_i}$
being mapped to the points $P_i$. Furthermore, as $m$ varies within
its cell of $\Delta(S)$, we obtain the universal family of tropical maps
of the same combinatorial type mapping to $\Delta(X)$ and with the edges
$E_{p_i}$ being mapped to $P_i$.
\end{proposition}

%---------------------------------------------------------------------------------
\subsubsection{Restatement of the decomposition formula}
Denote by
\[
\scrM(Y^\dagger/b_0,\beta,\sigma) :=  \scrM(X_0, \beta) \times_{X^n} B
\]
and for $\ttau=(\tau,\bA)$ a decorated type of a rigid tropical map (Definition~\ref{Def: type of tropical map}),
\[
\scrM(Y^\dagger,\ttau,\sigma) :=  \scrM(X_0, \ttau) \times_{X^n} B.
\]

Theorem~\ref{Thm: Main} now translates to the following:

\begin{theorem}[The logarithmic decomposition formula for point conditions]
\label{Th:decomposition-points} 
Suppose $Y$ is logarithmically smooth. Then
\[
[\scrM(Y^\dagger/b_0,\beta,\sigma)]^\virt = \sum_{\ttau=(\tau,\bA)}
\frac{m_\tau}{|\Aut(\tau)|}\, {j_\ttau}_*[\scrM(Y^\dagger,\ttau,\sigma)]^\virt.
\]
\end{theorem}

\begin{example}
\label{Expl: NS}
The above discussion allows a reformulation of the approach of \cite{NS06} to
tropical counts of curves in toric varieties. Take $Y$ to be a toric variety
with the toric logarithmic structure, and fix a curve class $\beta$. By fixing
an appropriate number $n$ of points in $Y$, one can assume that the
expected dimension of the moduli space of curves of genus $0$
and class $\beta$ passing through these points is
$0$. Next, after choosing suitable degenerating
sections $\sigma_1,\ldots,\sigma_n$, one obtains points $P_1,\ldots,P_n
\in\Sigma(Y)$, the fan for $Y$. Finally, one explicitly describes
$\scrM(X/B,\beta,\sigma)$ through an analysis for each rigid tropical map to
$\Sigma(Y)$ with the correct topology. In particular, the domain curve
is rational and should have $D_{\rho}\cdot\beta$ unbounded edges
parallel to a ray $\rho\in\Sigma(Y)$, where $D_{\rho}\subset Y$ is the
corresponding divisor. The argument of \cite{NS06} essentially carries out an
explicit analysis of possible logarithmic curves associated with each such rigid
curve after a log blow-up $\tilde Y^\dagger\arr Y^\dagger$.
Theorem~\ref{Th:decomposition-points} also generalizes part of
\cite{NS06} to some higher genus cases, with the determination of the
contribution of individual maps left open.
\end{example}

%---------------------------------------------------------------------------------
\subsection{An example in \texorpdfstring{$\FF_2$}{F2}}
\label{Subsect: FF2}

We now consider a very specific case of \S\ref{pointdegensection} above.
This example deliberately deviates slightly from the toric case of
Example~\ref{Expl: NS} and exhibits new phenomena. 

%---------------------------------------------------------------------------------
\subsubsection{A non-toric logarithmic structure on a Hirzebruch surface} Let $Y$
be the Hirzebruch surface $\FF_2$. Viewed as a toric surface, it has $4$
toric divisors, which we write as $f_0$, $f_{\infty}$, $C_0$ and $C_{\infty}$.
Here $f_0,f_{\infty}$ are the fibres of $\FF_2\arr \PP^1$ over
$0$ and $\infty$, $C_0$ is the unique section with self-intersection $-2$,
and $C_{\infty}$ is a section disjoint from $C_0$, with $C_{\infty}$ 
linearly equivalent to $f_0+f_{\infty}+C_0$. 

We will give $Y$ the (non-toric) divisorial logarithmic structure coming from
the divisor $D=f_0+f_{\infty}+C_{\infty}$, deliberately omitting $C_0$.

%---------------------------------------------------------------------------------
\subsubsection{The curves and their marked points} We will consider rational
curves representing the class $C_{\infty}$ passing through $3$ points
$y_1,y_2,y_3$. Of course there should be precisely one such curve. 

A general curve of class 
$C_{\infty}$ will intersect $D$ in four points, so we will set this up
as a logarithmic  Gromov-Witten problem by considering genus $0$ stable logarithmic maps
\[
f:(C,p_1,p_2,p_3,x_1,x_2,x_3,x_4)\arr Y,
\]
imposing the condition that $f(p_i)=y_i$, and $f$ is constrained to be
transversal to $f_0, f_{\infty}, C_{\infty}$ and $C_{\infty}$ at $x_i$ for
$i=1,\ldots, 4$, respectively. This transversality determines the vectors
$u_{x_i}$, while we take the contact data $u_{p_i}=0$.

Since the maps have the points $x_3$ and $x_4$ ordered, we expect the final count to amount to $2$ rather than $1$.

%---------------------------------------------------------------------------------
\subsubsection{Choice of degeneration} We will
now see what happens when we degenerate the point 
conditions
as in \S\ref{pointdegensection}, by taking $X=Y\times \AA^1$ and considering
sections $\sigma_i:\AA^1\arr X$, $1\le i \le 3$. We choose
these sections to be general subject to the condition that
\[
\sigma_1(0)\in f_0, \quad \sigma_2(0)\in f_{\infty},\quad \sigma_3(0)\in C_0.
\]
Since $C_0\cap C_{\infty}=\emptyset$, any curve in the linear system
$|C_{\infty}|$ which passes through this special choice of $3$ points
must contain $C_0$, and hence be the curve $f_0+f_{\infty}+C_0$.

%---------------------------------------------------------------------------------
\subsubsection{The complex $\Delta(X)$ and the tropical sections}
Note that $\Delta(X)$ is as depicted in Figure~\ref{DeltaF2}, an abstract gluing
of two quadrants, {not linearly embedded in the plane}. The choice of sections
$\sigma_i$ determines points $P_i\in\Sigma(X)$ as explained in
\S\ref{pointdegensection}. For example, if, say, the section $\sigma_1$ is
transversal to $f_0\times \AA^1$, then $P_1$ is the point at distance $1$ from
the origin along the ray corresponding to $f_0$. Since $C_0$ is not part of the
divisor determining the logarithmic structure, $P_3$ is in fact the origin.

\begin{figure}
\input{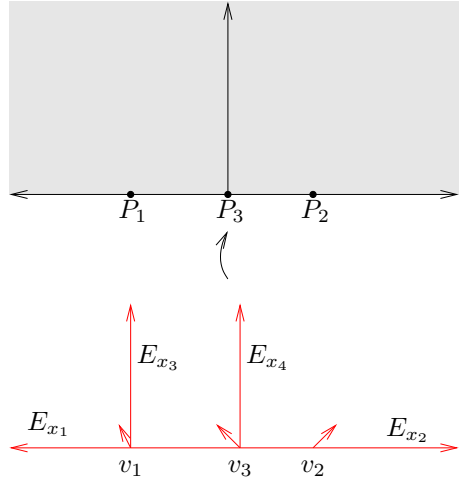}
\caption{The polyhedral complex $\Delta(X)=\Sigma(Y)$ and a potential tropical
map. The small arrows indicate $E_{p_i}$ which are contracted to $P_i$. The
suggested positions of $E_{x_3}$ and $E_{x_4}$ are shown below to contribute $0$
to the virtual count.}
\label{DeltaF2}
\end{figure}

%---------------------------------------------------------------------------------
\subsubsection{The tropical maps}
One then considers rigid decorated tropical maps passing through these points. 
\begin{itemize} 
\item
The curves must have $7$ unbounded edges, $E_{p_i}$,  $E_{x_j}$.
\item
The map contracts $E_{p_i}$ to $P_i$.
\item
Each $E_{x_j}$ is mapped to
an unbounded ray going to infinity in the direction indicating which of
the three irreducible components of $D$ the point $x_j$ is mapped to.
\end{itemize}

%---------------------------------------------------------------------------------
\subsubsection{Rigid tropical maps}
It is then easy to see that to be rigid, the domain of the tropical map must have
three vertices, $v_1,v_2,v_3$, with the edge $E_{p_i}$ attached to $v_i$
and $v_i$ necessarily being mapped to $P_i$. 

The location of the $E_{x_i}$ is less clear. One can show using the balancing
condition \cite[Prop.~1.15]{GS} that $E_{x_1}$ must be attached to $v_1$
and $E_{x_2}$ must be attached to $v_2$. There remains, however, some choice
about the location of $E_{x_3}$ and $E_{x_4}$. Indeed, they may be attached to
the vertices $v_1$, $v_2$ or $v_3$ in any manner. Figure~\ref{DeltaF2} shows one
such choice.

%---------------------------------------------------------------------------------
\subsubsection{Decorated rigid tropical maps}
We must however consider \emph{decorated} rigid tropical maps, and in particular
we need to assign curve classes $\bA(v)$ to each vertex $v$. Let $n_i$ be
the number of edges in $\{E_{x_3},E_{x_4}\}$ attached to the vertex $v_i$. Since
$E_{x_3}$ and $E_{x_4}$ indicate which ``virtual'' components of the domain
curve have marked points mapping to $C_{\infty}$, it then becomes clear that
\[
\bA(v_1)=n_1 f,\quad \bA(v_2)=n_2 f,\quad \bA(v_3)=C_0+n_3f,
\]
where $n_1+n_2+n_3 = 2$.

%---------------------------------------------------------------------------------
\subsubsection{The seeming contradiction} In fact, as we shall see shortly,
there are logarithmic curves whose tropicalization yields any one of the curves
with $n_1=n_2=1$, and there is no logarithmic curve \emph{over the standard log
point} whose tropicalization is the tropical map with $n_3=2$. Surprisingly at
first glance, the only decorated rigid tropical map which provides a non-trivial
contribution to the Gromov-Witten invariant is the one which cannot be realised,
with $n_3=2$. We will also see that the case $n_1=2$ or $n_2=2$ plays no role.
Before we exhibit this counterintuitive behavior, we point out that this is no
contradiction. Indeed, consider a stable log map to $X_0$ with non-rigid
tropicalization. This stable log map will lie in the intersection of the images
in $\scrM(X_0,\beta)$ of more than one subspace $\scrM_\tau(X_0,\beta)$ from
\eqref{Eqn: scrM_tau(X_0)}. Tropical geometry cannot thus tell how these stable log
maps with non-rigid tropicalizations contribute to the virtual count on any
of these components.

%---------------------------------------------------------------------------------
\subsubsection{Curves with $n_1=n_2=1$ contribute 0} To exhibit this
seemingly contradictory behavior, first recall the standard fact that
there is a flat family $\ul{W}\arr \ul{\AA}^1$ such that $\ul{W}_0\simeq \ul
X_0=\FF_2$ and $\ul{W}_t\simeq\ul{\PP}^1\times\ul{\PP}^1$ for $t\not=0$.
Furthermore, the divisor $f_0\cup f_{\infty}\cup C_{\infty}$ extends to a normal
crossings divisor on $\ul{W}$ with three irreducible components: $\{0\}\times
\PP^1, \{\infty\}\times \PP^1,$ and a curve of type $(1,1)$. This endows
$\ul{W}$ with a divisorial logarithmic structure, logarithmically smooth over
$\AA^1$ with the trivial logarithmic structure. However, no curve of class $C_0$
in $W_0$ deforms to $W_t$ for $t\not=0$. Hence no curve representing a point in
the moduli space $\scrM_{\tau}(X_0,\beta)$ for $\tau$ one of the decorated rigid
tropical maps with $n_3=0$ deforms. The usual deformation invariance
of Gromov-Witten invariants then implies that the contribution to the
Gromov-Witten invariant from such a $\tau$ is zero.

Another transparent explanation for the vanishing of this count is
given by the gluing formalism further developed in \cite{punctured}:
The moduli space of punctured stable maps corresponding to the $(-2)$-curve has
negative virtual dimension and by \cite[Thm.~A.16]{associativity} any moduli
space of stable maps with such a component has vanishing virtual count.

%---------------------------------------------------------------------------------
\subsubsection{Expansion and description of moduli space for $n_1=n_2=1$} To explore the
existence of the relevant logarithmic curves, we again turn to
\S\ref{calculationalsection}. First let us construct a curve whose decorated
tropical map has $n_1=n_2=1$. The image of this curve in $\Delta(X)$ yields a
subdivision of $\Delta(X)$ which in turn yields a refinement of $\Sigma(X)$, and
hence a log \'etale morphism $\tilde X\arr X$. It is easy to see that
this is just a weighted blow-up of $f_0\times \{0\}$ and $f_{\infty}\times
\{0\}$ in $X=Y\times\AA^1$; the weights depend on the precise location of $P_1$
and $P_2$, but if they are taken to have distance 1 from the origin, the
subdivision will correspond to an ordinary blow-up. The central fibre is now as
depicted in Figure~\ref{F2curvebad}, with the proper transforms of the sections
meeting the central fibre at the points $p_1,p_2,p_3$ as depicted.

\begin{figure}
\input{F2curvebad.pspdftex}
\caption{}
\label{F2curvebad}
\end{figure}

The logarithmic curve then has three irreducible components, one mapping to
$C_0$ and the other two mapping to the two exceptional divisors, each isomorphic
to $\PP^1\times\PP^1$. These latter two components each map isomorphically to a
curve of class $(1,1)$ on the exceptional divisor, and is constrained to pass
through $p_i$ and the point where $C_0$ meets the exceptional divisor. There is
in fact a pencil of such curves. We remark that all $7$ marked points are
visible in Figure~\ref{F2curvebad}, but the curves in the exceptional
divisors meet the left-most and right-most curves transversally, and not tangent
as it appears in the picture. By Theorem~\ref{curveconstruction}, {any such}
stable map then has a log enhancement, and composing with the map $\tilde X\arr
X$ gives a stable logarithmic map over the standard log point whose
tropicalization is one of the rigid curves with $n_1=n_2=1$. Thus the relevant
moduli space of stable log maps to $\tilde X_0$ has two components, each
isomorphic to $\PP^1 \times \PP^1$, depending on which sides $x_3$ and $x_4$
lie. This moduli space maps injectively to the moduli space of stable log maps
to $X_0$.

To see that the virtual count gives $0$, one can again consider the
absolute deformation and obstruction theory of all the maps parameterized by
$\PP^1 \times \PP^1$. Over the open subset $\CC^* \times \CC^*$ the maps admit a
normal line bundle with degrees $2,-2,2$ on the three components of the curve.
To account for the point conditions we twist down by the points $x_i$, obtaining
a line bundle $N$ of degrees $1,-3,1$ respectively. Restricting to the middle
components gives an isomorphism of the obstruction space $H^1(C, N) \to
H^1(\PP^1, \cO(-3)) = \CC^2$. One checks that the isomorphism extends across the
boundary of $\PP^1 \times \PP^1$, giving a trivial obstruction bundle with zero
Chern class representing the virtual fundamental class 0.

%---------------------------------------------------------------------------------
\subsubsection{Curves with $n_1=n_2=0$}
Now consider the case that $n_1=n_2=0$ and $n_3=2$. This {rigid tropical} curve
cannot be realised as the tropicalization of a stable logarithmic map over the
standard log point. Indeed, to be realised, the curve must have an irreducible
component of class $C_0+2f=C_{\infty}$, and we know there is no such curve
passing through $\sigma_3(0)$, a general point on $C_0$. However, this tropical
map can in fact be realised as a degeneration of a different,
non-rigid tropical map, as depicted in Figure~\ref{F2curvegood1}.

\begin{figure}
\input{F2curvegood1.pspdftex}
\caption{}
\label{F2curvegood1}
\end{figure}

To construct an actual logarithmic curve with $n_1=n_2=0$, we use
refinements again. Assume for simplicity of the discussion that $P_1$ and $P_2$
have been taken to have distance $2$ from the origin. Subdivide $\Delta(X)$ by
introducing vertical rays with endpoints $P_1$ and $P_2$, and in addition
introduce vertical rays which are the images $E_{x_3}$ and $E_{x_4}$; again for
simplicity of the discussion take the {endpoints} of these rays to be at
distance $1$ from the origin. 

This corresponds to a blow-up $\tilde X\arr X$ involving four exceptional
components, and Figure~\ref{F2curvegood2} shows the central fibre of $\tilde
X\arr \AA^1$, along with the image of a stable logarithmic map which
tropicalizes appropriately (once again the curves on the second and
fourth components of $\tilde X$ meet the first and fifth components with order
1, and no tangency). Composing this stable logarithmic map with $\tilde
X\arr X$ then gives a non-basic stable logarithmic map to $X$ over the
standard log point. It is not hard to see that the corresponding basic monoid
$Q$ has rank $3$, parameterizing the image of the curve in $\Sigma(B)$ as well
as the location of the edges $E_{x_3}$ and $E_{x_4}$. The degenerate tropical
curve where the edges $E_{x_3}$ and $E_{x_4}$ are attached to the vertex $v_3$
represents a one-dimensional face of $Q^{\vee}$, so the rigid tropical map
with $n_3=2$ does appear in the family $Q^{\vee}$, but only as a degeneration of
a tropical map which is realisable by an actual stable logarithmic curve over
the standard log point.

One can again show that the relevant moduli space in $\tilde X_0$ has two
components isomorphic to $\PP^1 \times \PP^1$. This time the virtual fundamental
class of each component is the top Chern class of $\cO(1)\boxplus \cO(1)$, which
has degree 1. Each of these maps to $\tilde X_0$ define the same map to
$X_0$, and indeed the corresponding moduli space $\scrM_{n_i=0}(X_0)$ is
discrete and unobstructed.

\begin{figure}
\input{F2curvegood2.pspdftex}
\caption{}
\label{F2curvegood2}
\end{figure}

%---------------------------------------------------------------------------------
\subsubsection{Curves with $n_1+n_2=1$} In this case $\bA(v_3)=C_0+n_3 f=C_0+f$
and either $\bA(v_1)=f,\bA(v_2)=0$ or $\bA(v_1)=0,\bA(v_2)=f$. The expanded
degeneration picture then looks like a hybrid of Figures~\ref{F2curvebad}
and~\ref{F2curvegood1}, with the depicted behavior describing one end each. A
computation similar to the one presented for the case $n_1=n_2=1$ shows
vanishing of this count as well.

%---------------------------------------------------------------------------------
\subsubsection{Curves with $n_i=2$} To complete the analysis, we end by noting
that the case $n_1=2$ or $n_2=2$ cannot occur. Consider the case $n_1=2$. Any
stable logarithmic curve over the standard log point with a tropicalization
which degenerates to such a rigid tropical map must have a decomposition into
unions of irreducible components corresponding to the vertices $v_1, v_2$ and
$v_3$, with the homology class of the image of the stable map restricted to each
of these unions of irreducible components being $2[f_0]$, $0$ and $[C_0]$
respectively. In particular, this will prevent the possibility of having any
irreducible component whose image contains $\sigma_2(0)$. Thus this case does
not occur.

%=================================================================================


\begin{thebibliography}{KKMSD}

\bibitem[AC]{AC} D.~Abramovich and Q.~Chen:
\emph{Stable logarithmic maps to {D}eligne-{F}altings pairs {II}},
Asian J.\ Math.~\textbf{18} (2014), 465--488.

\bibitem[ACGS]{punctured} D.~Abramovich, Q.~Chen, M.~Gross, and B.~Siebert:
\emph{Punctured logarithmic maps},
preliminary version available at \url{https://www.dpmms.cam.ac.uk/~mg475/punctured.pdf}.

\bibitem[ACMW]{ACMW14} D.~{Abramovich}, Q.~{Chen}, S.~{Marcus}, and J.~{Wise}:
\emph{Boundedness of the space of stable logarithmic maps}, August 2014,
J.\ Eur.\ Math.\ Soc.~(JEMS)~\textbf{19} (2017), 2783--2809.

\bibitem[ACP]{ACP}
D.~Abramovich, L.~Caporaso, and S.~Payne:
\emph{The tropicalization of the moduli space of curves},
Ann.\ Sci.\ \'Ec.\ Norm.\ Sup\'er.~(4)~\textbf{48} (2015), 765--809.

\bibitem[AK]{AK}
D.~Abramovich and K.~Karu:
\emph{Weak semistable reduction in characteristic 0},
Invent.\ Math.~\textbf{139} (2000), 241--273.

\bibitem[AW]{AW}
D.~{Abramovich} and J.~{Wise}:
\emph{Birational invariance in logarithmic Gromov-Witten theory},
Compos.\ Math.~\textbf{154} (2018), 595--620.

\bibitem[ACG]{ACG}
E.~Arbarello, M.~Cornalba, P.~Griffiths:
\emph{Geometry of algebraic curves, Volume II},
Springer~2011.

\bibitem[Be]{Behrend-GW}
K.~Behrend:
\emph{Gromov-Witten invariants in algebraic geometry},
Invent.\ Math.~\textbf{127} (1997), 601--617. 

\bibitem[BF]{Behrend-Fantechi}
K.~Behrend and B.~Fantechi:
\emph{The intrinsic normal cone},
Invent.\ Math.~\textbf{128} (1997), 45--88.

\bibitem[BM]{Behrend-Manin}
K.~Behrend and Y.~Manin:
\emph{Stacks of stable maps and {G}romov-{W}itten invariants},
Duke Math.\ J.~\textbf{85} (1996), 1--60.

\bibitem[BrLe]{Bryan-Leung}
J.~Bryan and N.~C.\ Leung:
\emph{The enumerative geometry of {$K3$} surfaces and modular forms},
J.\ Amer.\ Math.\ Soc.~\textbf{13} (2000), 371--410.

\bibitem[BS]{recession cones}
J.~Burgos Gil, M.~Sombra:
\emph{When do the recession cones of a polyhedral complex form a fan?},
Discrete Comput.\ Geom.~\textbf{46} (2011), 789--798.

\bibitem[Ca]{Cadman}
C.~Cadman:
\emph{Using stacks to impose tangency conditions on curves},
Amer.\ J.\ Math.~\textbf{129} (2007), 405--427.

\bibitem[CCUW]{CaChUW}
R.~{Cavalieri}, M.~{Chan}, M.~{Ulirsch}, and J.~{Wise}:
\emph{{A moduli stack of tropical curves}},
\url{https://arxiv.org/abs/1704.03806}.

\bibitem[Ch]{Chen}
Q.~Chen:
\emph{Stable logarithmic maps to {D}eligne-{F}altings pairs {I}},
Ann.\ of Math.~\textbf{180} (2014), 455--521

\bibitem[Co]{CO}
K.~Costello:
\emph{Higher genus {G}romov-{W}itten invariants as genus zero invariants of 
symmetric products},
Ann.\ of Math.~\textbf{164} (2006), 561--601.

\bibitem[GS1]{GS}
M.\ Gross and B.\ Siebert:
\emph{Logarithmic {G}romov-{W}itten invariants},
J.\ Amer.\ Math.\ Soc.~\textbf{26} (2013), 451--510.

\bibitem[GS2]{associativity}
M.\ Gross and B.\ Siebert:
\emph{Intrinsic mirror symmetry},
\url{https://arxiv.org/pdf/1909.07649}.

\bibitem[Kf1]{FKato}
F.\ Kato:
\emph{Log smooth deformation and moduli of log smooth curves},
Internat.\ J.\ Math.~\textbf{11} (2000), 215--232.

\bibitem[Kf2]{FKato2}
F.\ Kato:
\emph{Exactness, integrality, and log modifications},
\url{https://arxiv.org/pdf/math/9907124}.

\bibitem[Kk]{KKato}
K.\ Kato:
\emph{Logarithmic structures of {F}ontaine-{I}llusie},
in: Algebraic analysis, geometry, and number theory ({B}altimore, {MD}, 1988),
191--224, Johns Hopkins Univ.\ Press~1989.

\bibitem[KKMS]{KKMS}G.\ Kempf, F.\ Knudsen, D.\ Mumford and B.\ Saint-Donat:
\emph{Toroidal Embeddings I}, Springer, LNM 339, 1973.

\bibitem[KLR]{KimLhoRuddat}
B.\ Kim, H.\ Lho, H.\ Ruddat,
\emph{The degeneration formula for stable log maps},
\url{https://arxiv.org/abs/1803.04210}.

\bibitem[Kn]{Knudsen}
F.\ Knudsen:
\emph{The projectivity of the moduli space of stable curves. II. The stacks $M_{g,n}$},
Math.\ Scand.~\textbf{52} (1983), 161--199.

\bibitem[Kr]{Kresch}
A.\ Kresch:
\emph{Cycle groups for {A}rtin stacks},
Invent.\ Math.~\textbf{138} (1999), 495--536.

\bibitem[Li]{Jun2}
J.\ Li:
\emph{A degeneration formula of {GW}-invariants},
J.\ Differential Geom.~\textbf{60} (2002), 199--293.

\bibitem[MR]{MandelRuddat}
T.~Mandel, H.~Ruddat:
\emph{Descendant log Gromov-Witten invariants for toric varieties and tropical curves},
\url{https://arxiv.org/pdf/1612.02402}, to appear in Trans.\ Amer.\ Math.\ Soc.

\bibitem[Ma]{Mano}
C.\ Manolache:
\emph{Virtual pull-backs},
J.\ Algebraic Geom.~\textbf{21} (2012), no.~2, 201--245.

\bibitem[Mi]{Mikhalkin05}
G.\ Mikhalkin:
\emph{Enumerative tropical algebraic geometry in {$\mathbb{R}^2$}},
J.\ Amer.\ Math.\ Soc.~\textbf{18} (2005), 313--377.

\bibitem[Mo]{Mochizuki}
S.\ Mochizuki:
\emph{The geometry of the compactification of the {H}urwitz scheme},
Publ.\ Res.\ Inst.\ Math.\ Sci.~\textbf{31} (1995), 355--441.

\bibitem[NS]{NS06}
T.\ Nishinou and B.\ Siebert:
\emph{Toric degenerations of toric varieties
and tropical curves}, Duke Math.\ J.~\textbf{135} (2006), no.~1, 1--51.

\bibitem[Og]{Ogus} A.\ Ogus:
\emph{Lectures on logarithmic algebraic geometry},
Cambridge University Press, 2018.

\bibitem[Ol]{LogStack}
M.\ Olsson:
\emph{Logarithmic geometry and algebraic stacks}, Ann.\ Sci.
\'Ecole Norm.\ Sup. (4)~\textbf{36} (2003), no.~5, 747--791.

\bibitem[Pa1]{Parker: reg}
B.\ Parker:
\emph{Holomorphic curves in exploded manifolds: regularity},
Geom.\ Topol.~\textbf{23} (2019), 1621--1690.

\bibitem[Pa2]{Parker: cmp}
B.\ Parker:
\emph{Holomorphic curves in exploded manifolds: compactness},
\url{https://arxiv.org/pdf/0911.2241}.

\bibitem[Pa3]{Parker: Kuranishi}
B.\ Parker:
\emph{Holomorphic curves in exploded manifolds: Kuranishi structure},
\url{https://arxiv.org/pdf/1301.4748}.

\bibitem[Pa4]{Parker: vfc}
B.\ Parker:
\emph{Holomorphic curves in exploded manifolds: virtual fundamental class},
Geom.\ Topol.~\textbf{23} (2019), 1877--1960.

\bibitem[Pa5]{Parker: gluing}
B.\ Parker:
\emph{Tropical gluing formulae for Gromov-Witten invariants},
\url{https://arxiv.org/pdf/1703.05433}.

\bibitem[Ra]{Ranganathan}
D.\ Ranganathan:
\emph{Logarithmic Gromov-Witten theory with expansions},
\url{https://arxiv.org/pdf/1903.09006}.

\bibitem[{SP}]{stacks-project}
The {Stacks Project Authors}:
\emph{\itshape stacks project},
\url{http://stacks.math.columbia.edu}, 2017.

\bibitem[{Ul1}]{Ulirsch}
M.~{Ulirsch}:
\emph{{Functorial tropicalization of logarithmic schemes: the case of constant
coefficients}},
\url{https://arxiv.org/pdf/1310.6269}.

\bibitem[Ul2]{Ulirsch-thesis}
M.~{Ulirsch}:
\emph{Tropical geometry of logarithmic schemes}, Ph.D.\ thesis,
Brown Univrsity, 2015, pp.~viii+160.

\bibitem[Ul3]{Ulirsch-teichmuller}
M.~{Ulirsch}:
\emph{A non-Archimedean analogue of Teichm{\"u}ller space and its tropicalization}, \url{https://arxiv.org/pdf/2004.07508}.

\bibitem[{Wi}16a]{Wise-finite}
J.~{Wise}:
\emph{{Uniqueness of minimal morphisms of logarithmic schemes}},
 Algebr.\ Geom.~\textbf{6} (2019), 50--63. 
 
\bibitem[{Wi}16b]{Wise-minimality}
J.~{Wise}:
\emph{Moduli of morphisms of logarithmic schemes},
Algebra Number Theory~\textbf{10} (2016), 695--735.

\bibitem[Yu]{Yu}
T.Y.\ Yu:
\emph{Enumeration of holomorphic cylinders in log Calabi-Yau surfaces. II. Positivity, integrality and the gluing formula},
\url{https://arxiv.org/abs/1608.07651}, to appear in Geom.\ Topol.

\end{thebibliography}
\end{document}